\documentclass[10pt]{article}
\usepackage {amssymb}
\usepackage{mathrsfs}
\usepackage {amsmath}
\usepackage {amsthm}
\usepackage{graphicx}
\usepackage {amscd}
\usepackage{subfigure}
\usepackage{wasysym}

\usepackage{xypic}   
\usepackage[colorlinks, citecolor=red]{hyperref}
\makeatletter % `@' now normal "letter"
\@addtoreset{equation}{section}
\makeatother  % `@' is restored as "non-letter"

\newcommand{\mX}{{\mathcal{X}}}

\newcommand{\mD}{{\mathcal{D}}}
\newcommand{\mU}{{\mathcal{U}}}
\newcommand{\mL}{{\mathcal{L}}}
\newcommand{\mW}{{\mathcal{W}}}

\newcommand{\mO}{{\mathcal{O}}}
\newcommand{\mI}{{\mathcal{I}}}
\newcommand{\ov}{\overline}
\newcommand{\fz}{{\mathfrak{z}}}
\newcommand{\cX}{{\mathcal{X}}}
\newcommand{\bC}{{\mathbb{C}}}
\newcommand{\cU}{{\mathcal{U}}}
\newcommand{\Ord}{{\rm Ord}}
\newcommand{\bD}{{\mathbb{B}}}
\newcommand{\cY}{{\mathcal{Y}}}
\newcommand{\cZ}{{\mathcal{Z}}}
\newcommand{\cI}{{\mathcal{I}}}
\newcommand{\cO}{{\mathcal{O}}}
\newcommand{\red}{{\rm red}}
\newcommand{\KS}{{\bf KS}}
\newcommand{\fT}{{\bf T}}

\newcommand{\cA}{{\mathcal{A}}}
\newcommand{\sddbar}{{\sqrt{-1}\partial\bar{\partial}}}
\newcommand{\bB}{{\mathbb{B}}}
\newcommand{\bP}{{\mathbb{P}}}

\newcommand{\cS}{{\mathcal{S}}}
\newcommand{\re}{{\rm Re}}

\newcommand{\cW}{\mathcal{W}}
\newcommand{\vphi}{\varphi}

\newcommand{\cK}{\mathcal{K}}
\newcommand{\bN}{\mathbb{N}}

\newtheorem{thm}{Theorem}[section]
\newtheorem{prop}[thm]{Proposition}%[section]
\newtheorem{defn}[thm]{Definition}%[section]
\newtheorem{cor}[thm]{Corollary}%[section]
\newtheorem{rem}[thm]{Remark}%[section]
\newtheorem{exmp}[thm]{Example}%[section]
\newtheorem{lem}[thm]{Lemma}%[section]
\newtheorem{dig}[thm]{Digression}

\begin{document}
\title{On sharp rates and analytic compactifications of asymptotically conical K\"{a}hler metrics}
\author{Chi Li}

%\date{}
\maketitle

\abstract{
Let $X$ be a complex manifold and $S\hookrightarrow X$ be an embedding of complex submanifold. Assuming that the embedding is $(k-1)$-linearizable or $(k-1)$-comfortably embedded, 
we construct via the deformation to the normal cone a diffeomorphism $F$ from a small neighborhood of the zero section in the normal bundle $N_{S}$ to a small neighborhood of $S$ in $X$ such that $F$ 
is in a precise sense holomorphic up to the $(k-1)$-th order. Using this $F$ we obtain optimal estimates on asymptotical rates for asymptotically conical Calabi-Yau metrics constructed by Tian-Yau. Furthermore, when $S$ is an ample divisor satisfying an appropriate cohomological condition, we relate the order of comfortable embedding to the weight of the deformation of the normal isolated cone singularity arising from the deformation to the normal cone. We also give an example showing that the condition of comfortable embedding depends on the splitting liftings. We then prove an analytic compactification result for the deformation of the complex structure on an affine cone that decays to any positive order 
at infinity. This can be seen as an analytic counterpart of Pinkham's result on deformations of cone singularities with negative weights. %Finally, we also prove an algebraic counterpart, which says that infinitesimal deformations of an affine cone with 
%negative weights come from infinitesimal deformations of the compactified projective cone.  }

\tableofcontents{}

\section{Introduction and main results}

Our original motivation for this paper is to understand the optimal convergence rate of asymptotically conical Calabi-Yau K\"{a}hler metrics on non-compact K\"{a}hler manifolds. However it leads us to the study of embeddings of complex submanifolds and deformations of isolated normal singularities. We start the discussion with the embedding problem.

Let $S$ be a complex submanifold of an ambient complex manifold $X$. 
The comparison between neighborhoods of $S$ inside $X$ with neighborhoods of $S$ inside the normal bundle $N_S$ is a classical subject in complex geometry, which was studied in \cite{Grau, Grif, CaMo, CMS}.  It's clear that although that in general $N_S$ has a different holomorphic structure than that of any neighborhood of $S$ inside $X$, $N_S$ can be viewed as a first order approximation of a small neighborhood of $S$. More precisely, we will denote by $S(k)$ the ringed analytic space $(S, \mathcal{O}_X/\mathcal{I}_S^{k+1})$, which is called the $k$-th infinitesimal 
neighborhood of $S$ inside $X$. Recall the following definition.
\begin{defn}\label{def-linear}
$S$ is k-linearizable inside $X$ if its k-th infinitesimal neighbourhood $S(k)$ in X is isomorphic to its k-th infinitesimal
neighbourhood $S_N(k)$ in $N_S$. Here we identify S with the zero section $S_0$ of $N_S=:N$.
\end{defn}
%As an application of the this diffeomorphism constructed, we can get a lower bound estimate of the optimal rate of some class of AC Calabi-Yau metrics constructed by Tian-Yau \cite{TiYa}. This is the original motivation of this paper and is motivated by the work of Conlon-Hein (\cite{CH1}, \cite{CH2}) who studied the optimal decay rate of asymptotically conical Calabi-Yau K\"{a}hler metrics. To state this result, assume %Let $(X,J)$ be a compact complex manifold with the integrable complex structure $J$, and $S$ be a complex submanifold of $(X,J)$. Let $N_S$ denote the normal bundle of the embedding $S\subset X$ , and $S_0\subset N_S$ denote the zero section.  
Our first preliminary result is that there is a diffeomorphism 
from a neighborhood  %To get condition \eqref{pbO1} and \eqref{pbJ1} 
of $S\subset X$ to a neighborhood of $S_0\subset N_S$ that is in some sense the most holomorphic one.  Although the existence of such a diffeomorphism may be known to experts after the celebrated work of Grauert \cite{Grau} (cf. \cite{vanC1, ADL, CH2, HHN}), here we would like to give an almost explicit construction using the work of Abate-Bracci-Tovena \cite{ABT} together with the deformation to the normal cone construction.  
Let $\tilde{g}_0$ be a smooth Riemannian metric on a neighborhood $W_0$ of $S_0$ inside $N_S$. 
Denote by $\|\cdot\|_{\tilde{g}_0}$ the $C^0$-norms of tensors on $W_0$ with respect to $\tilde{g}_0$ and by $\tilde{r}$ the distance function to $S_0$ with respect to $\tilde{g}_0$.
\begin{prop}\label{prop-linmap}
Assume $S$ is a smooth submanifold of $X$. If $S\hookrightarrow X$ is $(k-1)$-linearizable, then there exist a small neighborhood $W_0$ of $S_0\hookrightarrow N_S$ and a diffeomorphism $F: W_0\rightarrow F(W_0)\subset W$ where $W$ is a small neighborhood of $S\subset X$, such that for any $j\ge 0$, there exists a constant $C_j>0$ and $F$ satisfies
\begin{equation}\label{pbJ2}
\|\nabla_{\tilde{g}_0}^j(F^*J-J_0)\|_{\tilde{g}_0}\le C_j \tilde{r}^{k-j} \text{ on } W_0.
\end{equation}
\end{prop}
Our next result deals with a special situation, that arises in Tian-Yau's construction of asymptotically conical (AC) Calabi-Yau (CY) metric on the complement of some divisor inside a Fano manifold.
To state the result, we need to use the notion of conical metrics on affine cones. In this paper, by an affine cone $C(D, L)$, we will mean the normal affine variety obtained by contracting the zero section of a negative line bundle $L^{-1}$ over a smooth projective manifold $D$. We will also consider the compactified cone $\bar{C}(D,L)=C(D, L)\cup D_{\infty}$ obtained by adding the divisor $D_\infty$ at infinity. These varieties can be expressed using pure algebras ($x$ has degree $1$ in the second graded ring):
\[
C:=C(D, L)={\rm Spec} \bigoplus_{m=0}^\infty H^0(D, mL), \quad \bar{C}:=\bar{C}(D, L)={\rm Proj}\bigoplus_{m=0}^{\infty} \left(\bigoplus_{r=0}^m H^0(D, L^{r})\cdot x^{m-r}\right).
\]
Now let $h$ be a Hermitian metric on the negative line bundle $L^{-1}\rightarrow D$ with negative Chern curvature. Since $C=C(D,L)$ is obtained from $L^{-1}$ by contracting the zero section, $h$ can be considered as a non-negative function on the cone $C$. For any $\delta>0$, there is a complete K\"{a}hler cone metric on $C(D,L)$ whose K\"{a}hler form on the regular part $C\setminus\{o\}$ is given by
\begin{equation}\label{eq-om0del}
\omega_0^{(\delta)}:=\sddbar h^{\delta}.
\end{equation}
It's easy to verify that the associated K\"{a}her metric tensor $g_0^{(\delta)}$ is indeed a Riemannian cone metric (see Section \ref{sec-conemetric}). %, i.e. has the shape $dr^2+r^2 g_T$ where $T$ denotes the circle bundle of the line bundle $L\rightarrow D$.

In the following proposition, we need to use the notion of comfortable embedding, which is a property that appeared in the study of embeddings of complex submanifolds in \cite{Grau}. It refines the notion of linearizability in Definition \ref{def-linear} and was explicitly introduced in \cite{ABT}. We refer to Definition \ref{def-cft} for its definition. 
%Although this is a technical refinement of linearizability defined in Definition \ref{def-linear}, it turns out to be crucial for this. For this reason, we briefly explain it here and refer to Apprendix \ref{linearizable} for the detailed definition. 
%Grauert \cite{Grau} showed that the obstruction for extending an isomorphism $S(k-1)\rightarrow S_N(k-1)$ to an isomorphism $S(k)\rightarrow S_N(k)$ lies in the cohomology group $H^1(S, \Theta_X|_S\otimes \mathcal{I}_S^{k}/\mathcal{I}_S^{k+1})$. Using an exact sequence, Grauert also pointed out that this obstruction consists of two parts: one from $H^1(S, N_S\otimes \mI_S^{k}/\mI_S^{k+1})$ and the other from $H^1(S, \Theta_S\otimes\mI_S^k/\mI_S^{k+1})$. In \cite{ABT}, Abate-Bracci-Tovena explicitly described these two cohomological obstruction classes, and introduced the notion of {\it k-splitting} and {\it k-comfortably embedded} such that
%$k$-linearizable=$k$-splitting+$(k-1)$-comfortably embedded. For more details, see Apprendix \ref{linearizable}.
\begin{prop}\label{prop-cftmap}
%\begin{enumerate}
%\item
Let $X$ be an $n$-dimensional projective manifold and $D$ be smooth divisor such that $N_D$ is ample over $D$.  Let $\omega_0=\omega_0^{(\delta)}$ be a conical metric on $C(D, N_D)$ as defined in \eqref{eq-om0del}. Assume that the embedding $D\hookrightarrow X$ is $(k-1)$-comfortable.
Then there exists a diffeomorphism away from compact sets %($B_R(\underline{o})\subset C(D,N_D)$, $K\subset X\backslash D$) 
 $F_K: C(D,N_D)\backslash B_R(\underline{o}) \rightarrow (X\backslash D)\backslash K$ such that 
\begin{equation}\label{lambda1}
 \|\nabla_{\omega_0}^{j}(F_K^*J-J_0)\|_{\omega_0}\le r^{-\frac{k}{\delta}-j} \text{ for any } j\ge 0,
\end{equation}
where $J$ (resp. $J_0$) denotes the complex structure on $X\setminus D$(resp. $C(D, N_D)\setminus\{o\}$).
\end{prop}
%The special number $\frac{\alpha-1}{n}$ is the exponent in the Calabi-ansatz for K\"{a}hler-Ricci flat cone metric (see \eqref{Calabicone} in Section \ref{sec-conemetric}). 
Note that the norm used in \eqref{pbJ2} is with respect 
to $\tilde{g}_0$ while the norms used in \eqref{lambda1} is with respect to the cone metric $\omega_0$ (or $g_0$) (see section \ref{sec-conemetric} for the comparison between these two K\"{a}hler metrics). This difference corresponds to the difference between the linearizable and comfortable embeddings. 

The next corollary follows from Proposition \ref{prop-cftmap} combined with the regularity theory developed by Conlon-Hein in \cite{CH1} (see \eqref{CHestlambda}). %and Cheeger-Tian's work in \cite{ChTi}.  
In many cases, Proposition \ref{prop-cftmap} improves the regularity in \cite{CH2} (see also \cite[Remark 1.2]{CH3}). 
%We refer to Appendix \ref{AppTY} for more background details on Tian-Yau's metric. 
\begin{cor}\label{cor-estTY}
With the same notations as Proposition \ref{prop-cftmap},  
assume that $X$ be an $n$-dimensional Fano manifold and assume $-K_X=\alpha D$ with $\alpha>1$. Denote $\delta=\frac{\alpha-1}{n}$. Suppose $D$ has a K\"{a}hler-Einstein metric and $D$ is $(k-1)$-comfortably embedded into $X$. Then the metric $\omega_{\rm TY}$ constructed by Tian-Yau (see section \ref{sec-TY}) satisfies:
\[
\|\nabla_{\omega_0}^{j}(F_K^*\omega_{\rm TY}- \omega_0)\|_{\omega_0}\le r^{-\min\{2, \frac{k}{\delta}\}-j} \mbox{ for any } j\ge 0.
\]
If moreover we assume that the K\"{a}hler class is contained in the compactly supported cohomology $H_c^2(X\backslash D)$, then we get:
\[
\|\nabla_{\omega_0}^{j}(F_K^*\omega_{\rm TY}- \omega_0)\|_{\omega_0}\le r^{-\min\{2n, \frac{k}{\delta}\}-j} \mbox{ for any } j\ge 0.
\]
\end{cor}
 The special number $\delta=\frac{\alpha-1}{n}$ in the above corollary is the exponent in the Calabi-ansatz for K\"{a}hler-Ricci flat cone metric (see \eqref{Calabicone} in Section \ref{sec-conemetric}). 
%We will show that if $S=D$ is furthermore $(k-1)$-comfortably embedded, we get exactly the conditions \eqref{pbO1}-\eqref{pbJ1} that we want. 

%Denote $\Omega$ to be the holomorphic volume form on $X\backslash D$ which has pole of order $\alpha$ along $D$. Similarly, $\Omega_0$ denotes the holomorphic volume form on $C(D,N_D)\backslash D$ with pole of order $\alpha$. 
%We will see in Theorem \ref{anacpt} that the estimates in \eqref{lambda1} should be sharp. Under the same assumption as in Proposition \ref{prop-cftmap}, it's easy to see that $D$ is also a Fano manifold. If we assume furthermore that $D$ has K\"{a}hler-Einstein metric, then Tian-Yau constructed an 
%As a consequence, we immediately get a lower bound estimate of the optimal rate of a class of AC Calabi-Yau metrics constructed by Tian-Yau \cite{TiYa}. 
%\begin{cor}\label{sharplambda}
% the Tian-Yau metric $\omega_{\rm TY}$ satisfies:
%\exp^*(\omega_c)-c\omega_0=O(r^{\max\{-2, -\frac{n}{\alpha-1}\}+\delta}) \mbox{ with } g_0-\mbox{derivatives, for any } \delta>0.
%\[
%\|\nabla_{\omega_0}^{j}(F_K^*\omega_{\rm TY}- \omega_0)\|_{\omega_0}\le r^{-\min\{2, \frac{n k}{\alpha-1}\}-j} \mbox{ for } j\ge 0.
%\]
%If moreover we assume that the K\"{a}hler class is contained in the compactly supported cohomology $H_c^2(X\backslash D)$, then we get:
%\[
%\|\nabla_{\omega_0}^{j}(F_K^*\omega_{\rm TY}- \omega_0)\|_{\omega_0}\le r^{-\min\{2n, \frac{n k}{\alpha-1}\}-j} \mbox{ for } j\ge 0
%\]
%\end{cor}
 %defined in \eqref{redKS}.
%Again we assume that $D$ is a smooth ample divisor inside $X$. 
Under appropriate assumptions, our next result relates the order of embedding of $D\rightarrow X$ to the order and the weight of a deformation of $C(D, N_D)$. To construct the deformation that we
like to use, let $X$ be a projective manifold of dimension greater than 2 and $D$ a smooth ample divisor on $X$. Let $\cX$ denote the flat family that is obtained by first blowing up $D\times \{0\}$ inside $X\times \bC$ and then blowing down the strict transform of $X\times\{0\}$. Let $\mD$ be the strict transform of $D\times \bC$. It's easy to see that $\mD\cong D\times\bC$. Assume that the central fibre $\cX_0$ coincides with $\bar{C}(D, N_D)$ so that $\cX^\circ=\cX\setminus \mathcal{D}$ is a flat deformation $\cX^{\circ}\rightarrow \bC$ of $C(D, N_D)$.  We remark that this assumption is always satisfied when $X$ is Fano and $-K_X=\alpha D$ with $\alpha>1$.

Denote by $m(X,D)$ the maximum positive integer $m$ such that the embedding $D\hookrightarrow X$ is $(m-1)$-comfortably embedded. Let $\Ord(\cX^\circ)$ denote the order of deformation (Definition \ref{def-ord2}) and $w(\cX^\circ)$ the weight of the reduced Kodaira-Spencer class $\KS^\red_{\cX^\circ}$ (Definition \ref{def-app-KSred}).

\begin{thm}\label{thm-eqwt}
In the setting of the above paragraph, we have the identities:
\begin{equation}
m(X,D)=\Ord(\cX^\circ)=-w(\cX^\circ).
\end{equation} 
\end{thm}
%\begein{rem}
Notice the integer $m(X, D)$ in the above theorem was considered in \cite[Remark 4.6]{ABT}. If ${\rm dim} D\ge 2$ and $D$ is ample, then, by remark \ref{psdiv}, $m(X,D)$ is also the maximal order of linearizability. In other words, $D\subset X$ is $(m(X,D)-1)$-linearizable but not $m(X,D)$-linearizable. 
%\end{rem}
When $\dim D=1$, we expect the conclusion of Theorem \ref{thm-eqwt} is also true. In fact, a parallel analytic result will be shown in Theorem \ref{thm-anacpt} without the restriction on dimension. On the other hand, we will calculate the example of diagonal embedding $\mathbb{P}^1\hookrightarrow \mathbb{P}^1\times\mathbb{P}^1$ explicitly to see some new phenomenon about the embedding of submanifolds in Proposition \ref{counterexample}. In particular, this example shows that the condition of comfortable embedding depends on the choice of splitting liftings, and thus answers a question by Abate-Bracci-Tovena negatively.

Combining Theorem \ref{thm-eqwt} with Proposition \ref{prop-cftmap}, we can give algebraic interpretations of ad hoc calculations in \cite{CH1} on the asymptotical rates of holomorphic volume forms. See Examples in Section \ref{sec-TY}.

%the proof of Theorem \ref{thm-eqwt} is to use the ``holographic principle": the Hartogs' type extension theorem of cohomology by Andreotti-Grauert.
%\begin{cor}
%Assume $\mX\rightarrow \mathbb{C}$ is an equivariant degeneration whose Kodaira-Spencer element is of weight $w$. Then there exists a diffeomorphism $\Phi: C\backslash K_0\rightarrow X\backslash K_1$ such that $|\Phi^*J-J_0|\le O(r^{-w})$.
%\end{cor}
%The Fano condition may not be necessary. 

Finally, we ask if any deformation of complex structure on $C$ that decays at infinity comes from this construction. We have a good understanding of the algebraic version of this problem thanks to the work of Pinkham. His results in particular implies that any (formal) deformation of $C$ with negative weight can be extended to a (formal) deformation of $\bar{C}$ (see Theorem \ref{Pink2}). For the application to the study of asymptotical conical K\"{a}hler metrics, we prove an analytic compactification result, which can be seen the analytic counterpart of Pinkham's result.
% For example, for asymptotically conical Calabi-Yau with fast decay, this result was claimed by Cheeger-Tian in \cite[Paragraph after Theorem 0.16]{ChTi}. Here we will give a detailed proof. 
Note that a similar compactification result in the asymptotically cylindrical Calabi-Yau case has recently appeared in \cite{HHN}. See Remark \ref{cmpHHN} for some comparison.
%We want to show an analytic compactification result, which can be viewed as the differential geometry counterpart of Lemma \ref{algcpt}.
%Let $L\rightarrow D$ be any line bundle and $\tilde{U}_0$ be a neighborhood of $D\subset L$ and $\tilde{g}_0$ be a smooth Riemannian metric on $\tilde{U}_0$. 
%\begin{thm}
%Assume that $J$ is a complex structure on $\tilde{U}_0\backslash D$ such that
%\[
%\|\nabla_{g_0}^k(J-J_0)\|_{g_0}\le \rho^{\tau}.
%\]
%for some $\tau>0$. Then the complex structure extends to a complex structure $J$ on $\ov{C}\backslash \{\underline{o}\}$.
%\end{thm}

To state this result in a general form, let $h$ be a Hermitian metric on any negative line bundle $L^{-1}\rightarrow D$ with negative Chern curvature and use the notation  $\omega_0:=\omega^{(\delta)}_0$ in \eqref{eq-om0del}.
%Since $C=C(D,L)$ is obtained from $L^{-1}$ by contracting the zero section, $h$ can be thought as a function on the cone $C$. Fixing any $\delta>0$, there is a complete K\"{a}hler cone metrics on $C(D,L)$ given as (see section \ref{sec-conemetric})
%\[
%\omega_{0}=\sddbar h^{\delta}=dr^2+r^2 g_Y.
%\]
%It's easy to verify that $r\sim |\xi|^{-\beta}$.
%For example, when $D$ is a Fano manifold and $L\sim -K_D$. Assume with K\"{a}hler-Einstein metric, then there exists a conical Calabi-Yau metric on $C$.
Let $U_\epsilon$ denote a neighborhood of the infinity end of $C(D,L)$. Equivalently $U_\epsilon$ is a punctured neighborhood of the embedding $D=D_\infty \hookrightarrow \ov{C}(D,L)$. Denote $J_0$ the standard complex structure on $C(D,L)$, and $\ov{U}_\epsilon=U_{\epsilon}\cup D$
the compactification of $U_\epsilon$ in $\ov{C}(D, L)$.
\begin{thm}\label{thm-anacpt}
Assume that $J$ is a complex structure on $U_\epsilon=\ov{U_\epsilon} \backslash D$ such that there exists $\lambda>0$ such that 
\[
\|\nabla_{g_0}^k(J-J_0)\|_{\omega_0}\le r^{-\lambda-k}, \text{ for any } k\ge 0.
\]
Then the complex analytic structure on $U_\epsilon$ extends to a complex analytic structure on $\ov{U}_\epsilon$. Moreover, if we denote by $m=\lceil\delta\lambda\rceil$ the minimal integer which is bigger than or equal to $\delta\lambda$, then in the compactification $(\ov{U}_\epsilon, J)$ the divisor $D$ is $(m-1)$-comfortably embedded.
\end{thm}
This can be seen as a converse to the first part of Proposition \ref{prop-cftmap} and implies that the estimate in Proposition \ref{prop-cftmap} is sharp.
\begin{rem}
%\begin{enumerate}
%\item
Because our proof uses only locally information near the divisor, the argument in the proof should apply in the more general orbifold case.  Actually Conlon-Hein \cite{CH3} has recently used the compactification obtained in Theorem \ref{thm-anacpt} to prove any AC CY metric with quasi-regular metric tangent cone at infinity comes from Tian-Yau's construction.
%\item
%We expect the refinement of Theorem \ref{thm-anacpt} is true.
%\end{enumerate}
\end{rem}

\noindent
%\begin{enumerate}
%\item $\Theta_{Y}$: tangent sheaf of complex manifold $Y$.
%\item
%\end{enumerate}

We end this introduction with the organization of this paper. More detailed summary of materials will be given at the beginning of each section. In Section \ref{KSdef}, we recall 
the standard Kodaira-Spencer theory of infinitesimal deformations and generalize it to a higher order setting. We 
also explain how the (higher order) abstract deformations and embedded deformations are related via Schlessinger's exact sequence.
In Section 3, we relate the order of embedding to the order of deformation of neighborhoods of complex submanifolds. This is achieved by
writing down explicitly a reduced Kodaira-Spencer class and relate it to obstructions to extension of embeddings (in Proposition \ref{embvsdef}).
 In Section 4, we treat the case when the submanifold is an ample divisor and prove Theorem \ref{thm-eqwt}. In Section 5, we apply the result in Section 4 to estimating the asymptotic rates of complex structures on asymptotic conical K\"{a}hler manifolds in order to prove Proposition \ref{prop-cftmap}. In Section 6, we adapt Newlander-Nirenberg's work to prove an analytic compactification result for asymptotically conical complex manifolds. In the appendices, we collect some background results, including Abate-Bracci-Tovena's work on embedding of submanifolds, and theory of infinitesimal deformations of normal affine varieties with isolated singularities.

%Theorem \ref{thm-anacpt} can be seen as the analytic counterpart of the following proposition which is well known by the work of Pinkham \cite{Pink, Pink2}.  
%On the algebraic geometry side, we will show the following theorem which may be of independent interest and can be seen as generalization of one of Pinkham's results (\cite{Pink}) from the very ample to the ample line bundle case.
%\begin{prop}\label{algcpt}
%Any infinitesimal deformation of the cone $C(D,L)$ of negative weight comes from an infinitesimal deformation of the $\ov{C}(D,L)$. 
%\end{prop}
{\bf Acknowledgement:} The author is partially supported by the NSF grant DMS-1405936. I am grateful to H.-J. Hein for his interest, criticism and helpful suggestions on the organization and several proofs in the paper. In particular, he communicated to me the result in Lemma \ref{normalem} and Lemma \ref{lem-ordF}, pointed out some gaps in the proof of results and brought the reference \cite{KS} to my attention. I would like to thank R. Conlon for stimulating discussion that helps me to realize the correct notion of reduced Kodaira-Spencer class is defined by using higher order deformations. I would like to thank the referees for careful readings, helpful comments and constructive suggestions and for bringing the reference \cite{GLS07} to my attention. I would also like to thank Professor J. Wahl for help with several technical points and C. Xu for bringing the reference \cite{Artin} to my attention. Some revisions of this paper was done when I visited MSRI in Spring 2016. I would like to thank the institute for its hospitality and financial support.

\section{Preliminaries on deformation theory}\label{KSdef}
%\section{Kodaira-Spencer class for deformations of complex manifolds}\label{sec-KStr}
%In this section, we first recall the infinitesimal deformation theory developed by Kodaira-Spencer via coordinate changes, and relate it to embedded deformations. 
Our primary object of interest will be a normal affine variety $Z$ with an isolated singularity $o$ and we would like to explain what it means for a deformation of $Z$ to be trivial up to a certain order and to classify the next order of deformations in terms of a Kodaira-Spencer class in ${\bf T}_Z^1$. This is done in section \ref{subsec-hod}, following Artin and Schlessinger relying on manipulations with defining equations. We will show that these concepts are ``identical" to certain analogous concepts in the deformation theory of the complex manifold $Z\backslash K$ where $K$ is a small pseudoconvex neighborhood of $o$. We will define such concepts in section \ref{subsec-ptrivial} following essentially Kodaira-Spencer. The desired identification is proved in Proposition \ref{prop-3eqwt}. For this purpose we will introduce a notation of ``p-trivial embeddings", which connects the two primary concepts to each other. We will be working in the category of analytic varieties.

%We do this by introducing the concepts of $p$-trivial atlases and $p$-trivial embeddings for any $p\ge 0$ in Section \ref{subsec-ptrivial} . The main goal of this section is to prove Proposition \ref{prop-3eqwt} which allows us to characterize the order and the weight of deformation in terms of $p$-trivial atlases or $p$-trivial embeddings for some specific $p$. In other words, Proposition \ref{prop-3eqwt} allows us to show that the reduced Kodaira-Spencer class defined via coordinate changes (following Kodaira-Spencer's approach) is essentially the same as the reduced Kodaira-Spencer class defined via defining functions (following Schlessinger's approach). The materials in Section \ref{subsec-ptrivial} and Section \ref{subsec-hod} may be well-known to experts, but we couldn't locate a reference in the literature and so we write them down in details.

\subsection{Infinitesimal deformations via coordinate changes and embedded deformations}\label{subsec-clKS}

In this subsection, we recall how to get first order Kodaira-Spencer class for an analytic family by using the variation of holomorphic coordinate changes (see \cite{Koda}) and its relation to embedded deformations. Suppose $\cY \rightarrow \bB$ is an analytic family  of complex manifolds over the unit disk $\bB=\{z\in \mathbb{C}; |z|<1\}$. 
\begin{defn}
An atlas covering $\cY_0$ is a collection of coordinate charts $\{\cU_\alpha, \Phi_\alpha=(z_\alpha, t)\}_{\alpha\in \cA}$ such that 
\begin{enumerate}
\item
For each $\alpha\in \cA$,
$\cU_\alpha\subset \cY$ is biholomorphic to polydisk $\bB^{n+1}$,  and $\cY_0\subset \bigcup_{\alpha} \cU_\alpha$ i.e. $\cY_0=\bigcup_\alpha (\cU_\alpha\cap \cY_0)$;
\item there is a biholomorphic map $\Phi_\alpha=(z_\alpha, t): \cU_\alpha\rightarrow \Phi_\alpha(\cU_\alpha)\subset \bC^n\times \bC$ such that $t$ is the coordinate on $\bB$. In particular, $U_\alpha:=\cY_0\cap \cU_\alpha=\{t=0\}$.
\end{enumerate}
\end{defn}
%Suppose we have a collection of coordinate charts $\mathfrak{U}=\{\mU_\alpha, \{z_\alpha^i, t\} \}$ such that $\{\mU_\alpha\}$ is a covering of the neighborhood of $\cY_0$ and on each $\mU_\alpha$, and $\cY_0\cap \mU_\alpha=\{t=0\}$.  
\begin{rem}
\begin{enumerate}
\item
Since we only care about the behavior near the central fibre $\cY_0$, the base $\bB$ is not very important. For example, we will frequently shrink $\bB$ to become $\bB_\epsilon=\{t\in \bC; |t|<\epsilon\}$ for any $0<\epsilon\ll 1$ in the following discussion.
\item
Since we can always shrink $\mU_\alpha$, the assumption that $\mU_\alpha$ is biholomorphic to polydisk $\bB^{n+1}$ is just for the simplicity of the argument.
\end{enumerate}
\end{rem}
We first recall two ways to get the first order Kodaira-Spencer class for an holomorphic family of complex manifolds by using the variation of holomorphic coordinate changes . 
\begin{enumerate}
\item (\v{C}ech cohomology) %We first define the classical Kodaira-Spencer class by using the coordinate change:
Suppose the coordinate changes are given by:
\begin{equation}\label{eq-transit}
z^i_\alpha=F^i_{\alpha\beta}(z_\beta,t), \quad t|_{\mU_\alpha}=t|_{\mU_\beta}.
\end{equation}
Then we can deduce:
\begin{eqnarray*}
F_{\alpha\beta}^i(F_{\beta\gamma}(z_\gamma,t),t)=F^i_{\alpha\gamma}(z_\gamma,t)  &\Longrightarrow & \left.\sum_{j=1}^n\frac{\partial F_{\alpha\beta}^i(z_\beta,t)}{\partial z_\beta^j}\frac{\partial F^j_{\beta\gamma}(z_\gamma,t)}{\partial t}+\frac{\partial F_{\alpha\beta}^{i}(z_\beta,t)}{\partial t}
\right|_{t=0}=\left.\frac{\partial F_{\alpha\gamma}^i(z_\gamma,t)}{\partial t}\right|_{t=0}. 
\end{eqnarray*}
So if we denote:
\begin{equation}\label{eq-theta}
\theta_{\beta\gamma}=\sum_{i=1}^n\left.\frac{\partial F^i_{\beta\gamma}(z_\gamma,t)}{\partial t}\right|_{t=0}\frac{\partial}{\partial z_\beta^{i}}=\sum_{i=1}^n\left. \frac{\partial z_\beta^i(z_\gamma,t)}{\partial t}\right|_{t=0}
\frac{\partial}{\partial z^i_\beta},
\end{equation}
then it satisfies the cocycle condition $\theta_{\beta\gamma}=\theta_{\alpha\gamma}-\theta_{\alpha\beta}$ so that $\{\theta_{\alpha\beta}\}\in \check{H}^1(\{U_\alpha\}, \Theta_{\cY_0})$ where $U_\alpha=\mU_\alpha\cap \cY_0$ and $\Theta_{\cY_0}$ is the tangent sheaf on $\cY_0$. The class defined by $\theta=\{\theta_{\alpha\beta}\}$ in $H^1(\cY_0, \Theta_{\cY_0})$ is the classical Kodaira-Spencer class associated to the analytic family $\cY\rightarrow \bD$. %which we will denote by ${\bf KS}_\cY$. It's well known ${\bf KS}_\cY$ 

\item (Dolbeault cohomology)
It's well known that the above $\theta$ can be represented by using Dolbeault cohomology. For this purpose take $\{\rho_\alpha\}$ to be a partition of unity for the covering $\{\mU_\alpha\}$ and define 
\[
\xi_\alpha=\sum_{i=1}^n\sum_{\gamma}\rho_\gamma \left.\frac{\partial F^i_{\alpha\gamma}(z_\gamma,t)}{\partial t}\right|_{t=0}\frac{\partial}{\partial z_\alpha^i}.
\]
It's easy to verify that $\theta_{\alpha\beta}=\xi_\alpha-\xi_\beta$, so that $\bar{\partial}\xi_\alpha=\bar{\partial}\xi_\beta$ is a globally defined $\Theta_{\cY_0}$-valued closed (0,1)-form and it represents a cohomology class, still denoted by $\theta$,  in $H^{(0,1)}_{\bar{\partial}}(\cY_0, \Theta_{\cY_0})$. 
%In either way, ${\bf KS}_{\cY}$ represents the Kodaira-Spencer class of the deformation given by $\cY$. 
On the other hand, $\theta$ measures the first order variation of the complex structure. We can follow the method in 
Kodaira's book \cite[Section 2.3]{Koda} to define a differentiable vector field $\mathbb{V}$. First notice that by the chain rule
\[
\left(\frac{\partial }{\partial t}\right)_\beta=\sum_{i=1}^n\frac{\partial F^i_{\alpha\beta}(z_\beta,t)}{\partial t}\frac{\partial}{\partial z^i_\alpha}+\left(\frac{\partial}{\partial t}\right)_\alpha,\quad
\frac{\partial }{\partial z_\beta^{j}}=\sum_{i=1}^n\frac{\partial F_{\alpha\beta}^i (z_\beta,t) }{\partial z_\beta^{j}}\frac{\partial}{\partial z_\alpha^i}.
\]
We can define a differentiable vector field locally on $\mathcal{U}_\alpha$ for fixed $\alpha$ by:
\begin{eqnarray*}
\mathbb{V}&=&\sum_{\beta}\rho_\beta\left(\frac{\partial}{\partial t}\right)_\beta=\sum_\beta \rho_\beta \sum_{i=1}^n\frac{\partial F_{\alpha\beta}^i(z_\beta,t)}{\partial t}\frac{\partial}{\partial z_\alpha^i}+\left(\frac{\partial}{\partial t}\right)_\alpha\\
&=&\sum_{i=1}^n\left(\sum_\beta\rho_\beta\frac{\partial F^i_{\alpha\beta}(z_\beta,t)}{\partial t}\right)\frac{\partial}{\partial z_\alpha^i}+\left(\frac{\partial}{\partial t}\right)_\alpha.
\end{eqnarray*}
Then $\mathbb{V}$ is a globally defined vector field in a neighborhood of $\cY_0$. % that locally on $\mathcal{U}_\alpha$ looks like $\mathbb{V}=\left(\frac{\partial}{\partial t}\right)_\alpha$. 
Let $\sigma(t)$ be the flow associated with $\mathbb{V}$ which exists for sufficiently small $t$. We have the identity:
\begin{eqnarray*}
\frac{d}{dt}(\sigma(t)^*J)=(\mathcal{L}_{\mathbb{V}}J) (\partial_{\bar{z}^j}) d\bar{z}^j=\bar{\partial}\mathbb{V}.
\end{eqnarray*} 
Notice that $\bar{\partial}\mathbb{V}|_{t=0}=\bar{\partial}\xi_{\alpha}=\theta\in H^{(0,1)}_{\bar{\partial}}(\cY_0, \Theta_{\cY_0})\cong H^1(\cY_0, \Theta_{\cY_0})$.
\end{enumerate}
Assume that an holomorphic family of complex manifolds $\cY\rightarrow \bB$ is embedded into $\bC^N\times \bB$. Then the Kodaira-Spencer class can also be obtained by using the relation between embedded deformations and abstract deformations. In the following discussion we assume $Y=\cY_0$ is smooth. % but we don't assume $Y$ is affine. 
First there is an exact sequence of sheaves:
\[
0\rightarrow \mI_{Y}/\mI_{Y}^2\rightarrow \left.\Omega^1_{\bC^N}\right|_Y\rightarrow \Omega^1_Y\rightarrow 0,
\]
where $\Omega^1$ denotes the cotangent sheaf. 
The dual of this sequence is given by:
\[
0\rightarrow \Theta_Y\rightarrow \left.\Theta_{\bC^N}\right|_Y\rightarrow N_Y\rightarrow 0,
\]
where $N_Y=N_{Y|\bC^N}$ is the normal sheaf of $Y$ as a complex submanifold of $\bC^N$.
Then there is a long exact sequence:
\begin{equation}\label{em2ab}
0\rightarrow H^0(Y, \Theta_Y)\rightarrow H^0(Y, \Theta_{\bC^N})\rightarrow H^0(Y, N_Y)\stackrel{\delta_Y}{\rightarrow} H^1(Y, \Theta_Y)\rightarrow H^1(Y, \Theta_{\bC^N}|_Y).
\end{equation}
%Denote:
%\[
%{\bf T}^1_Y={\rm Im}\left(H^0(Y, \Theta_{\bC^N})\rightarrow H^0(Y, N_Y)\right).
%\]
%Then we have exact sequences:
%\[
%H^0(Y, \Theta_{\bC^N})\rightarrow H^0(Y, N_Y)\rightarrow {\bf T}^1_Y\rightarrow 0, \quad 0\rightarrow {\bf T}^1_Y\rightarrow H^1(Y, \Theta_Y)\rightarrow H^1(Y, \Theta_{\bC^N}|_Y).
%\]
Choose an atlas covering $\cY_0$, denoted by $\{\cU_\alpha, (z^i_\alpha, t)\}$,  such that the embedding $\cU_\alpha \rightarrow \bC^N\times \bD$ is given by holomorphic functions:
\[
w^b=w_\alpha^b(z_\alpha^i, t), \quad 1\le b\le N.
\]
Note that we will use $\{w^b; b=1, \cdots, N\}$ to denote the coordinates of $\bC^N$ and use $w^b_\alpha$ (i.e. depending on $\alpha$) to denote $w^b$ as functions of the coordinates $\{z^i_\alpha, t\}$. Then there is a locally defined section $v_\alpha\in H^0(U_\alpha, \Theta_{\bC^N}|_{U_\alpha})$ given by 
\[
v_\alpha=\sum_{b=1}^N \left.\frac{\partial w^b_\alpha}{\partial t}\right|_{t=0}\frac{\partial}{\partial w^b}.
\]
Let $[v_\alpha]\in H^0\left(U_\alpha, \left.N_Y\right|_{U_\alpha}\right)$ denote the induced local section under the natural projection $\Theta_{\bC^N}|_{\cY_0}\rightarrow N_{\cY_0}$.

\begin{lem}\label{lem-barv}
$\{[v_\alpha]\}$ can be glued together to become a global section $\frak{v}$ in $H^0(Y, N_Y)$. Moreover, $\delta_Y(\frak{v})=\theta$ where $\delta_Y$ is the connecting morphism in \eqref{em2ab} and $\theta$ is the classical Kodaira-Spencer class defined in \eqref{eq-theta}. 
\end{lem}
\begin{proof}
Notice that we have the relation:
\[
w^b=w_\beta^b(z_\beta, t)=w_\beta^b(z_\beta^i(z_\alpha^j, t), t)=w_\alpha^b(z^j_\alpha, t).
\]
Taking derivatives on both sides with respect to $t$ at $t=0$, we get:
\[
\sum_{b=1}^N \left.\frac{\partial w_\alpha^b}{\partial t}\right|_{t=0} \frac{\partial}{\partial w^b}=\sum_{b=1}^N\sum_{i=1}^n\frac{\partial w^b_\beta}{\partial z^i_\beta} \left.\frac{\partial z^i_\beta}{\partial t}\right|_{t=0}\frac{\partial}{\partial w^b}+\sum_{b=1}^N
\left.\frac{\partial w_\beta^b}{\partial t}\right|_{t=0}\frac{\partial}{\partial w^b}.
\]
Denote by $\iota_Y: Y\rightarrow \bC^N$ the induced embedding. Then the above equality is equivalent to:
\[
v_\alpha-v_\beta=\sum_{b=1}^N \left(\sum_{i=1}^n \left.\frac{\partial z^i_\beta(z_\alpha,t)}{\partial t}\right|_{t=0}\frac{\partial w^b}{\partial z_\beta^i}\right)\frac{\partial}{\partial w^b}=(\iota_Y)_{*}(\theta_{\beta\alpha}),
\]
where we used the identity \eqref{eq-theta}.
Since $\theta_{\beta\alpha}\in \Theta_{\cY_0}(U_\alpha\cap U_\beta)$, we get $[v_\alpha]=[v_\beta]$. By the definition of the connecting morphism $\delta_Y$ in \eqref{em2ab}, we indeed have $\delta_Y(\frak{v})=\theta$.
\end{proof}

\subsection{$p$-trivial atlas and $p$-trivial embeddings}\label{subsec-ptrivial}

We can generalize the above discussion to higher order deformations. 
%Assume the flat deformation $\cY\rightarrow \bD$ can be described by the coordinate change:
%\[
%z_\alpha^i=F^i_{\alpha\beta}(z_\beta,t),\quad t|_{\mU_\alpha}=t|_{\mU_\beta} \mbox{ on } \mU_\alpha\cap \mU_\beta.
%\]
%We always assume $\cY_0\subset \bigcup_\alpha \cU_\alpha$.
%We assume that $\cY(k)\rightarrow \bD(k)$ is the trivial deformation. Then by induction, it's easy to see that we can choose the holomorphic transition function such that 
Let us introduce a condition that will be important in the following discussion.
\begin{defn}\label{def-ptrivial}
Assume that there is an atlas $\cU=\{\cU_\alpha, \Phi_\alpha=(z_\alpha, t)\}$ covering $\cY_0$ with coordinate change functions $z^i_\alpha=F^i_{\alpha\beta}(z_\beta, t)$ on $\cU_\alpha\cap \cU_\beta$. We say that $\cU$ is $p$-trivial if 
$F^i_{\alpha\beta}(z_\beta,t)-F^i_{\alpha\beta}(z_\beta,0)$ vanishes up to order $p$ at $t=0$:
\[
\left.\frac{\partial^l (F_{\alpha\beta}^i(z_\beta,t)-F^i_{\alpha\beta}(z_\beta, 0))}{\partial t^l}\right|_{t=0}=0,  \text{ for } 0\le l\le p.
\]
Notice that, since $l=0$ case is automatically true, this $p$-trivial condition is equivalent to:
% the partial derivatives of transition function $f^i_{\alpha\beta}$ with respect to $t$ vanish up to order $(k-1)$ at $t=0$:
\begin{equation}\label{eq-p-vanish}
\left.\frac{\partial^l F_{\alpha\beta}^i(z_\beta,t)}{\partial t^l}\right|_{t=0}=0,  \text{ for } 1\le l\le p.
\end{equation}
If this is the case, we define the $(p+1)$-order Kodaira-Spencer (\v{C}ech) class, denoted by $\theta_{p+1}(\cU)$ or simply by $\theta_{p+1}$ if the atlas is clear, as the (\v{C}ech) cohomology defined by the cocycle:
\begin{equation}\label{redKSccl}
(\theta_{p+1})_{\alpha\beta}=\frac{1}{(p+1)!}\sum_{i=1}^n \left.\frac{\partial^{p+1} F^i_{\alpha\beta}(z_\beta,t)}{\partial t^{p+1}}\right|_{t=0}\frac{\partial}{\partial z_\alpha^i}\in H^0(\mathcal{U}_\alpha\cap \mathcal{U}_\beta\cap \cY_0, \Theta_{\cY_0}).
\end{equation}
\end{defn}

\begin{lem}\label{lem-thetap}
\begin{enumerate}
\item
$\theta_{p+1}:=\theta_{p+1}(\cU)$ is well-defined, i.e. $\theta_{p+1}=\{(\theta_{p+1})_{\alpha\beta}\}$ satisfies the cocycle condition
$
(\theta_{p+1})_{\beta\gamma}=(\theta_{p+1})_{\alpha\gamma}-(\theta_{p+1})_{\alpha\beta}.
$
\item If we have another $p$-trivial atlas $\tilde{\cU}=\{\tilde{\cU}_\alpha, \tilde{\Phi}_\alpha=(\tilde{z}_\alpha, t)\}$, then $\tilde{\theta}_{p+1}=\theta_{p+1}(\tilde{\cU})$ defines the same 
\v{C}ech cohomology class as $\theta_{p+1}$.
\item
Assume that there exists a $(p-1)$-trivial atlas covering $\cY_0$ and $\theta_p=0\in H^1(\cY_0, \Theta_{\cY_0})$. Then for any relatively compact open subset $\cK\Subset \cY$ such that $\cK_0=\pi^{-1}(0)\cap \cK$ is a relatively compact open set of $\cY_0=\pi^{-1}(0)$, there exists a $p$-trivial atlas covering $\cK_0$.
%There exists a $p$-trivial atlas covering $\cY_0$ if and only if there exists a $(p-1)$-trivial atlas covering $\cY_0$ and $\theta_{p}=0\in H^1(\cY_0, \Theta_{\cY_0})$.
\end{enumerate}
\end{lem}
\begin{proof}
Using the cocycle condition of $\{F_{\alpha\beta}\}$ and the vanishing condition \eqref{eq-p-vanish}, we can take higher order derivatives with respect to $t$ to get:
\begin{eqnarray*}
&&F_{\alpha\beta}^i(F_{\beta\gamma}(z_\gamma,t),t)=F^i_{\alpha\gamma}(z_\gamma,t)  \Longrightarrow  \sum_{j=1}^n\frac{\partial F_{\alpha\beta}^i(z_\beta,t)}{\partial z_\beta^j}\frac{\partial F^j_{\beta\gamma}(z_\gamma,t)}{\partial t}+\frac{\partial F_{\alpha\beta}^{i}(z_\beta,t)}{\partial t}
=\frac{\partial F_{\alpha\gamma}^i(z_\gamma,t)}{\partial t}\nonumber\\
&\Longrightarrow& 
\sum_{j=1}^n\frac{\partial F^i_{\alpha\beta}(z_\beta,t)}{\partial z_\beta^j}\frac{\partial^{p+1} F^j_{\beta\gamma}(z_\gamma,t)}{\partial t^{p+1}}+O(t)+\frac{\partial^{p+1} F^i_{\alpha\beta}(z_\beta,t)}{\partial t^{p+1}}=\frac{\partial^{p+1} F^i_{\alpha\gamma}(z_\gamma,t)}{\partial t^{p+1}}.
\nonumber\\
%&\Longrightarrow& \theta_{\beta\gamma}^{(k)}=\theta_{\alpha\gamma}^{(k)}-\theta_{\alpha\beta}^{(k)}, \quad \theta_{\alpha\beta}^{(k)}=\sum_{i=1}^n \left.\frac{\partial^k f^i_{\alpha\beta}(z_\beta,t)}{\partial t^k}\right|_{t=0}\frac{\partial}{\partial z_\alpha^i}\in H^0(\hat{U}_\alpha\cap \hat{U}_\beta, \Theta_{\cY_0}).
\end{eqnarray*}
%Denoting $\hat{U}_\alpha=\mU_\alpha\cap \tilde{\cY}_0$, 
From this it's clear that $\theta_{p+1}=\{(\theta_{p+1})_{\beta\alpha}\}$ satisfies the cocycle condition. 

To prove the second item,  we first choose a common refinement of $\cU$ and $\tilde{\cU}$ and assume we have the same collection of open sets: $\cU_\alpha=\tilde{\cU}_\alpha$ for $\alpha\in \cA$. Suppose that the coordinate function $\tilde{\cU}_\alpha$ is denoted by $\tilde{\Phi}_\alpha=(\tilde{z}_\alpha, t)$. We then have the following relation on the composition of coordinate functions
\[
z_\alpha=z_\alpha(\tilde{z}_\alpha, t)=z_\alpha(\tilde{z}_\alpha(\tilde{z}_\beta, t), t)=z_\alpha(\tilde{z}_\alpha(\tilde{z}_\beta(z_\beta, t),t),t)=z_\alpha(z_\beta, t).
\]
Taking derivatives on both sides with respect to $t$ we get:
\begin{eqnarray}\label{eq-chain}
\frac{\partial z^i_\alpha(z_\beta, t)}{\partial t}&=&\sum_{j=1}^n \frac{\partial z^i_\alpha(\tilde{z}_\alpha, t)}{\partial \tilde{z}^j_\alpha}\left(\sum_{k=1}^n \frac{\partial \tilde{z}^j_\alpha(\tilde{z}_\beta, t)}{\partial \tilde{z}^k_\beta}\frac{\partial \tilde{z}^k_\beta(z_\beta, t) }{\partial t}+\frac{\partial \tilde{z}^j_\alpha(\tilde{z}_\beta, t)}{\partial t}\right)+\frac{\partial z^i_\alpha(\tilde{z}_\alpha, t)}{\partial t}
\end{eqnarray}
Note that we used the Einstein summation rule. On the other hand, we have
\begin{equation}\label{eq-chain2}
\tilde{z}_\beta=\tilde{z}_\beta(z_\beta(\tilde{z}_\beta, t), t)\Longrightarrow \sum_{j=1}^n\frac{\partial \tilde{z}^k_\beta(z_\beta, t)}{\partial z^j_\beta}\frac{\partial z^j_\beta(\tilde{z}_\beta, t)}{\partial t}+\frac{\partial \tilde{z}^k_\beta(z_\beta, t)}{\partial t}=0.
\end{equation}
%\[
%z_\beta=z_\beta(\tilde{z}_\beta(z_\beta, t), t)\Longrightarrow \frac{\partial z^j_\beta(\tilde{z}_\beta, t)}{\partial \tilde{z}^k_\beta}\frac{\partial \tilde{z}^k_\beta(z_\beta, t)}{\partial t}+\frac{\partial z^j_\beta(\tilde{z}_\beta, t)}{\partial t}=0.
%\]
%The second identity above is equivalent to:
%\begin{eqnarray}\label{eq-chain2}
%\frac{\partial \tilde{z}^k_\beta(z_\beta, t)}{\partial t}=-\frac{\partial \tilde{z}^k_\beta(z_\beta, t)}{\partial z^j_\beta}\frac{\partial z^j_\beta(\tilde{z}_\beta, t)}{\partial t}.
%\end{eqnarray}
Combining \eqref{eq-chain}-\eqref{eq-chain2} and chain rule, we get:
\begin{eqnarray}\label{eq-cbdry}
\sum_{i=1}^n\frac{\partial z^i_\alpha(z_\beta, t)}{\partial t}\frac{\partial}{\partial z^i_\alpha}-\sum_{j=1}^n\frac{\partial \tilde{z}^j_\alpha(\tilde{z}_\beta, t)}{\partial t} \frac{\partial}{\partial \tilde{z}^j_\alpha}&=& \sum_{i=1}^n\frac{\partial z^i_\alpha (\tilde{z}_\alpha, t)}{\partial t}\frac{\partial}{\partial z^i_\alpha}-\sum_{j=1}^n \frac{\partial z^j_\beta(\tilde{z}_\beta, t)}{\partial t}\frac{\partial}{\partial z^j_\beta}.
\end{eqnarray}
At $t=0$, this shows that $\theta_1-\tilde{\theta}_1$ is indeed a coboundary. For $p$-trivial atlases $\cU$ and $\tilde{\cU}$, we can take higher order Lie derivatives $\left(\mathcal{L}_{\partial_t}\right)^{p+1}$ on both sides of \eqref{eq-cbdry} at $t=0$ to get 
\begin{equation}
\left.\frac{\partial^{p+1} z^i_\alpha(z_\beta, t)}{\partial t^{p+1}}\right|_{t=0}\frac{\partial}{\partial z^i_\alpha}-\left.\frac{\partial^{p+1} \tilde{z}^j_\alpha(\tilde{z}_\beta, t)}{\partial t^{p+1}}\right|_{t=0}\frac{\partial}{\partial \tilde{z}^j_\alpha}=
\left.\sum_{i=1}^n \frac{\partial^{p+1} z^i_\alpha(\tilde{z}_\alpha, t)}{\partial t}\right|_{t=0}\frac{\partial }{\partial z^i_\alpha}-\left.\frac{\partial^{p+1} z^j_\beta(\tilde{z}_\beta, t)}{\partial t^{p+1}}\right|_{t=0}\frac{\partial}{\partial z^j_\beta}.
\end{equation}  
So, using the definition in \eqref{redKSccl}, $\theta_{p+1}-\tilde{\theta}_{p+1}$ is indeed a coboundary.

Finally we prove the 3rd item. Assume $\cU=\{\cU_\alpha, \Phi_\alpha=(z_\alpha, t)\}_{\alpha\in \cA}$ is a $(p-1)$-trivial atlas. Then by definition of $\theta_{p}$ and the assumption, we have
\begin{equation}\label{eq-thetacobdry}
\theta_{p}=\frac{1}{p!}\left.\sum_{i=1}^n\frac{\partial^p z^i_\alpha(z_\beta, t)}{\partial t^p}\right|_{t=0}\frac{\partial}{\partial z^i_\alpha}=\sum_{i=1}^n c_\alpha^i\frac{\partial}{\partial z^i_\alpha}-c_\beta^i\frac{\partial}{\partial z^i_\beta}.
\end{equation}
Define the new coordinate $\tilde{z}_\alpha^i=z^i_\alpha+t^p c^i_\alpha$ which are genuine coordinate charts on an open neighborhood of $\cK_0$ inside $\cY$, since $\cK\subset \cY$ and $\cK_0\subset \cY_0$ are relatively compact open subsets.
\[
\tilde{z}^i_\alpha=z^i_\alpha(z_\beta, t)+t^p c^i_\alpha=z^i_\alpha\left(\tilde{z}^j_\beta-t^p c^j_\beta,t\right)+t^p c^i_\alpha=\tilde{z}^i_\alpha(\tilde{z}_\beta, t).
\]
Taking $p$-order derivative with respect to $t$ on both sides, we get:
\[
\frac{1}{p!}\left.\frac{\partial^p \tilde{z}^i_\alpha(\tilde{z}_\beta, t)}{\partial t^p}\right|_{t=0}=-\left.\sum_{j=1}^n \frac{\partial z^i_\alpha}{\partial z^j_\beta} \cdot c^j_\beta\right|_{t=0}+\frac{1}{p!}\left.\frac{\partial^p z^i_\alpha(z_\beta, t)}{\partial t^p}\right|_{t=0}+c^i_\alpha.
\]
Notice that $\frac{\partial}{\partial \tilde{z}^i_\alpha}=\frac{\partial}{\partial z^i_\alpha}$ at $t=0$, so we get by \eqref{eq-thetacobdry} that
\begin{eqnarray*}
\frac{1}{p!}\left.\sum_{i=1}^n\frac{\partial^p \tilde{z}^i_\alpha}{\partial t^p }\right|_{t=0}\frac{\partial}{\partial \tilde{z}^i_\alpha}=-\sum_{j=1}^n c^j_\beta\frac{\partial}{\partial z^j_\beta}+\sum_{i=1}^n c^i_\alpha\frac{\partial}{\partial z^i_\alpha}+\frac{1}{p!}\sum_{i=1}^n
\left.\frac{\partial^p z^i_\alpha(z_\beta, t)}{\partial t^p}\right|_{t=0}\frac{\partial}{\partial z^i_\alpha}=0.
\end{eqnarray*}
So the new atlas $\{\cU_\alpha, \tilde{\Phi}=(\tilde{z}_\alpha, t)\}$ is a $p$-trivial atlas covering $\cK_0$.

%\begin{eqnarray*}
%\left.\frac{\partial^l}{\partial t^l}z_\alpha(w_\beta, t)\right|_{t=0}&=& \left. \frac{\partial^l}{\partial t^l} z_\alpha(w_\alpha, t)\right|_{t=0} \te\tilde{z}t{ for } 0\le l<k;
%\end{eqnarray*}
%\begin{eqnarray*}
%\left.\frac{\partial^k}{\partial t^k}z_\alpha(w_\beta, t)\right|_{t=0}&=& \left. \frac{\partial^k}{\partial t^k} z_\alpha(w_\alpha, t)\right|_{t=0}+\frac{\partial}{\partial w_\alpha}z_\alpha(w_\alpha, t)\left.\frac{\partial^k}{\partial t^k}w_\alpha(w_\beta, t)\right|_{t=0}.
%\end{eqnarray*}
%\begin{eqnarray*}
%\left.\frac{\partial}{\partial t}z_\alpha(z_\beta, t)\right|_{t=0}&=& \frac{\partial }{\partial w_\beta}z_\alpha(w_\beta, t) \frac{\partial}{\partial t}w_\beta(z_\beta, t)+\frac{\partial}{\partial t}z_\alpha(w_\beta, t).
%\end{eqnarray*}
%\begin{eqnarray*}
%\frac{\partial^k z_\alpha(z_\beta, t)}{\partial t^k}=\frac{\partial z^*_\alpha}{\partial w_\alpha}\frac{\partial w_\alpha}{\partial w_\beta}\frac{\partial^k w_\beta}{\partial t^k}+\frac{\partial^k z_\alpha^*}{\partial t^k}
%\end{eqnarray*}

\end{proof}
To make connection with embedded deformations, we introduce the following definition.
\begin{defn}
Let $\cY\rightarrow\bB$ be an holomorphic family of complex manifolds that can be embedded into $\bC^N\times\bB$. 
We say an embedding $\iota_{\cY}: \cY\rightarrow \bC^N\times\bB$ is $p$-trivial (along $\cY_0=:Y$), if there exists an atlas $\cU=\{\cU_\alpha, \Phi_\alpha=(z_\alpha^i, t)\}_{\alpha\in\cA}$ covering $\cY_0$ such that the following
condition is satisfied:
for each $\alpha\in \cA$, if the embedding $\cU_\alpha\rightarrow \bC^N\times\bB$ is represented by the functions $w^b=w_\alpha^b(z_\alpha, t)$ then the following vanishing conditions are satisfied:
\begin{equation}\label{eq-ptemb}
\left.\frac{\partial^l w_\alpha^b(z_\alpha, t)}{\partial t^{l}}\right|_{t=0}=0, \quad 1\le l\le p.
\end{equation}
In this case, we say that $\cU$ is an adapted atlas for the $p$-trivial embedding, or simply a $p$-adapted atlas. 
\end{defn}

To state the next result, we introduce additional notations. Let $\pi: \cY\rightarrow \bB$ be an holomorphic family of complex manifolds over the unit disk. For any $0<\epsilon<1$ and any subset $\cK\subseteq \cY$, denote $\bB_\epsilon=\{t\in \bB; |t|<\epsilon\}$ and
\begin{equation}\label{eq-Kepsilon}
\cY_\epsilon=\cY\times_{\bB}\bB_\epsilon=\pi^{-1}(\bB_\epsilon), \quad \cK_\epsilon=\pi^{-1}(\bB_\epsilon)\cap \cK.
\end{equation}

\begin{lem}\label{lem-2trivial}
%Assume $\cY\rightarrow \bB$ is an holomorphic family of complex manifolds that can be embedded into $\bC^N\times\bB$. 
With the above notations, 
if there exists a $p$-trivial embedding $\cY_\epsilon \hookrightarrow \bC^N\times\bB_\epsilon$ for some $0<\epsilon\ll 1$, then there exists a $p$-trivial atlas covering $\cY_0$. 

Conversely, assume that there exists a $p$-trivial atlas covering $\cY_0$. Then for any relatively compact open subset $\cK\Subset \cY$, there is a $p$-trivial embedding $\cK_\epsilon\hookrightarrow \bC^N\times \bB_\epsilon$ for  $0<\epsilon\ll 1$. More precisely, given an embedding $\cY\rightarrow \bC^N\times\bB$ and a relatively compact open set $\cK\Subset \cY$, there exist $0<\epsilon\ll 1$, a neighborhood $\cW_\epsilon$ of $\cK_\epsilon$ inside $\bC^N\times\bB$ and a biholomorphism $\Phi$ of the form $\Phi(w,t)=(\Psi_t(w), t)$, $\Psi_0={\rm Id}$, from $\cW_\epsilon$ onto its image in $\bC^N\times\bB$ such that $\Phi|_{\cK_\epsilon}$ is a $p$-trivial embedding.

%Then the following conditions are equivalent:
%\begin{enumerate}
%\item[(1)] There exists a sufficiently small $\epsilon\ll 1$, such that there exists a $p$-trivial embedding ;
%\item[(2)] .
%\item There exists a $(p-1)$-trivial atlas covering $\cY_0$ and $\theta_{p}=0\in H^1(\cY_0, \Theta_{\cY_0})$;
%\item There exists an embedding $\cY\rightarrow \bC^N\times \bD_{\epsilon}$ that is $(p-1)$-trivial and the $\mathfrak{v}_p=0\in H^0(\cY_0, N_{\cY_0})$.
%\end{enumerate}
\end{lem}
\begin{proof}
Assume that there is a $p$-trivial embedding with $p$-adapted atlas $\{\cU_\alpha, \Phi_\alpha=(z^i_\alpha, t)\}_{\alpha\in \cA}$. We prove that the $p$-adapted atlas is a $p$-trivial atlas defined in Definition \ref{def-ptrivial}. In other words, we want to show that:
\[
\left.\frac{\partial^l (z_\alpha(z_\beta, t)-z_\alpha(z_\beta, 0))}{\partial t^l}\right|_{t=0}=0, \text{ for } 0\le l\le p.
\] 
We prove this by induction. The case of $l=0$ is automatically true. Assume this is proved for the $(l-1)$-th order derivative for some $1\le l\le p$.  Then we take $l$-th order derivative on both sides of the following relation with respect to $t$ at $t=0$,
\[
w^b=w^b_\alpha(z_\alpha, t)=w^b_\alpha(z_\alpha(z_\beta, t), t)=w^b_\beta(z_\beta, t),
\] 
and use the $(l-1)$-trivial and $l$-adapted property to get:
\[
0=\left.\frac{\partial^l w^b_\beta(z_\beta, t)}{\partial t^l}\right|_{t=0}=\left.\sum_{i=1}^n\frac{\partial w^b_\alpha(z_\alpha, t)}{\partial z^i_\alpha}\frac{\partial^l z^i_\alpha(z_\beta, t)}{\partial t^l}\right|_{t=0}+\left.\frac{\partial^l w^b_\alpha(z_\alpha, t)}{\partial t^l}\right|_{t=0}=\left.\sum_{i=1}^n\frac{\partial w^b_\alpha(z_\alpha, t)}{\partial z^i_\alpha}\frac{\partial^l z^i_\alpha(z_\beta, t)}{\partial t^l}\right|_{t=0}.
\]
Because the $N\times n$ matrix 
\[
M_{b i}=\frac{\partial w^b_\alpha(z_\alpha, t)}{\partial z^i_\alpha}
\]
has rank $n$ and zero kernel, we get $\left.\frac{\partial^l z^i_\alpha(z_\beta, t)}{\partial t^l}\right|_{t=0}=0$. So the atlas is $l$-trivial. This completes the induction argument and shows that $p$-adapted atlas is indeed $p$-trivial.

Conversely, we choose a $p$-trivial atlas $\cU=\{\cU_\alpha, \Phi_\alpha=(z_\alpha,t)\}_{\alpha\in\cA}$ covering $\cY_0$ and an embedding which, for each $\alpha\in \cA$, is represented by $w^b=w^b_\alpha(z_\alpha, t)$. Then we have the relation:
\[
w^b=w^b_\beta(z_\beta, t)=w^b_\beta(z_\beta(z_\alpha, t), t)=w^b_\alpha(z_\alpha, t).
\]
Taking the derivative on both sides at $t=0$ and using the $p$-trivial condition of the atlas, we get:
\begin{equation*}
\left.\frac{\partial^{l} w^b_\alpha}{\partial t^{l}}\right|_{t=0}=\left.\sum_{i=1}^n \frac{\partial w^b_\beta }{\partial z^i_\beta}\frac{\partial^{l} z^i_\beta(z_\alpha, t)}{\partial t^{l}}\right|_{t=0}+\left.\frac{\partial^{l}w^b_\beta}{\partial t^{l}}\right|_{t=0}=\left.\frac{\partial^l w^b_\beta}{\partial t^l}\right|_{t=0}, \quad 1\le l\le p.
\end{equation*}
So we see that for each $1\le l\le p$, there is a globally defined vector field:
\[
v^{(l)}=\left.\sum_{b=1}^N \frac{\partial^l w^b_\beta}{\partial t^l}\right|_{t=0}\frac{\partial}{\partial w^b}\in H^0(Y, \Theta_{\bC^N}|_Y)
\]
We claim that the given embedding can be modified to become a $p$-trivial embedding on any relatively compact open subset. We do this by induction as follows. Assume that we already get an $(l-1)$-trivial embedding for some $1\le l\le p$. 
Let $\sigma^{(l)}(w, s)$ be the flow generated by an extension of holomorphic vector field $-v^{(l)}/l!$ to $\bC^N$. Note that $\sigma^{(l)}(w, s)$ exists on a relatively compact open subset for $|s|$ sufficiently small. 

Set $\Phi(w, t)=(\sigma^{(l)}(w, t^l), t)=:(\Psi_t(w), t)$. Then $\Phi$ is a biholomorphism defined on a relatively compact open neighborhood $\cW_\epsilon$ of $\cK_\epsilon$ when $\epsilon$ is sufficiently small.
Define a new embedding $\tilde{\iota}_{\cW_\epsilon}:=\Phi\circ \iota_\cY|_{\cW_\epsilon}$. Then there is a new representation 
$\tilde{w}^b=\tilde{w}^b(w_\alpha(z_\alpha, t), t)=\tilde{w}^b_\alpha(z_\alpha, t)$. We can then take derivative with respect to $t$ by using the $(l-1)$-trivial condition to see that  $\tilde{\iota}_\cY$ is indeed an $l$-trivial embedding:
\begin{eqnarray*}
\left.\sum_{b=1}^N\frac{\partial^l \tilde{w}^b(z_\alpha, t)}{\partial t^l}\right|_{t=0}\frac{\partial }{\partial w^b}&=&\left.
\sum_{b=1}^N \sum_{c=1}^N\frac{\partial \tilde{w}^b}{\partial w^c}\frac{\partial^l w^c(z_\alpha, t)}{\partial t^l}\right|_{t=0}\frac{\partial}{\partial w^b}+\left.\sum_{b=1}^N\frac{\partial^l \tilde{w}^b(w, t)}{\partial t^l}\right|_{t=0}\frac{\partial}{\partial w^b}\\
&=&\left.\sum_{c=1}^N\frac{\partial^l w^c_\alpha(z_\alpha, t)}{\partial t^l}\right|_{t=0}\frac{\partial}{\partial w^c}-v^{(l)}=0.
\end{eqnarray*}

\end{proof}

The first statement of following lemma generalizes Lemma \ref{lem-barv}.

\begin{lem}\label{lem-pkv1}
\begin{enumerate}
\item
If there is a $p$-trivial embedding $\iota_\cY: \cY\rightarrow \bC^N\times \bD$ with $p$-adapted atlas $\{U_\alpha, \Phi_\alpha=(z_\alpha, t)\}$, we can define a global section $\mathfrak{v}_{p+1}:=\mathfrak{v}_{p+1}(\iota_\cY, \Phi_\alpha)\in H^0(\cY_0, N_{\cY_0})$ such that 
\begin{equation}\label{eq-defvp}
\mathfrak{v}_{p+1}(U_\alpha)=\frac{1}{(p+1)!}\left[\left.\sum_{b=1}^N\frac{\partial^{p+1} w^b_\alpha(z_\alpha, t)}{\partial t^{p+1}}\right|_{t=0}\frac{\partial}{\partial w^b}\right] \in H^0(\cU_\alpha\cap \cY_0, N_{\cY_0})
\end{equation}
where we used the natural morphism $\Theta_{\bC^N}|_{\cY_0}\rightarrow N_{\cY_0}$ (remember that $(w^b)_{b=1}^N$ denotes the standard coordinates on $\bC^N$). Furthermore, $\delta_Y(\mathfrak{v}_{p+1})=\theta_{p+1}$ where $\delta_Y$ is the connecting morphism $\delta_Y: H^0(\cY_0, N_{\cY_0})\rightarrow H^1(\cY_0, \Theta_{\cY_0})$ introduced in \eqref{em2ab} and $\theta_{p+1}$ is the reduced Kodaira-Spencer cocycle associated to the $p$-adapted atlas. 
\item
Assume that there is another $p$-adapted atlas $\{\tilde{\cU}_\alpha, \tilde{\Phi}_\alpha=(\tilde{z}_\alpha^i, t)\}$ for the same embedding $\iota_\cY$. If we denote $\tilde{\mathfrak{v}}_{p+1}=\mathfrak{v}_{p+1}(\iota_\cY, \tilde{\Phi}_\alpha)$, then 
$\delta_Y(\mathfrak{v}_{p+1}-\tilde{\mathfrak{v}}_{p+1})=0$.
\end{enumerate}
\end{lem}
\begin{proof}
By the proof of Lemma \ref{lem-2trivial}, a $p$-adapted atlas is $p$-trivial. So we can use the $p$-trivial condition to take the $(p+1)$-th order derivative 
with respect to $t$ at $t=0$ on both sides of the identity:
\[
w^b=w^b_{\beta}(z_\beta, t)=w^b_\beta(z^i_\beta(z^j_\alpha, t), t)=w^b_\alpha(z_\alpha, t)
\]
to get
\begin{equation}\label{eq-v2ks}
\frac{\partial^{p+1} w^b_\alpha(z_\alpha, t)}{\partial t^{p+1}}=\sum_{i=1}^n \frac{\partial w^b_\beta}{\partial z^i_\beta}\frac{\partial^{p+1} z^i_\beta}{\partial t^{p+1}}+\frac{\partial^{p+1} w^b(z_\beta, t)}{\partial t^{p+1}}.
\end{equation}
If we define 
\[
v_\alpha=\frac{1}{(p+1)!}\left.\sum_{b=1}^N\frac{\partial^{p+1} w^b(z_\alpha, t)}{\partial t^{p+1}}\frac{\partial}{\partial w^b}\right|_{t=0}
\]
then $v_\alpha-v_\beta=\iota_{Y*}(\theta_{p+1})_{\beta\alpha}$. So $\{[v_\alpha]\}_{\alpha\in\cA}$ can be glued to become a global section $\mathfrak{v}_p\in H^0(Y, N_Y)$ using the fact that $N_Y=\Theta_{\bC^N}/\Theta_{Y}$.

For the second item. We use \eqref{eq-v2ks} to get the following identities:
\begin{eqnarray*}
\delta(\mathfrak{v}_{p+1}-\tilde{\mathfrak{v}}_{p+1})(U_\alpha\cap U_\beta)&=&\iota_{Y*}((\theta_{p+1})_{\beta\alpha})-\iota_{Y*}((\tilde{\theta}_{p+1})_{\beta\alpha})\\
&=&\iota_{Y*}\left((\theta_{p+1})_{\beta\alpha}-(\tilde{\theta}_{p+1})_{\beta\alpha}\right)
\end{eqnarray*}
%\[
%\frac{\partial^{p+1}w^b(z_\alpha, t)}{\partial t^{p+1}}\frac{\partial }{\partial w^b}-\frac{\partial^{p+1} w^b(\tilde{z}_\alpha, t)}{\partial t^{p+1}}\frac{\partial}{\partial w^b}=
%\frac{\partial^{p+1}w^b(z_\beta, t)}{\partial t^{p+1}}\frac{\partial }{\partial w^b}-\frac{\partial^{p+1} w^b(\tilde{z}_\beta, t)}{\partial t^{p+1}}\frac{\partial}{\partial w^b}
%\]
By Lemma \ref{lem-thetap} item 2, more specifically identity \eqref{eq-cbdry}, we know that $\theta_{p+1}-\tilde{\theta}_{p+1}=0\in H^1(Y, \Theta_Y)$. So the proof is complete. 

\end{proof}

\begin{lem}\label{lem-chpte}

Assume that there exists a $(p-1)$-trivial embedding $\iota_\cY: \cY\rightarrow \bC^N\times\bB$ with $\mathfrak{v}_p(\iota_\cY)=0\in H^0(\cY_0, N_{\cY_0})$ (see Lemma \ref{lem-pkv1} for the definition of $\mathfrak{v}_p$). Then for any relatively compact open subset $\cK\Subset \cY$, there is a $p$-trivial embedding $\cK_\epsilon\hookrightarrow \bC^N\times \bB_\epsilon$ for  $0<\epsilon\ll 1$.

%\begin{enumerate}
%\item[(1)]
%There is a $p$-trivial embedding $\cY_\epsilon\rightarrow \bC^N\times\bB_\epsilon$ for $\epsilon\ll 1$.
%\item[(2)]
%There exists an embedding $\cY_\epsilon\rightarrow \bC^N\times \bB_\epsilon$ and an atlas $\cU=\{\cU_\alpha, \Phi_\alpha=(z_\alpha, t)\}_{\alpha\in \cA}$ covering $\cY_0$ such that the following condition is satisfied: for each $\alpha\in\cA$, 
%if the embedding $\cU_\alpha\rightarrow \bC^N\times\bB$ is represented by the function $w^b=w^b_{\alpha}(z_\alpha, t)$, then we have:
%\begin{equation}\label{eq-ptemb2}
%\sum_{b=1}^N \left.\frac{\partial^l w^b_\alpha(z_\alpha, t)}{\partial t^l}\right|_{t=0}\frac{\partial}{\partial w^b}\in \Theta_{\cY_0}(\cU_\alpha\cap \cY_0), \quad 
%\text{ for } 1\le l\le p.
%\end{equation}
%\item[(3)]
%There exists a 
%\end{enumerate}
\end{lem}
\begin{proof}%(cf. Proof of Lemma \ref{lem-pkv1}) 
We need to prove that there exists an atlas satisfying the condition \eqref{eq-ptemb}. By assumption there is an atlas 
$\cU=\{\cU_\alpha, \Phi_\alpha=(z_\alpha, t)\}_{\alpha\in \cA}$ covering $\cY_0$ such that the following condition is satisfied: for each $\alpha\in\cA$, 
if the embedding $\iota_\cY|_{\cU_\alpha}: \cU_\alpha\rightarrow \bC^N\times\bB$ is represented by the function $w^b=w^b_{\alpha}(z_\alpha, t)$, then we have:
$\left.\frac{\partial^l w^b_\alpha(z_\alpha, t)}{\partial t^l}\right|_{t=0}=0$ ($b=1,\dots, N$ and $1\le l\le p-1$), and (see \eqref{eq-defvp})
\begin{equation}\label{eq-ptemb2}
\frac{1}{p!}\sum_{b=1}^N \left.\frac{\partial^p w^b_\alpha(z_\alpha, t)}{\partial t^p}\right|_{t=0}\frac{\partial}{\partial w^b}\in \Theta_{\cY_0}(\cU_\alpha\cap \cY_0).
\end{equation}
So we get functions $d^i_\alpha(z_\alpha)$ satisfying
\begin{equation}\label{eq-ptemb4}
\sum_{b=1}^N \left.\frac{\partial^p w^b_\alpha(z_\alpha, t)}{\partial t^p}\right|_{t=0}\frac{\partial}{\partial w^b}=\sum_{i=1}^n d^i_\alpha \frac{\partial}{\partial z_\alpha^i}=\sum_{i=1}^n d^i_\alpha \frac{\partial w^b_\alpha(z_\alpha, 0)}{\partial z^i_\alpha}\frac{\partial}{\partial w^b}\in \Theta_{\cY_0}(\cU_\alpha\cap \cY_0).
\end{equation}
Define new functions $\tilde{z}^i_\alpha=z^i_\alpha+d^i_\alpha \frac{t^p}{p!}$ which are coordinates on $\cK_\epsilon$ for $0<\epsilon\ll 1$. Taking derivative all the way up to order $p$ on both sides of 
\[
w^b=w^b_\alpha(z_\alpha, t)=w^b_\alpha(z_\alpha(\tilde{z}_\alpha, t), t)=\tilde{w}^b_\alpha(\tilde{z}_\alpha, t)
\]
at $t=0$, we get:
\begin{eqnarray*}
\left.\frac{\partial^p w^b(\tilde{z}_\alpha, t)}{\partial t^p}\right|_{t=0}&=&\left.\sum_{i=1}^n \frac{\partial w^b_\alpha(z_\alpha, 0)}{\partial z_\alpha^i}\frac{\partial^p z_\alpha^i(\tilde{z}_\alpha, t)}{\partial t^p}+\frac{\partial^p w^b_\alpha(z_\alpha, t)}{\partial t^p}\right|_{t=0}\\
&=&-\sum_{i=1}^n d^i_\alpha \frac{\partial w^b_\alpha(z_\alpha, 0)}{\partial z^i_\alpha}+\left.\frac{\partial^p w^b_\alpha(z_\alpha, t)}{\partial t^p}\right|_{t=0}=0.
\end{eqnarray*}
So we see that the atlas $\{\cU_\alpha, (\tilde{z}_\alpha, t)\}$ is indeed a $p$-adapted atlas.

\end{proof}

\begin{lem}\label{lem-3trivial}
Let $\pi: \cY\rightarrow \bB$ be a holomorphic family of complex manifolds embedded into $\bC^N\times\bC$. Let $\cK\subset \cY$ be a relatively compact open set such that there exist a bounded open set $\cW\subset \bC^N\times\bC$ and $H_1, \dots, H_d\in \cO(\cW)$ satisfying:
\begin{equation}
\cK=\{(w,t)\in \cW: H_1(w, t)=\cdots=H_d(w,t)=0\}.
\end{equation}
Then for all $p\ge 1$, the following are equivalent (see \eqref{eq-Kepsilon} for notations):
\begin{enumerate}
\item[(1)]
There exists $0<\epsilon\ll 1$ such that there exists a %For any relatively compact open set $\cK\subset \cY$, there is a 
$p$-trivial atlas on $\cK_\epsilon$. % covering $\cK_0=(\pi|_{\cK})^{-1}(\{0\})$ for $\epsilon\ll 1$.
\item[(2)]
There exist $0<\epsilon\ll 1$ and a biholomorphism $\Phi$ of the form $\Phi(w,t)=(\Psi_t(w), t)$, $\Psi_0={\rm Id}$, from $\cW_\epsilon$ onto its image in $\bC^N\times\bC$ such that $\Phi|_{\cK_\epsilon}$ is a $p$-trivial embedding.
%For any relatively compact open set $\cK\subset \cY$, there is a $p$-trivial embedding $\cK_\epsilon\hookrightarrow \bC^N\times\bD_\epsilon$ for $\epsilon\ll 1$.

\item[(3)]
There exist $0<\epsilon\ll 1$ and a biholomorphism $\Phi$ of the form $\Phi(w,t)=(\Psi_t(w,t),t)$, $\Psi_0={\rm Id}$, from $\cW_\epsilon$ onto its image in $\bC^N\times\bC$ such that 
\[
\Phi(\cK_\epsilon)=\{(w,t)\in \Phi(\cW_\epsilon): F_1(w,t)=\cdots=F_d(w,t)=0\}
\]
for some holomorphic functions $F_1,\dots,F_d$ with $F_m(w,t)=F_m(w,0)+t^{p+1}G_m(w,t)$.
%For any $p\in \cY_0$, there exists an open set $\cW\subset \bC^N\times\bC$ with $p\in \cW$ and $\cI_{\cY}(\cW)$ is generated by $F_1(w,t), \dots, F_d(w,t)$ such that $F_m(w,t)=F_m(w,0)+t^{p+1}G_m(w,t)$ with $G_m(w,t)$ analytic in $\cW$.

%there exist functions $g_1,\dots, g_d$ on $\cW$ and functions $f_1,\dots, f_d$ on 
%$\cW\cap (\bC^N\times\{0\})$, such that $\cW\cap \cY_t$ is cut out by the functions $f_1+t^{p+1}g_1, \dots, f_d+t^{p+1}g_d$. 
\end{enumerate}
\end{lem}

\begin{proof}
The equivalence of (1) and (2) has been proved in Lemma \ref{lem-2trivial}. We now prove that (2) implies (3). 
So assume that $\Phi|_{\cK_\epsilon}$ is a $p$-trivial embedding with  
%For any point $p\in \cY_0$, choose a relatively compact open subset $\cK\subset \cY$ with $p\in \cK$ and a $p$-trivial embedding $\iota: \cK_\epsilon\rightarrow \bC^N\times\bB_\epsilon$ with a 
a $p$-adapted atlas $\{\cU_\alpha, (z_\alpha,t)\}$.  % Choose a small open neighborhood $\cW$ of $p$ in $\bC^N\times \bC$ such that $\cW\cap \cY\subseteq \cK$ and 
Set $F_m=H_m\circ \Phi^{-1}$. Then the ideal sheaf of $\Phi(\cK_\epsilon)$ is generated by $\{F_1(w,t), \dots, F_d(w,t)\}$. %In particular, the ideal sheaf of $\cY_t\cap \cW\cap \bC^N$, as a subvariety of $\bC^N$, is generated by $\{F_1(w,t),\dots, F_d(w,t)\}$. 
We will prove by induction that there exists a sequence of open sets 
$\cW=\cW^{(0)}\supseteq \cW^{(1)}\supseteq \dots \supseteq \cW^{(p+1)}$ and holomorphic functions $F^{(l)}_m(w,t)$ on $\cW^{(l)}$ such that 
\begin{enumerate}
\item
$\cY\cap \cW^{(l)}$ is generated by $F^{(l)}_m(w,t)$;
\item There exist holomorphic functions $G_{m,l}(w,t)$ on $\cW^{(l)}$ such that
\begin{equation}
F^{(l)}_m(w,t)=F^{(l)}_m(w,0)+t^l G_{m,l}(w,t)
\end{equation}
\end{enumerate} 
For $l=1$, let $\cW^{(1)}=\cW$, $F^{(1)}_m(w,t)=F_m(w,t)$ and $G_{m,1}(w,t)=\frac{1}{t}(F_m(w,t)-F_m(w,0))$. Then assume the statement is true for $1\le l\le p$. 
We have the identity:
\begin{equation}\label{eq-Flmexp}
F^{(l)}_m(w^1_\alpha(z_\alpha,t),\dots, w^N_\alpha(z_\alpha,t), t)\equiv 0.
\end{equation}
Taking derivative with respect to $t$ $l$ times and using the identity \eqref{eq-ptemb} and \eqref{eq-Flmexp} we get:
\begin{equation}
\left.G_{m,l}(w_\alpha(z_\alpha,t),t)\right|_{t=0}=0.
\end{equation}
Because the ideal sheaf of $\cK_0\cap \cW^{(l)}$ is generated by $\{F^{(l)}_1(w,0),\dots, F^{(l)}_d(w,0)\}$, there exists $h_{m,l,r}(w)$ such that
\begin{equation}
G_{m,l}(w,0)=\sum_{m=1}^d F^{(l)}_r(w,0)h_{m,l,r}(w).
\end{equation}
Now define:
\begin{eqnarray*}
F^{(l+1)}_m(w,t)&=&F^{(l)}_m(w,t)-t^{l}\sum_{r=1}^d F^{(l)}_r(w,t)h_{m,r,l}(w)\\
&=&F^{(l)}_m(w,0)+t^l G_{m,l}(w,t)-t^{l}\sum_{r=1}^d F^{(l)}_r(w,t)h_{m,r,l}(w).
\end{eqnarray*}
We then have:
\begin{eqnarray*}
\left.\frac{\partial^l F^{(l+1)}(w,t)}{\partial t^l}\right|_{t=0}=0.
\end{eqnarray*}
So we know that $F^{(l+1)}_{m,l}(w,t)$ has the following expansion:
\begin{equation}
F^{(l+1)}_m(w,t)=F^{(l+1)}_m(w,0)+t^{l+1} G_{m,l+1}(w,t)
\end{equation}
over an open subset $\cW^{(l+1)}$ of $\cW^{(l)}$. Note that $\{F^{(l+1)}_r\}$ generate the same ideal as the $\{F^{(l)}_r\}$. Indeed $\{F^{(l+1)}_r\}$ is obtained  by multiplying a holomorphic matrix of the form ${\rm Id}_{d\times d}+O(t^l)$ to $\{F^{(l)}_r\}$. Because $l\ge 1$, this matrix has a holomorphic inverse for $|t|\ll 1$. So  it is easy to see that $F^{(l+1)}_m$ satisfies the wanted properties.

Conversely, we assume (3) holds and consider the biholomorphism $\Phi$ of (3). %Then for any $p\in \cY_0$ and $|t|$ sufficiently small, there exists an open neighborhood $\cW\subset \cY$ such that $\cI_{\cY_t}(\cW\cap \cY_t)$ is generated by $\{f_1+t^{p+1}g_1, \cdots, f_d+t^{p+1} g_d\}$. 
Choose an arbitrary atlas $\cU=\{\cU_\alpha, \Phi_\alpha=(z_\alpha, t)\}$ covering $\cK_\epsilon$.
% and assume the embedding $\cY\subset \bC^N\times\bC$ is represented by the functions $w^b=w^b_\alpha(z_\alpha, t)$.
%Then we have:
We want to use induction to prove that there exists an $l$-adapted atlas for the embedding $\Phi|_{\cK_\epsilon}$ for $1\le l\le p$. Assume that this has been proved for $l-1$. This is trivially true when $l=1$. Then note that:
\begin{equation}\label{eq-pde22}
(F_r+t^{p+1}G_r)(w^b(z_\alpha, t))=0 \text{ for } 1\le r\le d.
\end{equation}
%So we assume that there is already an $(l-1)$-trivial embedding so that, under an $(l-1)$-adapted atlas, $\left.\frac{\partial^j w^b_\alpha}{\partial t^j}\right|_{t=0}=0$ for $1\le j\le l-1$. By 
Taking the $l$-th order derivative on both sides of \eqref{eq-pde22} and using the $(l-1)$-adapted property $\frac{\partial^j w^b_\alpha}{\partial t^j}|_{t=0}=0$ for $1\le j\le l-1$, we get:
\[
\left.\sum_{b=1}^N \frac{\partial F_r}{\partial w^b}\frac{\partial^{l}w^b_\alpha}{\partial t^{l}}\right|_{t=0}=0, \text{ for } 1\le r\le d.
\]
Since $\{F_r\}$ are defining functions of $\cK_0$, this means that the vector field 
$
\left.\sum_{b=1}^N \frac{\partial^{l}w_\alpha^b}{\partial t^{l}}\frac{\partial}{\partial w^b}\right|_{t=0}
$
is tangent to $\cK_0$. So there exists $c^i_\alpha=c^i_\alpha(z_\alpha)$ such that 
\begin{equation}\label{eq-plwba}
\left.\sum_{b=1}^N\frac{\partial^{l}w^b_\alpha}{\partial t^{l}}\frac{\partial}{\partial w^b}\right|_{t=0}=\sum_{i=1}^n c^i_\alpha\frac{\partial}{\partial z^i_\alpha}=\left.\sum_{b=1}^N\sum_{i=1}^n c^i_\alpha \frac{\partial w^b_\alpha}{\partial z^i_\alpha} \frac{\partial }{\partial w^b}\right|_{t=0}.
\end{equation}
Now define a new coordinate function:
\[
\tilde{z}^i_\alpha=z^i_\alpha+\frac{t^l}{l!} c^i_\alpha(z_\alpha).
\]
Then we get a new representation of the embedding on $\cU_\alpha$:
\[
\tilde{w}^b=w^b(z_\alpha, t)=w^b(z_\alpha(\tilde{z}_\alpha, t), t).
\]
Taking $l$-th order derivatives on both sides, by \eqref{eq-plwba} we get:
\[
\left.\frac{\partial^{l}\tilde{w}^b}{\partial t^{l}}\right|_{t=0}=\left.-\sum_{i=1}^n \frac{\partial w^b_\alpha}{\partial z_\alpha^i}c_\alpha^i+\frac{\partial^{l}w^b_\alpha}{\partial t^{l}}\right|_{t=0}=0.
\]
So by induction, we indeed get a $p$-adapted atlas on $\cK_\epsilon$ for $0<\epsilon\ll 1$. %For any relatively compact open subset $\cK\subset \cY$, we can choose a finite open covering of $\cK$ using the above set $\cW$ associated to finitely many points in $\cK$ and carry out the above construction to get a $p$-adapted atlas on $\cK_\epsilon$ for $\epsilon\ll 1$. 

%Conversely, let $\{\cU_\alpha, (z_\alpha,t)\}$ be an atlas. Then we have:
%\begin{equation}\label{eq-Fmvanish}
%F_m(w_1(z_\alpha, t), \dots, w_N(z_\alpha,t), t)\equiv 0.
%\end{equation}
%Taking derivative with respect to $t$ and evaluating at $t=0$, we get:
%\begin{equation}
%\left.\sum_{b=1}^N \frac{\partial F_m}{\partial w^b}\frac{\partial w_\alpha^b}{\partial t}+\frac{\partial F_m}{\partial t}\right|_{(w,t)=(w_\alpha(z_\alpha,0),0)}=\sum_{b=1}^N\frac{\partial F_m(w_\alpha(z_\alpha,0),0)}{\partial w^b}\frac{\partial w^b_\alpha(z_\alpha,0)}{\partial t}=0.
%\end{equation}
%Because $\{F_m(w,0)\}$ generates $\cI_{\cY_0}$ and $\cY_0$ is smooth, the Jacobian matrix $\{\partial F_m(w,0)/\partial w^b\}|_{w=w_\alpha(z_\alpha,0)}$ has rank equal to $n$. So we get the vanishing $\frac{\partial w^b_\alpha(z_\alpha,0)}{\partial t}=0$. Taking higher order derivatives of $t$ for \eqref{eq-Fmvanish} and using induction, it's easy to get the vanishing $\frac{\partial^l w^b_\alpha}{\partial t^l}$ for $1\le l\le p$. This means that the atlas $(z_\alpha,t)$ is a $p$-adapted atlas and the embedding is indeed $p$-trivial. 

\end{proof}

\subsection{Higher order deformation of normal isolated singularity via the higher order deformation of regular part}\label{subsec-hod}

%It will be important for us to understand the relation between higher order embedded deformations of affine algebraic varieties via the defining functions and higher order deformations of the regular part. 
Let $Z\subset \bC^N$ be an affine algebraic variety with exactly one singularity $o\in Z$ and we can assume that this singularity is the origin $0\in \bC^N$. Assume there is a holomorphic family of complex analytic varieties $\cZ\rightarrow \bB$ which is a deformation of the analytic germ $(\cZ_0, o)=(Z,o)$. For any $k\ge 0$, this induces a deformation over the analytic space $\bB(k)=(\bB, \mO_\bB/\mI_0^{k+1})$ where $\mI_0=(t)$ is the ideal sheaf of the point $0\in \bB$. Indeed, we have the flat morphism $\cZ(k):=\cZ\times_{\bB}\bB(k)\rightarrow \bB(k)$. 
\begin{defn}\label{def-ord2}
The order of the deformation $(\cZ, (\cZ_0,o))\rightarrow (\bB, 0)$ is defined to be the natural number:
\[
{\rm Ord}((\cZ, (\cZ_0,o))/(\bB, 0))=\max\left\{k+1; \cZ(k)\rightarrow \bB(k) \text{ is trivial}\right \}.
\]
If the pointed base $(\bB, 0)$ and the point $o\in \cZ$ are clear, we shall just write ${\rm Ord}(\cZ)$ for ${\rm Ord}((\cZ, (\cZ_0,o))/(\bB,0))$.
\end{defn}
It's well known that the higher order deformation theory in the algebraic category (see \cite{Artin}, \cite[Theorem 10.1]{Har10}) can also be developed in the analytic category (cf. \cite[Proposition 1.29]{GLS07}). Given 
a deformation of certain order, the space of possible deformations to the next order is a principal homogeneous space under ${\bf T}_Z^1$ i.e. an affine space without preferred origin. More precisely, suppose that there is a flat family $\cZ(k)\rightarrow \bB(k)$ and an extension to $\cZ^*(k+1)\rightarrow \bB(k+1)$ of $\cZ(k)$ with $\cZ^*(k)=\cZ^*(k+1)\times_{\bB(k+1)}\bB(k)=\cZ(k)$. Then the set of $(k+1)$-th order deformations that extend the $k$-th order deformation $\cZ(k)\rightarrow\bB(k)$ can be identified with ${\bf T}^1_Z$. In the special case at hand, there is a preferred origin given by the trivial deformation and this allows us to define a reduced Kodaira-Spencer class.

\begin{defn}\label{def-app-KSred}
Suppose there is a flat family $\cZ\rightarrow \bB$ of complex analytic varieties with $(\cZ_0, o)=(Z, o)$. Assume $\cZ(k)\rightarrow \bB(k)$ is trivial for a fixed $k\ge 0$. If the trivial deformation $\cZ^*(k+1):=Z\times \bB(k+1)$ is used as the base point so that $\cZ^*(k)=Z\times \bB(k)$ coincides with $\cZ(k)$, the corresponding class representing $\cZ(k+1)$ in ${\bf T}_Z^1$ is defined to be the $(k+1)$-th order Kodaira-Spencer class of $\cZ\rightarrow \bB$ and is denoted by ${\bf KS}_{\cZ}^{(k+1)}$. If $p+1=\Ord(\cZ)$, then we define the reduced Kodaira-Spencer class as $\KS^{\rm red}_{\cZ}=\KS^{(p+1)}_{\cZ}$.
\end{defn}
\begin{lem}\label{lem-ordF}
With the same notations as above, if $\Ord(\cZ)\ge p+1$, then there exist a small neighborhood $\cW$ of $o\in \bC^N\times\bC$ and a biholomorphism $\Phi$ of the form $\Phi(w,t)=(\Psi_t(w), t)$, $\Psi_0={\rm Id}$, from $\cW$ onto its image in $\bC^N\times\bC$ such that the ideal sheaf of $\Phi(\cZ)$ in $\Phi(\cW)$ is 
 generated by $F_1(w,t),\dots,F_d(w,t)$ satisfying $F_m(w,t)=F_m(w,0)+t^{p+1}G_m(w,t)$ on $\cW$ with $G_m(w,t)$ analytic in $\cW$.
\end{lem}
\begin{proof}
By assumption, there exists an isomorphism of quotients of power series rings $$\phi: \bC\{w^1,\dots, w^N, t\}\}/(F_1(w,t),\dots, F_d(w,t), t^{p+1})\rightarrow \bC\{\hat{w}^1,\dots, \hat{w}^N, t\}/(f_1(\hat{w}),\dots, f_d(\hat{w}), t^{p+1}),$$
where $F_1(w,t), \dots, F_d(w,t)$ are defining equations of the germ $(\cZ, (o,0))\subset (\bC^N\times\bC, (o,0))$. We will change the embedding of $(\cZ, (o,0))$ several times during the proof but will continue to use $F_m(w,t)$ to denote the defining equations of $(\cZ, (o,0))$ in each step.

 Assume $\phi$ is represented by functions $w^b=B_b(\hat{w}^1,\dots, \hat{w}^N, t)$. Then we have:
\begin{equation}\label{eq-Fvsf}
F_r(B_b(\hat{w},t), t)= \sum_{l=1}^{d} f_l(\hat{w}) h_{r,l}(\hat{w},t)+t^{p+1}u_r(\hat{w},t), \quad r=1,\dots, d,
\end{equation}  
where $h_{r,l}$ and $u_r$ are holomorphic near $o\in \bC^N\times\bC$.
We can assume $B_b(\hat{w}_1,\dots, \hat{w}^N, 0)=\hat{w}^b$ and $F_r(B_b(\hat{w}, 0), 0)=F_r(\hat{w},0)=:f_r(\hat{\omega})$ so that $h_{r,l}(\hat{w}, 0)=\delta_{rl}$. Multiplying \eqref{eq-Fvsf} by the inverse matrix $(h_{r,l})^{-1}$ (which exists for $|t|$ sufficiently small) and replacing $F_r$, we can assume $h_{r,l}(\hat{w},t)=\delta_{rl}$ so that the following identities hold:
\begin{equation}\label{eq-Fvsf2}
F_r(B_b(\hat{w},t),t)=f_r(\hat{w})+t^{p+1}u_r(\hat{w},t).
\end{equation}
We will prove by induction that there exist a small open neighborhood $\cW$ of $(o,0)\in \bC^N\times\bC$ and a biholomorphism $\Phi$ of the form $\Phi(w,t)=(\Psi_t(w),t)$, $\Psi_0={\rm Id}$, from $\cW$ onto its image $\bC^N\times\bC$ such that $\Phi(\cZ\cap \cW)$ is defined by equations $F_r(w,t)=0$, where 
% there is an embedding $\cZ\rightarrow \bC^N\times\bC$ (upon restricting to a relatively compact open subset of $\cZ$) and a small open neighborhood $\cW$ of $o\in \bC^N\times\bC$ such that $\cZ\cap \cW$ is defined by equations $\{F_r(w,t)=0\}$ where $w=(w^1, \dots, w^N)$ are coordinates of $\bC^N$ and the embedding satisfies the following vanishing identities for any $0\le l\le p$:
the following hold for any $0\le l\le p$:
\begin{eqnarray}\label{eq-dFdB}
%&&\frac{\partial^l (F_r(w,0)-f_r(w))}{\partial t^l}:=
\left.\frac{\partial^l (F_r(w,t)-f_r(w))}{\partial t^l}\right|_{t=0}=0 \quad \text{ and } \quad %\nonumber\\
%&& 
%\frac{\partial^{l}(B_b(\hat{w},0)-\hat{w})}{\partial t^{l}}:=
\left.\frac{\partial^{l} (B_b(\hat{\omega},t)-\hat{w})}{\partial t^{l}}\right|_{t=0}=0.
\end{eqnarray} 
This clearly implies the statement of the lemma. 

\eqref{eq-dFdB} holds for $l=0$. Assume \eqref{eq-dFdB} for $l-1$. Taking derivative for both sides of \eqref{eq-Fvsf2} with respect to $t$ $l$ times and evaluating at $t=0$, we get:
\begin{eqnarray}\label{eq-dFm}
\left.\sum_{b=1}^{N}\frac{\partial F_r(w,0)}{\partial w^b}\frac{\partial^l B_b(\hat{w},t)}{\partial t^l}+\frac{\partial^l F_r(w,t)}{\partial t^l}\right|_{t=0}=0.
\end{eqnarray}
Note that 
\begin{equation}\label{eq-highnormal}
v:=\left.\sum_{b=1}^Nv^b\frac{\partial}{\partial w^b}\right|_{\cZ_0}=-\sum_{b=1}^N\left.\frac{\partial^l B_b(\hat{w},t)}{\partial t^l}\right|_{t=0}\frac{\partial}{\partial w^b}\in H^0(\cZ_0, \Theta_{\bC^N}|_{\cZ_0})
\end{equation} 
is a globally defined vector field on $\cZ_0$. Let $\sigma(w, s)$ be the one parameter subgroup generated by a holomorphic extension of $v$. Then $\sigma(w,s)$ exists for $|s|$ sufficiently small on an open neighborhood of $o\in \cZ_0\subset \bC^N\times\{0\}$. Set 
\begin{equation}
\tilde{w}=\tilde{w}(w,t)=\sigma(w, t^l/l!), \quad \tilde{F}_r(\tilde{w}, t)=F_r(w(\tilde{w},t), t), \quad \tilde{B}(\hat{w},t)=\tilde{w}(w(\hat{w},t), t).
\end{equation} 
In particular, $\left.\frac{\partial^l \tilde{w}^b(w,t)}{\partial t^l}\right|_{t=0}=v^b=-\left.\frac{\partial^l w^b(\tilde{w},t)}{\partial t^l}\right|_{t=0}$ for $b=1,\dots, N$.
Then we get, since $l\ge 1$,
\begin{eqnarray*}
\left.\frac{\partial^l}{\partial t^{l} }(\tilde{F}_r(\tilde{w},t)-F_r(\tilde{w},0))\right|_{t=0}&=&\left. \frac{\partial^l}{\partial t^l}F_r(w(\tilde{w},t),t)\right|_{t=0}\\
&=&\left.\sum_{b=1}^N \frac{\partial F_r(w,0)}{\partial w^b}\frac{\partial^l w^b(\tilde{w},t)}{\partial t^l}+\frac{\partial^l F_r(w,t)}{\partial t^l}\right|_{t=0}\\
&=&\left.\sum_{b=1}^N\frac{\partial F_r(w,0)}{\partial w^b}\frac{\partial^l B_b(\hat{w},t)}{\partial t^l}+\frac{\partial^lF_r(w,t)}{\partial t^l}\right|_{t=0}=0. \quad (\text{ by } \eqref{eq-dFm})
\end{eqnarray*}
Moreover, we have the vanishing:
\begin{eqnarray*}
\left.\frac{\partial^l}{\partial t^l}\tilde{B}_b(\hat{w},t)\right|_{t=0}-\hat{w}&=&\left.\frac{\partial^l}{\partial t^l}\tilde{w}^b(B(\hat{w},t),t)\right|_{t=0}\\
&=&\left.\sum_{c=1}^N \frac{\partial \tilde{w}^b(w,0)}{\partial w^c} \frac{\partial^l}{\partial t^l}B_c(\hat{w},t)+\frac{\partial^l}{\partial t^l}\tilde{w}^b(w,t)\right|_{t=0}\\
&=&\left.\frac{\partial^l B_b(\hat{w},t)}{\partial t^l}\right|_{t=0}-\left.\frac{\partial^l B_b(\hat{w},t)}{\partial t^l}\right|_{t=0}=0.
\end{eqnarray*}
So the induction argument completes.

\end{proof}
%\begin{rem}
%The above lemma was pointed out to me by H.-J. Hein who also sketched a different proof by using \cite[Lemma 1.30]{GLS07}.
%By the discussion of the following paragraph, the above proof shows that $v$ in \eqref{eq-highnormal} is equal to $-\overline{\left\{\frac{\partial^l F_r(w,0)}{\partial t^l}\right\}}\in  H^0(Z, N_Z)$ as defined in \eqref{eq-barg} and is mapped to 0 under $\psi_Z: H^0(Z, N_Z)\rightarrow \fT^1_Z$. The argument is also consistent with the proof of Lemma \ref{lem-2trivial}.
%\end{rem}
If $\Ord(\cZ)\ge p+1$, then by Lemma \ref{lem-ordF}, after changing the embedding of $\cZ$, there exists a small open neighborhood $\cW$ of $(o, 0)\in \bC^N\times\bC$ such that  
$\mI_{\cZ}(\cW)$ is generated by $\{F_i(w,t)=F_i(w,0)+t^{p+1}G_i(w,t)\}$.
In particular $\mI_{\cZ_0}(\cW\cap \cZ_0)$ is generated by $\{f_1, \cdots, f_d\}$ where $f_i(w):=F_i(w,0)$ for $i=1,\dots,d$.
%and $\mI_{\cZ_t}(\cZ_t\cap \cW)$ is generated by $\{f_1+t^{p+1} g_1, \cdots, f_d+t^{p+1} g_d\}$. % (see also \cite[Proof of Proposition II.1.29]{GLS07} and \cite[Remark 6.1]{Artin}). 
%As explained in \cite[Section 6]{Artin}, 
Set $g_i(w)=G_i(w,0)$. 
The flatness condition of $\cZ\rightarrow \bB$ implies that $\{g_i\}$ (and hence $\{G_i\}$) determines a well-defined morphism (\cite[Proposition II.1.25]{GLS07} and \cite[Section 6]{Artin})
\begin{equation}\label{eq-barg}
\bar{g}: \mI_Z\rightarrow \mO_{\bC^N}/\mI_Z, \quad \sum_{r=1}^d f_r h_r\mapsto \sum_{r=1}^d g_r h_r.
\end{equation}
We have $\bar{g}\in {\rm Hom}_{\mathcal{O}_{\bC^N}}(\mI_Z, \mO_{\bC^N}/\mI_Z)={\rm Hom}_{\mathcal{O}_Z}(\mI_Z/\mI_Z^2, \mO_Z)=H^0(Z, N_Z)$. 
So if $\Ord(\cZ)\ge p+1$ then there is a well-defined class
\begin{equation}\label{eq-KSpsi}
-\psi_Z(\bar{g}) \in \fT^1_Z,
\end{equation}
where $\psi_Z: H^0(Z, N_Z)\rightarrow \fT^1_Z$ was defined by Schlessinger (see \eqref{exact1}). This class is exactly the Kodaira-Spencer class
$\KS^{(p+1)}_{\cZ}$ of Definition \ref{def-app-KSred}.
Notice that here we are working in the analytic category as in \cite{Schl1} and \cite{GLS07}. 

From now on assume $Z$ has a normal isolated singularity at $o$ and denote $U=Z\setminus\{o\}$. Schlessinger showed in \cite{Schl1} that the (infinitesimal) embeddable deformations can be determined by deformations of $U$ and ${\bf T}_Z^1$ is a subspace of $H^0(U, \Theta_U)$ (see Proposition \ref{prop-Sch} and \eqref{injexact}).
More precisely there are two exact sequences:
\begin{equation}\label{eq-2SchExact}
H^0(U, \Theta_{\bC^N}|_U)\rightarrow H^0(U, N_U)\stackrel{\psi_U}{\rightarrow} \fT^1_Z\rightarrow 0, \quad 0\rightarrow \fT^1_Z\stackrel{\tau_U}{\longrightarrow} H^1(U, \Theta_U)\rightarrow H^1(U, \Theta_{\bC^N}|_U).
\end{equation}
%We will need a variation of Schlessinger's result that says the deformation of $o\in Z$ is determined by the deformation of the complement of a strictly pseudo convex neighborhood of $o$ (see e.g. \cite{Miy}). Our result in Proposition \ref{prop-3eqwt}  will be a reflection of this principle.
Fix an embedding $o\in Z\hookrightarrow \bC^N$ and let $\{w_i\}_{i=1}^N$ be the standard coordinates of $\bC^N$ with $w_i(o)=0$. Choose a smooth strictly pluri-subharmonic function $\vphi$ on $\bC^N$ such that the following conditions are satisfied:
\begin{enumerate}
\item $\left.\vphi\right|_U>0$ is a strict plurisubharmonic function on $U=Z\backslash o$;
\item for any $\epsilon>0$ and $c>0$ the subset $\{p\in U; \epsilon<\vphi(p)<c\}$ is relatively compact in $U$;
\item for $c>0$, the subset $K_c:=\{p\in U; \vphi(p)\le c\}$ satisfies that $\partial K$ is compact and strongly pseudo-convex. 
\end{enumerate} 
Now assume that $(Z,o)$ is the germ of the vertex of an affine cone $Z=C(D, L)$ and $\cZ$ is an $\bC^*$-equivariant deformation of $Z$. We can then assume that the embedding of $\cZ$ into $\bC^N\times\bC$ is $\bC^*$-equivariant and the morphisms 
in the sequences \eqref{eq-2SchExact} are $\bC^*$-equivariant. Moreover we can choose $\vphi$ to be an $S^1$-invariant function so that the compact set $K_c$ becomes $S^1$-invariant. 
Fix $0<c_1\ll c_2<+\infty$. 
\begin{lem}
With the same notations as in the above paragraph, set $\mathscr{F}=\Theta_U$ or $N_U$. Then for $i\ge 1$, 
the natural morphism $R: H^i(Z\backslash K_{c_1}, \mathscr{F})\rightarrow H^i((Z\backslash K_{c_1})\cap \mathring{K}_{c_2}, \mathscr{F})$ induced by the inclusion is an isomorphism. 
\end{lem}
\begin{proof}
Since we are working with \v{C}ech cohomology, we first construct coverings by $S^1$-invariant Stein open sets in the following way. Let $\pi: Z\setminus\{o\}\rightarrow D$ be the natural projection realizing $Z\setminus \{o\}$ as a $\bC^*$-bundle over $D$. Choose a Stein covering $\{U^{D}_\alpha\}$ of $D$ and set $U_\alpha=\pi^{-1}(U^D_\alpha)\cap (Z\backslash K_{c_1})$. Similarly we get an $S^1$-invariant Stein covering $\{U_\alpha'\}$ of $(Z\backslash K_{c_1})\cap \mathring{K}_{c_2}$.

We first argue that $R$ is injective.
Represent the cohomology classes by \v{C}ech cocycles with respect to the above $S^1$-invariant Stein coverings. 
If $[\xi]=[\{\xi_{\alpha_{1}\dots\alpha_{i}}\}] \in H^i(Z\backslash K_{c_1}, \mathscr{F})$ satisfies $R([\xi])=0\in H^i((Z\backslash K_{c_1})\cap \mathring{K}_{c_2}, \Theta_Z)$. Then $\xi=\delta (\eta)$
is a coboundary where $\eta=\{\eta_{\alpha_1\alpha_2\dots \alpha_{i-1}}\}$ is a cochain (over $(Z\backslash K_{c_1})\cap \mathring{K}_{c_2}$). 

By using the result in \cite{Joh71}, we can 
decompose each component of $\eta$ into weight pieces. More precisely, we can write $\eta=\sum_k \eta_k$ where $\eta_k=\{(\eta_k)_{\alpha_1\dots\alpha_{i-1}}\}$ has weight $k$ under the $S^1$-action. 
Note that $\mathscr{F}$ is associated to a $\bC^*$-equivariant vector bundle over $\pi^{-1}(U^D_\alpha)$. So each $(\eta_k)_{\alpha_1\dots\alpha_{i-1}}$ is represented by holomorphic functions over $U'_{\alpha_1}\cap \cdots\cap U'_{\alpha_{i-1}}$ with respect to a $\bC^*$-equivariant trivialization of $\mathscr{F}$. Since homogeneous holomorphic functions on an annulus in $\bC$ uniquely extend to holomorphic functions on $\bC^*$, it is easy to see that $\eta$ extends uniquely to a holomorphic cochain of $\mathscr{F}$ with respect to the covering $\{U_\alpha\}$ such that $\xi=\delta(\eta)$ also holds on $Z\setminus K_{c_1}$. So $\xi$ is also a coboundary over $Z\setminus K_{c_1}$ and hence represents zero in $H^i(Z\setminus K_{c_1}, \mathscr{F})$.
%Since $S^1$ acts naturally on the space of cochains and commutes with the operator $\delta$, it also acts on the space of \v{C}ech cocycles. 

By using exactly the same argument, which again depends on the weight decomposition (using \cite{Joh71}) and the holomorphicity of cochains, we also prove that each cocycle over $(Z\setminus K_{c_1})\cap \mathring{K}_{c_2}$ extends to a cocycle over $Z\setminus K_{c_1}$. So the surjectivity of the morphism $R$ is also true.
%$S^1$ acts naturally on both $V_1:=H^1(Z\backslash K_{c_1}, \Theta_{Z})$ and $V_2:= H^1((Z\backslash K_{c_1})\cap \mathring{K}_{c_2}, \Theta_Z)$. 
 %By decomposing $\xi$ into weight pieces, we can assume $\xi$ is of weight $k\in \bZ$. This can be done by  $V_2(k)$.%\footnote{I am grateful to H.-J. Hein for pointing out this to me.}
%Then since $\delta$ commutes with the $S^1$-action, it's easy to see that each $\eta_\alpha$ is of weight $k$. Moreover, $\eta_\alpha$ extends uniquely to a weight $k$ holomorphic section of $\Theta_Z(U_\alpha)$ such that $\xi=\delta(\eta)$ also holds on $Z\backslash K_{c_1}$. So $\xi$ is a coboundary and represents zero in $H^1(Z\backslash K_{c_1}, \Theta_Z)$.
\end{proof}
With the same notations as in the above discussion, set $Y:=(Z\backslash K_{c_1})\cap \mathring{K}_{c_2}$ and $Y':=Z\backslash K_{c_1}$. 
By \cite[Theorem 15]{AnGr}, for any locally free sheaf $\mathscr{F}$ (whose depth is always $n$),
the natural restriction morphism $H^0(U, \mathscr{F})\rightarrow H^0(Y', \mathscr{F}|_{Y'})$ is an isomorphism and $H^1(U, \mathscr{F})\rightarrow H^1(Y', \mathscr{F}|_{Y'})$ is injective (since $n\ge 2$). 
Combining this with the above lemma, we get that the restriction morphism $\mu_0: H^0(U, \mathscr{F}|_Y)\rightarrow H^0(Y, \mathscr{F}|_Y)$ is an isomorphism and $\mu_1: H^1(U, \mathscr{F})\rightarrow H^1(Y, \mathscr{F}|_Y)$ is injective.
%For any fixed embedding $\cZ \rightarrow \bC^N\times\bB$, we can then apply Andreotti-Grauert's result for $\mathscr{F}=N_U$ and $\mathscr{F}=\Theta_U$ to see that the natural restriction morphism $\mu_0: H^0(U, N_U)\rightarrow H^0(Y, N_Y)$ is bijective and $\mu_1: H^1(U, \Theta_U)\rightarrow H^1(Y, \Theta_Y)$ is injective.

Now we have the following commutative diagram: 
\begin{equation}\label{eq-3morph}
\xymatrix{% @R=1.5pc @C=0.5pc{
& H^0(U, N_U)\ar[r]^{\mu_0}_{\cong}  \ar[d]^{\delta_U}  \ar@{->>}[ld]_{\psi_U} &H^0(Y, N_Y) \ar[d]^{\delta_Y}  \\
\fT^1_Z \ar@^{^{(}->}[r]^{\tau_U} & H^1(U, \Theta_U) \ar@^{{(}->}[r]^{\mu_1} & H^1(Y, \Theta_Y)
}
\end{equation}
Note that $\psi_U$ and $\tau_U$ are defined via Schlessinger's result in Proposition \ref{prop-Sch}. $\delta_U$ and $\delta_Y$ are connecting morphism as in \eqref{em2ab} (see also \eqref{preinj}).

\begin{prop}\label{prop-3eqwt}
With the above notations, let $\cY\rightarrow \bD$ be the holomorphic family of complex manifolds that is induced by $\cZ\rightarrow \bD$. % induces an holomorphic family of complex manifolds $\cY\rightarrow \bD$ with $\cY_0=Y$.
The following conditions are equivalent:
\begin{enumerate}
\item[(1)] $\Ord(\cZ)\ge p+1$ and hence there is a well-defined $\KS^{(p+1)}_{\cZ}\in \fT^1_Z$.
%\item There is a $p$-trivial embedding $\cY\rightarrow \bC^N\times\bB$.
%\[
%\left.\frac{\partial^l \left(w^b_\alpha(z_\alpha, t)-w^b_\alpha(z_\alpha, 0)\right)}{\partial t^l}\right|_{t=0}=0, \text{ for } 0\le l\le k.
%\]
\item[(2)] There is a $p$-trivial embedding of $\cY$ and hence there is a well-defined $\mathfrak{v}_{p+1}\in H^0(Y, N_Y)$.
\item[(3)] There is a $p$-trivial atlas covering $\cY_0$ and hence there is a well-defined $\theta_{p+1}\in H^1(Y, \Theta_Y)$.
\end{enumerate}
If one of the above conditions holds true, then we have the following identities:
\begin{equation}\label{eq-commu}
\delta_Y(\mathfrak{v}_{p+1})=\theta_{p+1}=\mu_1\circ \tau_U(\KS^{(p+1)}_{\cZ}) %, \quad \psi_Y(\mathfrak{v}_{p+1})=\KS^{(p+1)}_{\cZ/\bB}, \quad \tau_Y(\KS^{(p+1)}_{\cZ/\bB})=\theta_{p+1}.
\quad \text{ and } \quad
\KS^{(p+1)}_{\cZ}=\psi_U\circ \mu_0^{-1}(\mathfrak{v}_{p+1}).
\end{equation}
%\end{center}

%we have the identity $\bar{g}|_{Y}=\mathfrak{v}_{p+1}$ under the natural restriction map $H^0(Z, N_Z)\rightarrow H^0(Y, N_Y)$.
%\[
%\tau_Y\left(\KS^\red_{\cZ/\bB}\right)=\theta_{p+1},
%\]
%under the morphism $\tau_Y: \fT^1_Z\rightarrow H^0(Y, \Theta_Y)$.
\end{prop}
\begin{proof}
%We prove by induction. If $k=0$, then this is just Lemma \ref{}. Assume that we have proved the statement when $\Ord(\cZ/\bB)\ge k$. Now assume $\Ord(\cZ/\bB)=k+1>k$. Then by induction we can choose a 
Notice the equivalence of (2) and (3) was already proved in Lemma \ref{lem-2trivial}. So we only need to prove the equivalence of (1) and (2).

Assume $\Ord(\cZ)\ge p+1$. Then by Lemma \ref{lem-ordF}, after changing the embedding of $\cZ$, we can choose an open neighborhood $\cW$ of $(o, 0)\in \bC^N\times\bC$ such that $\cI_{\cZ}(\cW)$ is generated by $\{F_1(w,t)=F_1(w,0)+t^{p+1}G_1(w,t), \cdots, F_d(w,t)=F_d(w,0)+t^{p+1} G_d(w,t)\}$. 
By Lemma \ref{lem-3trivial}, the condition (2) holds, i.e. we get a $p$-trivial embedding and a $p$-adapted atlas.
%Choosing an atlas $\cU=\{\cU_\alpha, \Phi_\alpha=(z_\alpha, t)\}$ covering $\cY_0=Y$ and assume the embedding restricted to $\cY$ is represented by the $w^b=w^b_\alpha(z_\alpha, t)$.
%Then we have:
%\begin{equation}\label{eq-pde}
%(f_r+t^{p+1}g_r)(w^b(z_\alpha, t))=0.
%\end{equation}
%We get condition (2) by Lemma \ref{lem-3trivial}.

Now we verify the identities in \eqref{eq-commu} by using this $p$-adapted atlas. Set $f_r(w)=F_r(w,0)$ and $g_r(w)=G_r(w,0)$.
Taking $(p+1)$-th derivatives with respect to $t$ on both sides of the equation:
\[
(f_r+t^{p+1}G_r)(w^b(z_\alpha, t))=0,
\]
we get:
\[
\sum_{b=1}^N \frac{\partial f_r}{\partial w^b}\left.\frac{1}{(p+1)!}\frac{\partial^{p+1} w^b}{\partial t^{p+1}}\right|_{t=0}+ g_r=0.
\]
Comparing with the definition of $\mathfrak{v}_{p+1}$ in \eqref{eq-defvp} and the definition of $\bar{g}$ in equation \eqref{eq-barg}, this says $-\bar{g}|_Y=\mathfrak{v}_{p+1}\in H^0(Y, N_Y)$. It's clear that $\mathfrak{v}_{p+1}=-\mu_0(\bar{g}|_U)$ so that $-\bar{g}|_U=\mu_0^{-1}(\mathfrak{v}_{p+1})$ since $\mu_0$ is an isomorphism. On the other hand, we have $-\psi_U(\bar{g}|_U)=\KS^{(p+1)}_{\cZ}$. So we get:
\[
\psi_U\circ \mu_0^{-1}(\mathfrak{v}_{p+1})=\KS^{(p+1)}_{\cZ}.
\]
The identity $\delta_Y(\mathfrak{v}_{p+1})=\theta_{p+1}$ was proved in Lemma \ref{lem-pkv1}. The other identity is a consequence now:
\begin{eqnarray*}
\mu_1\circ \tau_U(\KS^{(p+1)}_{\cZ})&=&\mu_1\circ \tau_U\circ \psi_U\circ \mu_0^{-1}(\mathfrak{v}_{p+1})\\
&=& \mu_1\circ \delta_U\circ \mu_0^{-1}(\mathfrak{v}_{p+1})=\delta_Y(\mathfrak{v}_{p+1})=\theta_{p+1}.
\end{eqnarray*}

We are left to prove (2) implies (1). 
% and first recall that, by Proposition \ref{prop-Sch}, we have two short exact sequences:
%\[
%H^0(U, \Theta_{\bC^N}|_U)\rightarrow H^0(U, N_U)\stackrel{\psi_U}{\rightarrow} \fT^1_Z\rightarrow 0, \quad 0\rightarrow \fT^1_Z\stackrel{\tau_U}{\longrightarrow} H^1(U, \Theta_U)\rightarrow H^1(U, %\Theta_{\bC^N}|_U),
%\]
%such that $\tau_U\circ \psi_U$ is the connecting morphism $\delta_U: H^0(U, N_U)\rightarrow H^1(U, \Theta_U)$. 
Now assume (2) holds but on the contrary $\Ord(\cZ/\bB)=l+1$ with $l<p$. Then by using the defining functions $\{F_r(w,0)+t^{l+1}G_r(w,t)\}$ from Lemma \ref{lem-ordF} we have
$\psi_U(\bar{g})=-\KS^{(l+1)}_{\cZ} \neq 0\in \fT^1_Z$. So $\delta_U(\bar{g})=\tau_U\circ \psi_U(\bar{g})\neq 0\in H^1(U, \Theta_U)$ since $\tau_U$ is injective. 
By the discussion before Proposition \ref{prop-3eqwt}, $\mu_1$ is injective. So $\mu_1\circ \delta_U(\bar{g})\neq 0$. Hence 
\[
\theta_{l+1}=\delta_Y(\mathfrak{v}_{l+1})=-\mu_1\circ \delta_U(\bar{g})\neq 0.
\] 
On the other hand, we assumed that there is a $p$-trivial embedding $\tilde{\iota}_{\cY}$ with $p>l$. So by choosing $p$-adapted atlas the corresponding class $\tilde{\mathfrak{v}}_{l+1}:=\mathfrak{v}_{l+1}(\tilde{\iota}_{\cY})=0$. So 
$\delta_Y(\tilde{\mathfrak{v}}_{l+1})=0$. By the first item of Lemma \ref{lem-pkv1} and the second item of Lemma \ref{lem-thetap}, $\delta_Y(\mathfrak{v}_{l+1})=\delta_Y(\tilde{\mathfrak{v}}_{l+1}) \in H^1(Y, \Theta_Y)$. So we get a contradiction. 

%We claim that $\mathfrak{v}=-\bar{g}|_Y\neq 0$.  Assuming the claim, then by Lemma \ref{lem-chpte}, we get a contradiction.

\end{proof}

\section{Embeddings of submanifolds and deformations}

In the first subsection \ref{consdiff}, we will construct the ``most holomorphic" diffeomorphism between a neighborhood of a complex submanifold to a neighborhood of the zero section of its normal bundle. In particular, this allows us to get Proposition \ref{prop-linmap}. We do this by first using the ``deformation to normal cone" to construct a ``holomorphic family of neighborhoods" as the deformation of a neighborhood of the zero section of the normal bundle. We also construct a $(k-2)$-trivial (resp. $(k-1)$-trivial) atlas on this family under the assumption that the embedding is $(k-1)$-linearizable (resp. $(k-1)$-comfortable). Then we use the similar method as that in section \ref{subsec-clKS} to get the wanted diffeomorphism. Our main goal in this section is a technical Proposition \ref{embvsdef} which relates the reduced Kodaira-Spencer class of the ``holomorphic family of neighborhoods" to the obstructions to splitting embedding and comfortable embedding.

\subsection{Construction of comparison diffeomorphism and $(k-1)$-trivial atlas}\label{consdiff}

As mentioned above, the construction of diffeomorphism $F$ in Proposition \ref{prop-linmap} and \ref{prop-cftmap} uses a construction in algebraic geometry called deformation to the normal cone (see \cite[Chapter 5]{Fult}). This is a way to degenerate a neighborhood of $S\hookrightarrow X$ to a neighborhood of $S\hookrightarrow N_S$. The construction is simply to blow-up the submanifold $S\times\{0\}\subset X\times\bC$ which gives a total family  $\tilde{\mX}=Bl_{S\times\{0\}}(X\times\bC)$ with the projection $\pi: \tilde{\mX}\rightarrow\mathbb{C}$. The central fibre $\tilde{\mX}_0=Bl_SX\cup E$ is the union of two components. The exceptional divisor $E=\mathbb{P}(N_S\oplus\mathbb{C})$ is the projective compactification of the normal bundle $N_S$ of $S\subset X$. In this way we can view $S\hookrightarrow X$ as an analytic deformation of $S_0\hookrightarrow N_S$. More precisely, we will construct an analytic family $\mW$ as an open neighborhood of $\mathcal{S}\cong S\times\bC \hookrightarrow \tilde{\cX}$. In other words, $\mW$ is considered as a deformation of a neighborhood of $S\rightarrow X$.

The main result of this subsection is the following proposition which contains the statement of proposition \ref{prop-linmap}. 
%Under appropriate conditions on the embedding of $S\subset X$, we will use the analytic family $\mW$ and the method in Section \ref{subsec-clKS} to construct a diffeomorphism from a neighborhood of $S\hookrightarrow N_S$ to a neighborhood of $S\hookrightarrow X$ that satisfies the requirement in the statement of Proposition \ref{prop-linmap}.
For the construction in its proof, 
we refer to the appendix section \ref{linearizable} for preliminary results from \cite{ABT} that will enable us to read out the precise order of holomorphicity of the diffeomorphism constructed.
\begin{prop}\label{prop-goodmap}
Assume that $S$ is a smooth submanifold of $X$. If $S\hookrightarrow X$ is $(k-1)$-linearizable, then the following statements are true:
\begin{enumerate}
\item[(1)] there is an holomorphic family of complex manifolds $\cW$ such that $\cW_0$ is a neighborhood of $S_0\hookrightarrow N_S$ and $\cW_1=:W$ is a small neighborhood of $S\subset X$,
and there is a $(k-2)$-trivial atlas covering $\cW_0$.
\item[(2)]
 there is a diffeomorphism $F: \cW_0\rightarrow F(\cW_0)\subset W$ where $W=\cW_1$ , such that for any $j\ge 0$, there exists a constant $C_j$ such that $F$ satisfies
\begin{equation}\label{pbJ3}
\|\nabla_{\tilde{g}_0}^j(F^*J-J_0)\|_{\tilde{g}_0}\le C_j \tilde{r}^{k-j} \text{ on } W_0;
\end{equation}
\end{enumerate}
If $S\hookrightarrow X$ is furthermore $(k-1)$-comfortable, then the above properties can be improved as follows:
\begin{enumerate}
\item[(3)]
There is a $(k-1)$-trivial atlas cvering $\cW_0$;
\item[(4)]
There is a local decomposition of $\Phi:=F^*J-J_0$ into four types of components (see \eqref{eq-hvdecomp}):
\[
\Phi=\Phi_{v}^{h}+\Phi_{h}^{v}+\Phi_{v}^{v}+\Phi_{h}^{h},
\]
such that, for any $j\ge 0$, the following estimates hold over $W_0$ for a uniform constant $C_j$:
\begin{equation}\label{eq-cftimprove}
\|\nabla_{\tilde{g}_0}^j \Phi_h^v\|_{\tilde{\omega}_0}\le C_j \tilde{r}^{k+1-j}, \; \|\nabla_{\tilde{g}_0}^j\Phi_v^v\|_{\tilde{\omega}_0}\le C_j \tilde{r}^{k+1-j}, \; \|\nabla_{\tilde{g}_0}^j \Phi_{h}^h\|_{\tilde{\omega}_0}\le C_j \tilde{r}^{k-j}, \; \|\nabla_{\tilde{g}_0}^j \Phi_v^h\|_{\tilde{\omega}_0}\le C_j \tilde{r}^{k-j}.
\end{equation}
\end{enumerate}
\end{prop}
The improved estimates \eqref{eq-cftimprove} will be used to prove Proposition \ref{prop-cftmap} in section \ref{sec-ACK}.
\begin{proof}

Assume that the embedding $S\hookrightarrow X$ is $(k-1)$-linearizable.
%$|\phi^h_v|_{CY}\sim |\phi^h_v||\xi| \sim |\bar{\partial}_v\tilde{R}^p_{\beta\alpha}||\xi|\sim |\xi|^{k+1}$. 
%$|\phi^v_h|_{CY}\sim |\phi^v_h||\xi|^{-1}\sim |\bar{\partial}_h\tilde{R}^r_{\beta\alpha}||\xi|^{-1}\sim |\xi|^{k}$.
%$|\phi^v_v|_{CY}\sim |\phi^v_v|\sim |\xi|^k$, $|\phi^h_h|_{CY}=|\phi^h_h|\sim |\xi|^{k+1}$.
By Theorem \ref{crilinear} in appendix \ref{linearizable} we can find coordinate charts $\{ V_\alpha, (z_\alpha)\}$ of $X$ near the submanifold $S$ such that $S\cap V_\alpha=\{z_\alpha^1=\dots=z_\alpha^m=0\}$ and the transition functions on $V_\alpha\cap V_\beta$ are given by:
\begin{equation}\label{linearcoord}
\left\{
\begin{array}{ccll}
z_\beta^r&=&\sum_{s=1}^m (a_{\beta\alpha})^r_s(z''_\alpha)z_\alpha^s+R^r_{k}, & \mbox{ for } r=1,\dots, m, \\
&&&\\
z_\beta^p&=&\phi_{\beta\alpha}^p(z''_\alpha)+R^p_{k}, & \mbox{ for } p=m+1,\dots, n;
\end{array}
\right.
\end{equation}
where we have denoted by $z''=(z^{m+1}_\alpha, \cdots, z^n_\alpha)$ the tangent variables, which can also serve as coordinates on $S$.
Here $R^{r}_{k}, R^{p}_{k}\in \mathcal{I}_S^{k}$.  We also consider coordinate charts $\{V_\alpha\times\mathbb{C},(z_\alpha,t)\}$ on $X\times\bC$ so that $S\times\{0\}= \{z_\alpha^1=\dots=z_\alpha^m=t=0\}$.

Consider the blow up $\pi: \tilde{\mX}:=Bl_{S\times \{0\}}(X\times\bC)\rightarrow X\times\bC$ with the exceptional divisor $E=\mathbb{P}(N_S\oplus\mathbb{C})$. $E$ is the projective compactification of the normal bundle $N_S\rightarrow S$ and $S_0$ sits inside $N_S\subset E\subset \tilde{\mX}_0\subset \tilde{\mX}$ as the zero section of $N_S\rightarrow S$. The subset $\pi^{-1}(V_\alpha\times \mathbb{C})\subset \tilde{\mX}$ is defined as the following subvariety of $V_\alpha\times\mathbb{C}\times\mathbb{P}^{m}$:
\begin{eqnarray*}
&&\left\{(z_\alpha^r, z_\alpha^{p}, t, [Z_\alpha^r, T]); (z_\alpha^r, z_\alpha^p)\in V_\alpha, t\in \mathbb{C}, z_\alpha^r Z_\alpha^s-z_\alpha^s Z_\alpha^r=0, \right.\\
&& \hskip 2cm \left. z_\alpha^r\cdot T- t \cdot Z_\alpha^r=0; \mbox{ for } r,s=1,\cdots, m; p=m+1,\cdots, n\right\},
\end{eqnarray*}
where $[Z_\alpha^r, T]$ are homogenous coordinates on $\mathbb{P}^{m}$.
Near $S_0$, the coordinate $T\neq 0$, and so we can define new coordinate charts $\{w_\alpha, t\}$ such that the map $\pi$ is given by:
\[
z_\alpha^1=t w_\alpha^1,\dots, z_\alpha^m=t w_\alpha^m;\quad z_\alpha^{m+1}=w_\alpha^{m+1}, \dots, z_\alpha^{n}=w_\alpha^n; \quad t=t.
\]
Without loss of generality we can assume $V_\alpha=\{z_\alpha; |z_\alpha|< \epsilon\}$ for sufficiently small $\epsilon>0$. Then if we denote the polydisc on the total space:
\[
\mathcal{U}_\alpha=\{(t,w_\alpha); |t|< 2, |w_\alpha|< \epsilon\},
\]
then $\pi(\mathcal{U}_\alpha)\subset V_\alpha\times \mathbb{C}$, and when $t\neq 0$ satisfies $|t|<2$, 
\[
\pi(\mU_\alpha)\cap X_t\cong \left\{z_\alpha; |z^r_\alpha|< 2 \epsilon t, |z^p_\alpha|< \epsilon; \mbox{ for } r=1,\dots,m; p=m+1, \dots, n\right\}.
\] 
Denote by $\mathcal{S}$ the strict transform of $S\times\mathbb{C}$ on $\tilde{\mX}$. Let $\pi_1$ be the composition $\tilde{\cX}\rightarrow X\times \bC\rightarrow \bC$. For any $a>0\in \mathbb{R}$, denote $\cS_{|t|<a}=\pi^{-1}\left(\{t; |t|<a\}\right)$. Then the collection of open sets $\{\mU_\alpha\}$ is a covering of $\mathcal{S}_{|t|<1}$ inside the total space $\tilde{\mX}$ and on $\mathcal{U}_\alpha$ the ideal sheaf $\mI_{\mathcal{S}}$ is generated by $w_\alpha^1, \cdots, w^m_\alpha$. Denote $\mU=\bigcup_{\alpha}\mU_\alpha$. We can find a small neighborhood $\mW$ of $\mathcal{S}_{|t|<1}\subset \tilde{\mX}$ such that $
\mW\subset\subset \mU$. 
Denote $w'_\alpha=(w^1_\alpha, \dots, w^m_\alpha)$, $w''_\alpha=(w^{m+1}_\alpha, \dots, w^{n}_\alpha)$ and define 
\[
\tilde{R}^r_{k}(t; w'_\alpha, w''_\alpha)=t^{-k} R^r_k(t w'_\alpha, w''_\alpha),\quad \tilde{R}^p_{k}(t; w'_\alpha, w''_\alpha)=t^{-k} R^p_k(t w'_\alpha, w''_\alpha).
\]
Then $\tilde{R}^r_{k}\in \mathcal{I}_{\mathcal{S}}^{k}$, $\tilde{R}^p_{k}\in\mI_{\mathcal{S}}^{k}$. Note that $\{ \mU_\alpha\cap \mW_t, w_\alpha\}_\alpha$ form an atlas covering $\mW_t:=\pi^{-1}(X_t)\cap \mW$ for $|t|<1$. The transition function on $(\mU_\alpha\cap\mW_t)\cap(\mU_\beta\cap\mW_t)$ is given by:
\begin{equation}\label{linearblowup}
\left\{
\begin{array}{ccll}
w_\beta^r&=&\sum_{s=1}^m (a_{\beta\alpha})^r_s(w''_\alpha)w_\alpha^s+t^{k-1} \tilde{R}_{k}^{r}, & \mbox{ for } r=1,\dots, m, \\
&&&\\
w_\beta^p&=&\phi_{\beta\alpha}^p(w''_\alpha)+t^{k} \tilde{R}_{k}^{p}, & \mbox{ for } p=m+1,\dots, n.
\end{array}
\right.
\end{equation}
So we get a $(k-2)$-trivial atlas covering $\mW_0$ in the sense of Definition \ref{def-ptrivial}. Next we can construct the diffeomorphism that we want. Choose a partition of unity $\{\rho_\alpha,\tilde{\rho}\}$ subordinate to the covering $\{\mathcal{U}_\alpha, \tilde{\mX}\backslash\ov{\mathcal{W}}\}$. In particular, ${\rm Supp}(\rho_\alpha)\subset \mathcal{U}_\alpha, {\rm Supp} (\tilde{\rho})\cap \mathcal{W}=\emptyset$. As in Section \ref{KSdef}, define the differentiable vector field in the small neighborhood $\mW$ of $\mathcal{S}_{|t|<1}\subset \tilde{\mX}$:
\begin{eqnarray*}
\mathbb{V}&=&\sum_{\alpha}\rho_\alpha\left(\frac{\partial}{\partial t}\right)_\alpha=\sum_{i=1}^n\left(\sum_\alpha\rho_\alpha\frac{\partial f^i_{\beta\alpha}(w_\alpha,t)}{\partial t}\right)\frac{\partial}{\partial w_\beta^i}+\left(\frac{\partial}{\partial t}\right)_\beta\\
&=&\sum_{r=1}^m \sum_{\alpha} \rho_\alpha \partial_t(t^{k-1}\tilde{R}^r_{k})\frac{\partial}{\partial w^r_\beta}
%+\left(\frac{\partial}{\partial t}\right)_\beta=\sum_{i=1}^{n}\sum_{\alpha}\rho_\alpha \left( k t^{k-1} \tilde{R}_{\beta\alpha}^{i}+t^k \frac{\partial \tilde{R}^i_{\beta\alpha}}{\partial t}\right)\frac{\partial}{\partial w_\beta^i}+\left(\frac{\partial}{\partial t}\right)_\beta.
+\sum_{p=m+1}^n\sum_{\alpha}\rho_\alpha \partial_t(t^{k}\tilde{R}_{k}^{p})\frac{\partial}{\partial w_\beta^p}+\left(\frac{\partial}{\partial t}\right)_\beta.
\end{eqnarray*}

 Let $\{\sigma(s); s\in (-\epsilon, \epsilon)\}$ be the one-parameter sugroup generated by $\re(\mathbb{V})$ which exists when $\epsilon$ is sufficiently small. Then we get a map $\sigma(s): \mathcal{W}\cap \tilde{\cX}_0\rightarrow \cU\cap \tilde{\cX}_s$ which gives a diffeomorphism to its image. 
 
 Note that the vector field $\mathbb{V}$ is tangent to $\mathcal{S}$ so that $\sigma(s)$ preserves $\mathcal{S}$. Denote $\mathcal{J}$ the complex structure on the total space $\tilde{\mX}$ of blow up. Denote
\[
\Phi(s)=\sigma(s)^*\mathcal{J}-\mathcal{J}.
\]
Then we can calculate:
\begin{eqnarray*}
\dot{\Phi}(s)&=&\frac{d}{ds}(\sigma(s)^*\mathcal{J})=\mathcal{L}_{\re(\mathbb{V})}\mathcal{J}=\bar{\partial}\mathbb{V}+\overline{\bar{\partial}\mathbb{V}} \\
&=&\left.\sum_{r=1}^m \sum_\alpha [\partial_t(t^{k-1}\tilde{R}^r_{k})](\bar{\partial}\rho_\alpha)\otimes\frac{\partial}{\partial w_\beta^r}+
\sum_{p=m+1}^n \sum_\alpha [\partial_t(t^{k}\tilde{R}^p_{k})](\bar{\partial}\rho_\alpha)\otimes\frac{\partial}{\partial w_\beta^p}\right|_{t=s}\\
&&+\text{ complex conjugates}.
%&=&\sum_{i,j=1}^n \sum_\alpha (\partial_{\ov{w}^{j}_\beta}\rho_\alpha) \partial_t(t^k \tilde{R}^i_{\beta\alpha}) d\ov{w}^j_\beta \otimes\frac{\partial}{\partial w_\beta^i}.
\end{eqnarray*}
%\item $(k=0)$
%If we just take $k=0$, then we can slightly improve the order. This is because that we can always
Assume $\tilde{\omega}_0$ is a smooth K\"{a}hler metric on the open set $\mW$. Because both $\tilde{R}^r_{k}, \tilde{R}^p_{k} \in \mathcal{I}_{\mathcal{S}}^{k}$, we get:
\[
|\dot{\Phi}|_{\tilde{\omega}_0}\le C s^{\max\{0,k-2\}} |w'|^{k}.
\]
So we can integrate to get:
\begin{equation}\label{Jlinear}
|\Phi(s)|_{\tilde{\omega}_0}=|\sigma(s)^*\mathcal{J}-\mathcal{J}|_{\tilde{\omega}_0}=\left|\int_0^s \sigma(s)^*(\mathcal{L}_{\mathbb{V}}\mathcal{J}) ds \right|_{\tilde{\omega}_0}\le C s^{k-1}|w'|^{k}.
\end{equation}
If $\tilde{r}$ is the distance function to $S_0$ with respect to $\tilde{g}_0$, then $\tilde{r}$ is comparable to the norm $|w'|$. So the above estimate proves the inequality \eqref{pbJ2} for $j=0$. The higher order estimates of $\Phi$ can be proved in the same way by taking higher order Lie derivative of $\mathbb{V}$.
%When $0<|s|<\epsilon$ for $\epsilon$ sufficiently small, we get a map
%$\sigma(s): \mW\cap \tilde{\mX}_0\rightarrow \mU\cap \tilde{\mX}_s$ which gives a diffeomorphism to its image. By construction, $\mW\cap \tilde{\mX}_0$ is a small neighborhood of $S_0$ and $\mU\cap \tilde{\mX}_s$ is a small neighborhood
%of $S\subset X=X_s$ for $|s|<1$. %Note that the vector field $V$ is tangent to $\mathcal{S}$.

Next we show that if $S\hookrightarrow X$ is $(k-1)$-comfortable, the estimates of of some components of $\Phi$ can be improved. In this case, by Theorem \ref{coordinate}, we can choose the coordinate charts such that the following holds:
\begin{equation}\label{cmftcoord}
\left\{
\begin{array}{ccll}
z_\beta^r&=&\sum_{s=1}^m (a_{\beta\alpha})^r_s(z''_\alpha)z_\alpha^s+R^r_{k+1}, & \mbox{ for } r=1,\dots, m, \\
&&&\\
z_\beta^p&=&\phi_{\beta\alpha}^p(z''_\alpha)+R^p_{k}, & \mbox{ for } p=m+1,\dots, n.
\end{array}
\right.
\end{equation}
where $R^r_{k+1}\in \mathcal{I}_{S}^{k+1}$, $R^p_{k}\in \mathcal{I}_S^{k}$. Similarly as before, denote $\tilde{R}_{k+1}^r(t; w'_\alpha, w''_\alpha)=t^{-(k+1)}R_{k+1}^r(t w'_\alpha, w''_\alpha)$ and $\tilde{R}_{k}^p(t; w'_\alpha, w''_\alpha)=t^{-k}R_{k}^p(t w'_\alpha, w''_\alpha)$. Then
$\tilde{R}_{k+1}^r\in \mathcal{I}_{\mathcal{S}}^{k+1}$ and $\tilde{R}_{k}^p\in \mathcal{I}_{\mathcal{S}}^{k}$.
On the total space of the deformation to the normal cone, we have 
\begin{equation}\label{cmftblowup}
\left\{
\begin{array}{ccll}
w_\beta^r&=&\sum_{s=1}^m (a_{\beta\alpha})^r_s(w''_\alpha)w_\alpha^s+t^{k} \tilde{R}_{k+1}^r, & \mbox{ for } r=1,\dots, m, \\
&&&\\
w_\beta^p&=&\phi_{\beta\alpha}^p(w''_\alpha)+t^{k}\tilde{R}_{k}^p, & \mbox{ for } p=m+1,\dots, n.
\end{array}
\right.
\end{equation}
Notice that this is a $(k-1)$-trivial atlas covering $\mW_0$ in the sense of definition \ref{def-ptrivial}.

Similarly as before the differentiable vector field $\mathbb{V}$ (see Section \ref{KSdef}) becomes 
\begin{eqnarray}\label{Vectorfield}
\mathbb{V}%=\sum_\gamma \rho_\gamma \sum_{i=1}^n\frac{\partial f_{\beta\gamma}^i(w_\gamma,t)}{\partial t}\frac{\partial}{\partial w_\beta^i}+\left(\frac{\partial}{\partial t}\right)_\beta\\
&=&\sum_{i=1}^n\left(\sum_\alpha\rho_\alpha\frac{\partial f^i_{\beta\alpha}(w_\alpha,t)}{\partial t}\right)\frac{\partial}{\partial w_\beta^i}+\left(\frac{\partial}{\partial t}\right)_\beta\nonumber\\
&=&\sum_{r=1}^{m}\sum_{\alpha} \rho_\alpha[\partial_t (t^{k}\tilde{R}^r_{k+1})]\otimes \frac{\partial}{\partial w_\beta^r}+\sum_{p=m+1}^n\sum_{\alpha} 
\rho_\alpha[\partial_t (t^{k}\tilde{R}^p_{k})]\otimes\frac{\partial}{\partial w_\beta^p}+\left(\frac{\partial}{\partial t}\right)_\beta.
\end{eqnarray}
Use the same notations $\sigma(s)$, $\mathcal{J}$, $\Phi(s)$ and $\dot{\Phi}(s)$ as before. We have:
%be the flow generated by $V$ which exists when $|t|\le \delta$ for sufficiently small $\delta$. Denote $\mathcal{J}$ the complex structure on $\mU$efine
%\[
%\phi(t)=\sigma(t)^*\mathcal{J}-\mathcal{J}.
%\]
\begin{eqnarray*}
\dot{\Phi}(s)&=&\frac{d}{ds}(\sigma(s)^*\mathcal{J})=\mathcal{L}_{\re(\mathbb{V})}\mathcal{J}=\bar{\partial}\mathbb{V}+\overline{\bar{\partial}\mathbb{V}}\\
&=&\left.\sum_{r=1}^m \sum_\alpha [\partial_t(t^{k}\tilde{R}^r_{k+1})](\bar{\partial}\rho_\alpha)\otimes\frac{\partial}{\partial w_\beta^r}+
\sum_{p=m+1}^n \sum_\alpha [\partial_t(t^{k}\tilde{R}^p_{k})](\bar{\partial}\rho_\alpha)\otimes\frac{\partial}{\partial w_\beta^p}\right|_{t=s}\\
&&+\text{ complex conjugates }.
%&=&\sum_{i,j=1}^n \sum_\alpha (\partial_{\ov{w}^{j}_\beta}\rho_\alpha) \partial_t(t^k \tilde{R}^i_{\beta\alpha}) d\ov{w}^j_\beta \otimes\frac{\partial}{\partial w_\beta^i}.
\end{eqnarray*}
We assume the index $v\in \{1,\dots, m, \overline{1}, \dots, \overline{m}\}$, $h\in \{m+1,\dots, n, \overline{m+1},\cdots, \overline{n}\}$ and decompose $\Phi$ into four types of components:
\begin{equation}\label{eq-hvdecomp}
\Phi=\Phi_{v}^{h}+\Phi_{h}^{v}+\Phi_{v}^{v}+\Phi_{h}^{h}:=\phi_{v}^h d w^{v}\otimes \partial_{w^h}+\phi_{h}^{v}dw^h\otimes \partial_{w^v}+\phi_{v}^v dw^v\otimes \partial_{w^v}+\phi_{h}^{h}dw^h\otimes \partial_{w^h}.
\end{equation}
Again we assume $\tilde{\omega}_0$ is a smooth K\"{a}hler metric on $\mathcal{W}$.

Since $\tilde{R}^r_{k+1}\in \mathcal{I}_{\mathcal{S}}^{k+1}$, $\tilde{R}_{k}^p\in \mathcal{I}_{\mathcal{S}}^{k}$, it's easy to see that:
\[
|\dot{\phi}_h^v|\le C s^{k-1} |w'|^{k+1}, |\dot{\phi}_v^v|\le C s^{k-1} |w'|^{k+1}, |\dot{\phi}_{h}^h|\le Cs^{k-1} |w'|^{k}, |\dot{\phi}_v^h|\le Cs^{k-1}|w'|^{k}.
\]
Integrating these, we get:
\begin{equation}\label{eq-improvePhi}
|\Phi_h^v|_{\tilde{\omega}_0}\le C s^{k}|w'|^{k+1}, |\Phi_v^v|_{\tilde{\omega}_0}\le Cs^{k} |w'|^{k+1}, |\Phi_{h}^h|_{\tilde{\omega}_0}\le Cs^{k} |w'|^{k}, |\Phi_v^h|_{\tilde{\omega}_0}\le Cs^{k} |w'|^{k}.
\end{equation}
When  $|s|<\epsilon$ with $\epsilon$ sufficiently small, 
since $\tilde{r}$ is comparable to $|w'|$, we get the estimates that improve the estimates in \eqref{eq-cftimprove} for $j=0$. The higher order estimates can be proved similarly by taking higher order Lie derivatives of $J$ with respect to $\mathbb{V}$. 
\end{proof}

\subsection{Order of embedding via deformation to the normal cone}

Let $S$ be a smooth submanifold of a complex manifold $X$. We will denote by $\pi_S: N_S\rightarrow S$ the normal bundle of $S$ inside $X$ and by $\Theta_{N_S}$ the tangent sheaf on the total space of $N_S$. The natural $\mathbb{C}^*$ action on $N_S$ induces $\mathbb{C}^*$ actions on various cohomology groups. Since we will use various \v{C}ech cohomology groups frequently, we choose a Stein covering $\{\hat{U}_\alpha\}$ of $N_S$ by first choosing a Stein covering $\{U_\alpha\}$ of $S$ and then defining $\hat{U}_S=\pi_S^{-1}(U_\alpha)$. In particular, $\hat{U}_\alpha$ is invariant under the natural $\mathbb{C}^*$ action. On each $\hat{U}_\alpha$, choose a coordinate system $w_\alpha=\{w_\alpha', w_\alpha''\}=\{w^r_\alpha, w^p_\alpha \; | \; r=1, \dots, m; p=m+1, \dots, n\}$ such that $w^r_\alpha$ are fiber variables and $w^p_\alpha$ are base variables. Then the $\bC^*$-action is given by
\[
t\cdot \{w'_\alpha, w''_\alpha\}=\left\{t^{-1} w'_\alpha, w''_\alpha\right\}.
\]
The transition function on $\hat{U}_\alpha\cap \hat{U}_\beta$ is of the form:
\begin{equation}
\left\{
\begin{array}{ccll}
w_\beta^r&=&\sum_{s=1}^m (a_{\beta\alpha})^r_s(w''_\alpha)w_\alpha^s, & \mbox{ for } r=1,\dots, m, \\
&&&\\
w_\beta^p&=&\phi_{\beta\alpha}^p(w''_\alpha), & \mbox{ for } p=m+1,\dots, n.
\end{array}
\right.
\end{equation}
Let $\mathbf{V}$ be a \v{C}ech cohomology space $H^i(X, \mathscr{F})$ where $X$ is an analytic space with a $\bC^*$-action and $\mathscr{F}$ is the coherent sheaf associated to a $\bC^*$-equivariant vector bundle $F\rightarrow X$. The space of cocycles of $\mathscr{F}$ with respect to a $\bC^*$-invariant Stein covering has a continuous $S^1$-action. By the result from \cite{Joh71}, this space can be written as the closure of the algebraic direct sum of eigenspaces. This induces a weight decomposition of the cohomology space $\mathbf{V}=H^i(X, \mathscr{F})$. We will denote by $\mathbf{V}(-k)$ the subspace of elements of weight $-k$.
%if $\mathbb{C}^*$ acts on a vector space $\mathbf{V}$, then we will denote by $\mathbf{V}(-k)$ subspace of elements of weight $k$ under the $\bC^*$-action. 
\begin{lem}\label{basexact}
For $k\ge 0$, we have the following commutative diagram of exact sequences
\begin{equation}
\xymatrix %@R=1pc @C=0.7pc
{
H^{1}(N_S, \Theta_{N_S}\otimes\mathcal{I}_S^{k+1})(-k) \ar[d]^{\cong}_{\mathfrak{R}_k} \ar[r]^{\mathfrak{N}_k'} & H^1(N_S, \Theta_{N_S}\otimes\mathcal{I}_S^k)(-k) \ar[r]^{\mathfrak{T}_k'} \ar[d]^{\cong}_{\mathfrak{I}_k}& H^1(S, \Theta_S\otimes\mathcal{I}_S^k/\mathcal{I}_S^{k+1}) \ar@{=}[d] \\
H^{1}(S, N_S\otimes \mathcal{I}_S^{k+1}/\mathcal{I}_S^{k+2}) \ar[r]^{\mathfrak{N}_k} &H^1(N_S, \Theta_{N_S})(-k) \ar[r]^{\mathfrak{T}_k} & H^1(S, \Theta_S\otimes\mathcal{I}_S^k/\mathcal{I}_S^{k+1}) 
}
\end{equation}
where the morphisms are given as follows:
\begin{enumerate}
\item $\frak{I}_k, \frak{N}'_k$ are induced by inclusion of sheaves;
\item $\frak{T}'_k$, $\frak{R}_k$ will be defined in the proof, and $\frak{R}_k$ is an isomorphism;
\item $\frak{N}_k=\mathfrak{I}_k\circ \mathfrak{N}_k'\circ \mathfrak{R}_k^{-1}$ and $\frak{T}_k=\frak{T}'_k\circ \frak{I}_k^{-1}$ are defined by using the commutativity of the diagram. 
\end{enumerate}

\end{lem}
Note that in the above diagram the sheaf $\cI^k_S/\cI^{k+1}_S$ is a sheaf supported on $S$. The Kodaira-Spencer class $\theta_k$ of the atlas constructed in the proof of Proposition \ref{prop-goodmap} lives in 
$H^1(N_S, \Theta_{N_S})(-k)$ and the bottom exact sequence will serve to compare $\theta_k$ to Abate-Bracci-Tovena's obstruction in Proposition \ref{embvsdef}.
\begin{proof}
We first notice that $\mathfrak{T}_k'$ is well defined as the composition of maps:
\[
H^1(N_S, \Theta_{N_S}\otimes \mathcal{I}_S^k)\rightarrow H^1(S, \Theta_{N_S}|_S\otimes \mathcal{I}_S^{k}/\mathcal{I}_S^{k+1})\rightarrow H^1(S, \Theta_{S}\otimes \mathcal{I}_S^{k}/\mathcal{I}_S^{k+1}).
\]
In the last map, we used the holomorphic splitting $\left.\Theta_{N_S}\right|_{S}=\Theta_S\oplus N_S$. Similarly $\mathfrak{R}_k$ is well defined as the composition of maps:
\[
H^1(N_S, \Theta_{N_S}\otimes \mathcal{I}_S^{k+1})\rightarrow H^1(S, \Theta_{N_S}|_S\otimes \mathcal{I}_S^{k+1}/\mathcal{I}_S^{k+2})\rightarrow H^1(S, N_{S}\otimes \mathcal{I}_S^{k+1}/\mathcal{I}_S^{k+2}).
\]
Let's first show that the first row of sequence is exact. Assume that $\theta_k\in H^1(N_S, \Theta_{N_S}\otimes\mathcal{I}_S^k)(-k)$.
Let $\theta_k$ be represented by a cocycle $\{\theta_{\alpha\beta}\}$ with respect to a $\bC^*$-invariant covering of $N_S$. Then by \cite{Joh71}, we can write $\theta_{\alpha\beta}$ as a convergent series
$\theta_{\alpha\beta}=\sum_\ell \theta_{\alpha\beta,\ell}$
where $\theta_{\alpha\beta,\ell}$ has weight $\ell$. Because $\delta$ commutes with the $\bC^*$-action, we know that $\theta_{\alpha\beta,\ell}$ is also a cocycle.
Because $\theta_k=[\{\theta_{\alpha\beta}\}]$ has weight $-k$ and the weight decomposition of cohomology is induced by the weight decomposition on the space of cocycles, we know that $[\{\theta_{\alpha\beta,\ell}\}]=0$ if $\ell\neq k$. So we can assume that $\theta_k$ is represented by a weight $(-k)$ cocycle:
\[
(\theta_k)_{\beta\alpha}=\sum_{r=1}^m b^r_{\beta\alpha}(w)\frac{\partial}{\partial w_\beta^r}+\sum_{p=m+1}^n c^p_{\beta\alpha}(w)\frac{\partial}{\partial w_\beta^p},
\]
where $b_{\beta\alpha}^r, c_{\beta\alpha}^p \in \mathcal{I}_S^k$. Since $\frac{\partial}{\partial w_\beta^r}$ (resp. $\frac{\partial}{\partial w_\beta^p}$) has weight $1$ (resp. $0$), we know that $b^r_{\beta\alpha}$ (resp. $c^p_{\beta\alpha}$) is homogeneous of degree
$(k+1)$ (resp. $k$) in $w'=\{w_\beta^r\}$. 
Then 
\[
\left(\mathfrak{T}'_k(\theta_k)\right)_{\beta\alpha}=\sum_{p=m+1}^n [c^p_{\beta\alpha}(w)]_{k+1}\frac{\partial}{\partial w_\beta^p}.
\]
If $\mathfrak{T}'_k(\theta_k)=0$, then we can write:
\begin{eqnarray*}
\sum_{p=m+1}^n [c^p_{\beta\alpha}(w)]_{k+1}\frac{\partial}{\partial w_\beta^p}&=&\sum_{p=m+1}^n [d^p_\beta]_{k+1}\frac{\partial}{\partial w_\beta^p}-\sum_{p=m+1}^n [d^q_\alpha]_{k+1}\frac{\partial}{\partial w_\alpha^q} \mbox{ over } \hat{U}_\alpha\cap \hat{U}_\beta.
%&=&[c_\beta^p]_{k+1}\frac{\partial}{\partial w_\beta^p}-[c_\alpha^q]_{k+1}\left(\frac{\partial w_\beta^p}{\partial w_\alpha^q}\frac{\partial}{\partial w_\beta^p}+\frac{\partial w^r_\alpha}{\partial w^q_\beta}%\frac{\partial}{\partial w^r_\alpha}\right)\\
%&=&[c_\beta^p]_{k+1}-[c^q_\alpha]_{k+1}\frac{w_\beta^p}{w_\alpha^q}.
\end{eqnarray*}
We can assume $d_\beta^p$ and $d_\beta^q$ are homogeneous of degree $k$. Then it's easy to see that $c^p_{\beta\alpha}=d^p_\beta-d^q_\alpha \frac{\partial w_\beta^p}{\partial w_\alpha^q}$. 
So if we define
\begin{eqnarray*}
(\tilde{\theta}_k)_{\beta\alpha}&=&(\theta_k)_{\beta\alpha}-\sum_{p=m+1}^n d_\beta^p\frac{\partial}{\partial w_\beta^p}+\sum_{q=m+1}^n d^q_\alpha \frac{\partial}{\partial w^q_\alpha} 
\end{eqnarray*}
then it's easy to see that $(\tilde{\theta}_k)_{\beta\alpha}\in H^0(\hat{U}_\alpha\cap \hat{U}_\beta, \Theta_{N_S}\otimes\mathcal{I}_S^{k+1})(-k)$ and we have $\theta_k=\mathfrak{N}_k'(\tilde{\theta}_k)$.
%\end{proof}

To show $\mathfrak{R}_k$ is an isomorphism, we will construct its inverse. Assume $\mathfrak{h}\in H^1(S, N_S\otimes\mathcal{I}_S^{k+1}/\mathcal{I}_S^{k+2})$, we
can represents it as a cocycle:
\begin{equation}\label{eq-repkh}
\mathfrak{h}_{\beta\alpha}=\sum_{r=1}^m [b_{\beta\alpha}^r]_{k+2}\frac{\partial}{\partial w_\beta^r}.
\end{equation}
We can assume $b_{\beta\alpha}^r$ is homogeneous of degree $k+1$ in $w'_\beta=\{w^r_\beta\}$. Then because of homogeneity the cocycle condition of $\{\mathfrak{h}_{\beta\alpha}\}$ becomes:
\begin{equation}\label{eq-homcocycle}
\sum_{r=1}^m \left(b^r_{\beta\alpha}(w_\beta)\frac{\partial}{\partial w^r_\beta}+b^r_{\alpha\gamma}\frac{\partial}{\partial w^r_\alpha}+b^r_{\gamma\beta}\frac{\partial}{\partial w^r_\gamma}\right)=0.
\end{equation}
So if we define 
\[
\mathfrak{h}'_{\beta\alpha}:=\mathfrak{R}_k^{-1}(\mathfrak{h}_{\beta\alpha})=\sum_{r=1}^m b_{\beta\alpha}^r\frac{\partial}{\partial w_\beta^r} \in H^0(\hat{U}_\alpha\cap \hat{U}_\beta, \Theta_{N_S}\otimes\mathcal{I}_S^{k+1})(-k),
\]
where $\frac{\partial}{\partial w^r_\beta}$ etc. are considered as tangent vectors along the fibres of $N_S\rightarrow S$,
then by \eqref{eq-homcocycle} $\{\mathfrak{h}'_{\beta\alpha}\}$ satisfies the cocycle condition and hence represents a cohomology class in $H^1(N_S, \Theta_{N_S}\otimes\mathcal{I}_S^{k+1})$ of weight $-k$.
%\begin{lem}
%For $k\ge 0$, $H^1(N_S, \Theta_{N_S})(-k)\cong H^1(N_S, \Theta_{N_S}\otimes \mathcal{I}_S^{k})(-k)$
%\end{lem}
%\begin{proof}
Now we can define $\mathfrak{N}_k$. Choose $\mathfrak{h}\in H^1(S, N_S\otimes \mI_S^{k+1}/\mI_S^{k+2})$ represented by the cocycle as in \eqref{eq-repkh} such that $b^p_{\beta\alpha}$ is homogeneous of degree $k+1$ in $w'_\beta=\{w^r_\beta\}$. Then we define:
\[
\mathfrak{N}_k(\mathfrak{h}_{\beta\alpha})=\mathfrak{I}_k\circ\mathfrak{N}'_k \circ\mathfrak{R}_k^{-1}(\mathfrak{h}_{\beta\alpha})=\sum_{r=1}^m b^r_{\beta\alpha}\frac{\partial}{\partial w^r_\beta}\in H^0(\hat{U}_\alpha\cap \hat{U}_\beta, \Theta_{N_S})(-k).
\]

Using similar homogeneity argument, one can also construct an inverse of $\mathfrak{I}_k$ showing that it is an isomorphism. 
Indeed, for any $\theta\in H^1(N_S, \Theta_{N_S})$ of weight $(-k)$, we can choose a $\mathbb{C}^*$-equivariant \v{C}ech cocycle $\{\theta_{\beta\alpha}\}$ of weight $(-k)$ representing $\theta$. On $\hat{U}_{\alpha}\cap \hat{U}_{\beta}$, we can write:
\[
\theta_{\beta\alpha}=\sum_{r=1}^m a^r_{\beta\alpha}(w_\alpha) \frac{\partial}{\partial w^r_\beta}+\sum_{p=m+1}^{n}b^p_{\beta\alpha}(w_\alpha)\frac{\partial}{\partial w^p_\beta}.
\]
Since $\frac{\partial}{\partial w_\beta^r}$ (resp. $\frac{\partial}{\partial w^p_\beta}$) has weight 1 (resp. 0), we see that $a_\alpha^r(w_\alpha)$ (resp. $b_\alpha^p$) is homogeneous of degree $(k+1)$ (resp. $k$) in $w_\alpha^r$. In particular, $a_{\beta\alpha}^r\in \mathcal{I}_D^{k+1}$ and $b_{\beta\alpha}^p\in \mathcal{I}_D^{k}$. So $\theta_{\beta\alpha}\in H^0(\hat{U}_\alpha\cap \hat{U}_\beta, \Theta_{N_S}\otimes\mathcal{I}_S^{k})$ and $\{\theta_{\beta\alpha}\}$ represents a cohomology class in $H^1(N_S, \Theta_{N_S}\otimes\mathcal{I}_S^{k})$ of weight $(-k)$.

\end{proof}

Our main result in this subsection is the following technical proposition which under appropriate assumption re-interprets the obstructions to splitting and comfortable embeddings via the deformation to normal cone construction:
\begin{prop}\label{embvsdef}
Assume that $S$ is $(k-1)$-comfortably-embedded submanifold of $X$ for some $k\ge 1$ and let $(\rho_{k-1}, \boldsymbol{\nu}_{k-1})$ be a $(k-1)$-comfortable pair. Then for the holomorphic family of complex manifolds $\mW$ from Proposition \ref{prop-goodmap}, the associated $k$-order Kodaira-Spencer class $\theta_k\in H^1(\cW_0, \Theta_{\cW_0})$ extends uniquely to a class in $H^1(N_S, \Theta_{N_S})$. This extension lies in the $(-k)$-weight space and will still be denoted by $\theta_k$. Moreover $\theta_k$ satisfies the following properties under the exact sequence from Lemma \ref{basexact}:
\begin{enumerate}
\item
$\mathfrak{T}_k(\theta_k)=\mathfrak{g}^{\rho_{k-1}}_k\in H^1(S, \Theta_S\otimes\mathcal{I}_S^{k}/\mathcal{I}_S^{k+1})$ is the obstruction to $k$-splitting relative to $\rho_{k-1}$. As a consequence, if $S$ is not $k$-splitting relative to $\rho_{k-1}$, then $\theta_k\in H^1(N_S, \Theta_{N_S})(-k)$ is non zero. 
\item
If $S$ is $k$-splitting relative to $\rho_{k-1}$, i.e. we have a $k$-th order lifting $\rho_{k}$ such that $\phi_{k,k-1}\circ\rho_k=\rho_{k-1}$, then $\theta_k=\mathfrak{N}_k
(\mathfrak{h}_k^{\rho_k})$ where $\mathfrak{h}_k^{\rho_{k}}\in H^1(S, N_S\otimes \mathcal{I}_S^{k+1}/\mathcal{I}_S^{k+2})$ is the obstruction to $k$-comfortably-embedding with respect to $\rho_k$.
\end{enumerate}

\end{prop}

\begin{proof}[Proof of Proposition \ref{embvsdef}] 
Suppose that the embedding $S\hookrightarrow X$ is $(k-1)$-comfortably embedded.  
%By Theorem \ref{coordinate} and the calculations leading to \eqref{cmftblowup}, 
As shown in \eqref{cmftblowup}, we can choose a $(k-1)$-comfortable atlas adapted to $(\rho_{k-1}, \nu_{k-1})$ such that we have induced atlas on the blow up with coordinate changes given by:
\begin{equation}\label{cmftbl2}
\left\{
\begin{array}{ccll}
w_\beta^r&=&\sum_{s=1}^m (a_{\beta\alpha})^r_s(w''_\alpha)w_\alpha^s+t^{k} \tilde{R}_{k+1}^{r}, & \mbox{ for } r=1,\dots, m, \\
&&&\\
w_\beta^p&=&\phi_{\beta\alpha}^p(w''_\alpha)+t^{k} \tilde{R}_{k}^{p}, & \mbox{ for } p=m+1,\dots, n.
\end{array}
\right.
\end{equation}
We can substitute the transition function in \eqref{cmftbl2} into \eqref{redKSccl} above to get:
%So we get ${\rm ord_0}{\bf KS}_{\mU}\ge k$, and as before, we have the cocycle 
\begin{equation}\label{kobs}
(\theta_k)_{\beta\alpha}=\frac{1}{k!}\sum_{i=1}^n\left.\frac{\partial^k f^i_{\beta\alpha}(w_\alpha,t)}{\partial t^{k}}\right|_{t=0}\frac{\partial}{\partial w_\beta^i}=\sum_{r=1}^m \tilde{R}^r_{k+1}(0; w_\alpha) \frac{\partial}{\partial w_\beta^r}+\sum_{p=m+1}^n \tilde{R}^p_{k}(0; w_\alpha)\frac{\partial}{\partial w_\beta^p},
\end{equation}
where in the last expression, $w_\alpha$ and $w_\beta$ are related by the following relation on $\tilde{\mX}_0$ near $S_0\cong S$:
\begin{equation}
\left\{
\begin{array}{ccll}
w_\beta^r&=&\sum_{s=1}^m (a_{\beta\alpha})^r_s(w''_\alpha)w_\alpha^s, & \mbox{ for } r=1,\dots, m, \\
&&&\\
w_\beta^p&=&\phi_{\beta\alpha}^p(w''_\alpha), & \mbox{ for } p=m+1,\dots, n;
\end{array}
\right.
\end{equation}
which is nothing but the transition function on $N_S$.
%\begin{rem}\label{remleading}
Recall that $\tilde{R}^r_{k+1}(t; w_\alpha', w_\alpha'')=t^{-(k+1)}R_{k+1}(t w'_\alpha, w''_\alpha)$ and $\tilde{R}^p_k(t; w'_\alpha, w_\alpha'')=t^{-k}R^p_k(t w'_\alpha, w''_\alpha)$. 
So $\tilde{R}^r_{k+1}(0; w_\alpha)$ (resp. $\tilde{R}^p_{k}(0; w_\alpha)$) is nothing but the $(k+1)$-th (resp. $k$-th) order leading term of $R^r_{k+1}(w_\alpha)$ (resp. $R^p_k(w_\alpha)$) in its Taylor expansion with respect to $w'_\alpha$. 
%\end{rem}

Since $w_\alpha'$ are global coordinates on the whole $\hat{U}_\alpha\subset N_S$, we see that $(\theta_k)_{\beta\alpha}$ is actually defined over $\hat{U}_\alpha\cap \hat{U}_\beta\subset N_S$. This shows the statement that $\theta_k\in H^1(\cW_0, \Theta_{\cW_0})$ extends uniquely to a class in $H^1(N_S, \Theta_{N_S})$ which will still be denoted by $\theta_k$.

So if we denote by $\pi_S: N_S\rightarrow S$ the natural projection of the normal bundle to its base, and by $\hat{U}_\alpha=\pi_S^{-1}(\mathcal{U}_\alpha\cap \tilde{\mX}_0\cap S_0)$ the $\mathbb{C}^*$-invariant open set on $N_S$, then we have:
\[
(\theta_k)_{\beta\alpha}\in H^0\left(\hat{U}_\alpha\cap \hat{U}_\beta, \Theta_{N_S}\otimes \mathcal{I}_S^{k}\right).
\]
So we get a \v{C}ech cohomology class:
\[
\theta'_k:=\{(\theta_k)_{\beta\alpha}\}\in \check{H}^1(N_S, \Theta_{N_S}\otimes \mathcal{I}_{S}^k).
\] 
From \eqref{kobs} and homogeneity of $\tilde{R}^r_{k+1}$, $\tilde{R}^p_k$ in $w'_\alpha$, we see that $\theta'_k$ has weight $(-k)$ under the natural $\mathbb{C}^*$-action on $N_S$.
%Note that
%\[
%V_\alpha\cap S=\{z_\alpha^r=0, |z_\alpha^p|<\epsilon; 1\le r\le m, m+1\le p\le n\}\cong \hat{U}_\alpha\cap S_0.
%\]
%Then we observe that
%\[
%(\theta_k)_{\beta\alpha}\in H^0(\hat{U}_\alpha\cap \hat{U}_\beta, \Theta_{\mU\cap \tilde{\mX}_0}\otimes\mI_{S_0}^{k})
%\]
%can be viewed as representing a cocycle in the cohomology group $H^1(\mU\cap \tilde{\mX}_0, \Theta_{\mU\cap \tilde{\mX}_0}\otimes \mI_{S_0}^{k})$. 
When we restrict to $S_0=S\subset N_S$ and mod-out by $\mathcal{I}_{S_0}^{k+1}$, we get:
\begin{eqnarray}\label{tthetak}
(\mathfrak{g}_k)_{\beta\alpha}:=\left.(\theta_k)_{\beta\alpha}\right|_{S_0}&=&\sum_{r=1}^m[\tilde{R}^{r}_{k+1}(0; w'_\alpha,w''_\alpha)]_{(k)}\frac{\partial}{\partial w^r_\beta}+\sum_{p=m+1}^n[\tilde{R}^p_{k}(0; w'_\alpha,w''_\alpha)]_{(k)}
\frac{\partial}{\partial w^p_\beta}\nonumber\\
&=&\sum_{p=m+1}^n[\tilde{R}^p_k(0; w'_\alpha, w''_{\alpha})]_{(k)}\frac{\partial}{\partial w_\beta^p},
\end{eqnarray}
which form a cocycle 
\[
\{(\mathfrak{g}_k)_{\beta\alpha}\}\in \check{H}^1(\{U_\alpha\}, \Theta_{N_S}|_{S_0}\otimes\mathcal{I}_{S_0}^{k}/\mathcal{I}_{S_0}^{k+1})=
\check{H}^1(\{U_\alpha\}, N_{S_0}\otimes \mathcal{I}_{S_0}^{k}/\mathcal{I}_{S_0}^{k+1})\oplus \check{H}^1(\{U_\alpha\}, \Theta_{S_0}\otimes \mathcal{I}_{S_0}^{k}/\mathcal{I}_{S_0}^{k+1}).
\]
%Because $\left.[\tilde{R}_{k}(w_\alpha)]_{k}\right|_{S}=\tilde{R}_{k}(0,w_\alpha)$, comparing with \eqref{spltccl}, we see that 
%\[
%(\theta_k)_{\beta\alpha}=(\mathfrak{h}_{k})_{\beta\alpha}=\tilde{R}^{r}_{k}(0,w_\alpha)\otimes \frac{\partial}{\partial w^r_\beta}\in H^0(V_\alpha\cap V_\beta\cap S, N_S\otimes \mathcal{I}_S^{k}/\mathcal{I}_S^{k+2}).
%\]
%is nothing but the obstruction to $k$-comfortable embedding.
In the last equality, we used the holomorphic splitting $\Theta_{N_S}|_{S_0}=\Theta_{S_0}\oplus N_{S_0}$. 
% since $\mathcal{U}\cap \tilde{\mX}_0\subset N_{S}\subset P(N_S\oplus\mathbb{C})$. 
Because we assumed that $S$ is $(k-1)$-comfortably-embedded, the component in the first summand is $0$ as seen in \eqref{tthetak}. So using the notation in Lemma \ref{basexact}, we can write $\mathfrak{g}_k=\mathfrak{T}_k(\theta_k)$.
By Proposition \ref{obsplit} we see that 
%\begin{equation}\label{gccl}
%(\mathfrak{g}_k)_{\beta\alpha}=\sum_{p=m+1}^n [\tilde{R}_{k}^p(0;w'_\alpha, w''_\alpha)]_{(k)} \otimes \frac{\partial}{\partial w^p_\beta}\in H^0(\hat{U}_\alpha\cap \hat{U}_\beta\cap S_0, \Theta_{S_0}\otimes \mathcal{I}_{S_0}^{k}/\mathcal{I}_{S_0}^{k+1})
%\end{equation}
$\mathfrak{g}_k=\{(\mathfrak{g}_k)_{\beta\alpha}\}$ is the obstruction to the existence of $\rho_k$ satisfying $\phi_{k, k-1}\circ\rho_{k}=\rho_{k-1}$. In other words, $\mathfrak{g}_k^{\rho_{k-1}}:=\mathfrak{g}_k$ is the obstruction to $k$-splitting relative to $\rho_{k-1}$. So we get the first part of Proposition \ref{embvsdef}.
%\begin{lem}
%Assuming $S$ is $(k-1)$-comfortably embedded, then there is a well-defined cohomology class $\theta'^{(k)}\in H^1(N_S, \Theta_S\otimes \mathcal{I}_S^{k})$ which has weight $(-k)$ under the natural $\mathbb{C}^*$ action.  Moreover $\theta'^{(k)}$ induces the obstruction $\mathfrak{g}_k\in H^1(S, \Theta_S\otimes\mathcal{I}_S^{k}/\mathcal{I}_S^{k+1})$ to $k$-splitting under the composition of linear maps:
%\[
%H^1(N_S, \Theta_{N_S}\otimes \mathcal{I}_S^k)\rightarrow H^1(S, \Theta_{N_S}|_S\otimes \mathcal{I}_S^{k}/\mathcal{I}_S^{k+1})\rightarrow H^1(S, \Theta_{S}\otimes \mathcal{I}_S^{k}/\mathcal{I}_S^{k+1}).
%\]
%\end{lem}

Now if we assume that the obstruction to $k$-splitting vanishes, i.e. the above $\mathfrak{g}_k^{\rho_{k-1}}$  vanishes, then by Theorem \ref{crilinear} 
the transition functions in \eqref{cmftbl2} can be improved to
\begin{equation}\label{lnrzbl2}
\left\{
\begin{array}{ccll}
w_\beta^r&=&\sum_{s=1}^m (a_{\beta\alpha})^r_s(w''_\alpha)w_\alpha^s+t^{k} \tilde{R}_{k+1}^{r}, & \mbox{ for } r=1,\dots, m, \\
&&&\\
w_\beta^p&=&\phi_{\beta\alpha}^p(w''_\alpha)+t^{k+1} \tilde{R}_{k+1}^{p}, & \mbox{ for } p=m+1,\dots, n.
\end{array}
\right.
\end{equation}
Substituting this into \eqref{kobs}, $(\theta_k)_{\beta\alpha}$ now becomes:
\begin{equation}\label{kobs2}
(\theta_k)_{\beta\alpha}=\frac{1}{k!}\sum_{i=1}^n\left.\frac{\partial^k f^i_{\beta\alpha}(w_\alpha,t)}{\partial t^k}\right|_{t=0}\frac{\partial}{\partial w_\beta^i}=\sum_{r=1}^m \tilde{R}^r_{k+1}(0; w_\alpha) \frac{\partial}{\partial w_\beta^r}.%+\sum_{p=m+1}^n \tilde{R}^p_{k}(0,w_\alpha)\frac{\partial}{\partial w_\beta^p}.
\end{equation}
So we see that in this case $(\theta_k)_{\beta\alpha}\in H^0(\hat{U}_\alpha\cap \hat{U}_\beta, \Theta_{N_S}\otimes\mathcal{I}_S^{k+1})$. Again we get a weight $(-k)$ \v{C}ech cohomology class:
\[
\theta''_k:=\{(\theta_k)_{\beta\alpha}\}\in \check{H}^1(\{\hat{U}_\alpha\}, \Theta_{N_S}\otimes\mathcal{I}_S^{k+1})(-k),
\]
which satisfies $\mathfrak{N}_k'(\theta''_k)=\theta_k$. When we restrict to $S_0$ and mod out by $\mathcal{I}_{S_0}^{k+2}$, we get:
\begin{equation}\label{hccl}
(\mathfrak{h}_k)_{\beta\alpha}:=\left.(\theta_k)_{\beta\alpha}\right|_{S_0}=\sum_{r=1}^m [\tilde{R}^r_{k+1}(0; w'_\alpha,w''_\alpha)]_{(k+1)} \frac{\partial}{\partial w_\beta^r}\in
 H^0(\hat{U}_\alpha\cap \hat{U}_\beta\cap S_0, N_{S_0}\otimes \mI_{S_0}^{k+1}/\mI_{S_0}^{k+2}).
\end{equation}
Comparing with \eqref{spltccl}, we see that 
$\mathfrak{h}_k:=\{(\mathfrak{h}_k)_{\beta\alpha}\}%=\{\mathfrak{h}_{\beta\alpha}\}=\tilde{R}^{r}_{k}(0,w_\alpha)\otimes \frac{\partial}{\partial w^r_\beta}\in H^0(V_\alpha\cap V_\beta\cap S, N_S\otimes \mathcal{I}_S^{k}/%\mathcal{I}_S^{k+2}).
$ is nothing but the obstruction $\mathfrak{h}_k^{\rho_k}$ to $k$-comfortable embedding with respect to the $k$-splitting $\rho_k$. 
By Lemma \ref{basexact}, we can write $\theta_k''=\mathfrak{R}_k^{-1}(\mathfrak{h}_k)$. 
%Note that by Theorem \ref{crilinear}, $(k-1)$-comfortably-embeded + $k$-splitting = $k$-linearizable. So from the above discussion, we get the second part of Proposition \ref{embvsdef}.
%\begin{lem}
%Assuming $S$ is $k$-linearizable (i.e $(k-1)$-comfortably-embedded and $k$-splitting), then there is a well defined cohomology class $\theta''_k\in H^1(N_S, \Theta_{N_S}\otimes \mathcal{I}_S^{k+1})$ which has weight $(-k)$ such that $\theta''_k$ induces the obstruction $\mathfrak{h}^{(k)}\in H^1(S, N_S\otimes \mathcal{I}_S^{k+1}/\mathcal{I}_S^{k+2})$ to $k$-comfortably-embedding under the composition of linear maps:
%\[
%H^1(N_S, \Theta_{N_S}\otimes \mathcal{I}_S^k)\rightarrow H^1(S, \Theta_{N_S}|_S\otimes \mathcal{I}_S^{k}/\mathcal{I}_S^{k+1})\rightarrow H^1(S, N_{S}\otimes \mathcal{I}_S^{k+1}/\mathcal{I}_S^{k%+2}).
%\]
%\end{lem}
\end{proof}
\section{Special case: $S=D$ is an ample divisor}
One of the main goals of this section is to prove Theorem \ref{thm-eqwt}. The proof is essentially based on the construction in Section \ref{consdiff} and Proposition \ref{embvsdef}. Roughly speaking,
under the assumption that $D\rightarrow X$ is $(m-1)$-comfortable, we get $(m-1)$-trivial atlas by the construction in Section \ref{consdiff} and hence a reduced 
Kodaira-Spencer class defined as a class in $H^1(U, \Theta_U)$. Then Proposition \ref{embvsdef} is also used to show that this reduced Kodaira-Spencer is non-trivial if the embedding $D\rightarrow X$ is not $m$-comfortable (and $n\ge 3$). 
Finally by Proposition \ref{prop-3eqwt}, the reduced Kodaira-Spencer class near the ``infinity" divisor via coordinate changes coincides with the reduced Kodaira-Spencer class for the deformation of the cone defined in Definition \ref{def-app-KSred}. This allows us to complete the proof. 

\subsection{Degeneration to the projective cone}
%In this section, we assume $S=D$ is an ample divisor in $X$.
%Our next result relates the number $k$ in the above proposition to the weight of deformation of some singularity.   
%Denote $L$ to be the line bundle $[D]$ defined by $D$. 
%The component $X\subset \tilde{\mX}_0$ from above can be contracted to a point. See figure \ref{presolution} for illustration. In this way, we get a flat family of irreducible projective varieties $\mathcal{X}\rightarrow \mathbb{C}$. For simplicity, we will say that $\mX$ is the degeneration obtained from the {\it Contracted Deformation to the Normal Cone associated to $(X, D)$}. There are equivalent ways to realize this construction, one using the graph construction (see section \ref{secamp}), and another using pure algebra (see Proof of Lemma \ref{normalem}).

%by the line bundle $\mathcal{L}_1=\pi^*L-E$ 
%Under appropriate cohomological assumptions (see Section \ref{secamp}), the contracted central fibre $\mX_0$ is normal and hence coincides with the projective cone $\ov{C}=\bar{C}(D, N_D)$. Then we get a $\mathbb{C}^*$-equivariant degeneration $\mX$ of $X$ to the projective cone $\ov{C}=\bar{C}(D, N_D)$ which in general has an isolated singularity $\underline{o}$. Note that $\underline{o}$ is simply the image of the infinity divisor $D_\infty$ of $E=\mathbb{P}(N_D\oplus\mathbb{C})$ under the contraction. 
%\[
%\bar{C}={\rm Proj}\left(\bigoplus_{r\ge 0}^{+\infty}\bigoplus_{m=0}^{r}( H^0(D,L^{m})\cdot x_{n+1}^{m-r})\right).
%\]
%Denote by $\mathcal{D}\cong D\times\mathbb{C}$ the strict transform of $D\times\mathbb{C}$ in this process.
\begin{figure}[h]\label{presolution}
\centering{
\includegraphics[height=2.8cm]{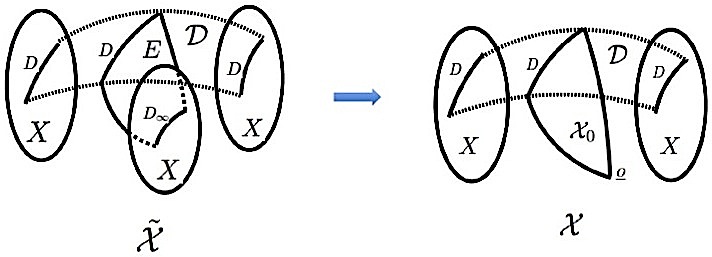}
\hskip 3mm
%\raisebox{3\height}{\includegraphics[width=0.7cm]{Slide3.jpg}}
\hskip 3mm
%\raisebox{0.1\height}{\includegraphics[height=2.5cm]{cX.pdf}}
\label{resolution}
\caption{$\mu: \tilde{\mX}\longrightarrow\mX$}
}
\end{figure}
From now on, we assume that $S=D$ is a smooth ample divisor in $X$. Then we can further modify the deformation to the normal cone construction. 
Recall that from the above section $\tilde{\mX}=Bl_{D\times\{0\}}(X\times\bC)$ and $\tilde{\mX}_0=(Bl_D X)\cup E=X\cup E$ where $E=\mathbb{P}(N_D\oplus\mathbb{C})$. Denote by $L:=L_D$ the holomorphic line bundle associated to the divisor $D$. Since $D$ is an ample divisor, one can verify that the line bundle $\tilde{\mL}=\pi_1^*L-E$ is $\pi_2$-relatively semi-ample where $\pi_1$ is the composition $\tilde{\mX}\rightarrow X\times \bC\rightarrow X$ and $\pi_2$ is the composition $\tilde{\mX}\rightarrow X\times \bC \rightarrow \bC$. Moreover, the strict transform of $X$ under the blow up becomes exceptional and can be blown down so that we get $\mX$ under the morphism associated to $\tilde{\mathcal{L}}$.
Then the canonical morphism $\pi: \mathcal{X}\rightarrow \mathbb{C}={\rm Spec} (\mathbb{C}[t])$ gives a flat family of projective varieties, satisfying that
$\cX_t\cong X$ for $t\neq 0$ and $\cX_0$ is obtained from $E$ by contracting the infinity section $D_\infty$. 
$\mX_0$ thus obtained is very close to being the projective cone $\ov{C}(D, L)$. One delicate point here is that $\mX_0$ may not be normal.

\begin{lem}\label{normalem}
The central fibre $\mX_0$ coincides with $\bar{C}(D, L)$ if the restriction map $\psi_m: H^0(X, mL)\rightarrow H^0(D, mL|_D)$ is surjective for any $m\ge 1$.
\end{lem}
\begin{proof}
We first describe the above construction of $\mX$ in the algebraic category (see \cite[Chapter 5]{Fult}). Let $\mathcal{I}_D$ denote the ideal sheaf of $D$ as a subvariety of $X$. Then $\tilde{\mX}$ is the blow up of the ideal sheaf $\mathcal{I}_D+(t)$ on $X\times \mathbb{C}$:
\[
\tilde{\mX}={\rm Proj}_{X\times \bC}\left( \bigoplus_{k=0}^{+\infty} (\mathcal{I}_D+(t))^k\right).
\]
%In other words, the natural map:
%\[
%\pi_2^* (\pi_2)_*(m\tilde{\mathcal{L}})\rightarrow m\tilde{\mathcal{L}},
%\]
%is surjective for $m\gg 1$. So 
Moreover $\mX={\rm Proj}_{\mathbb{C}[t]}\mathcal{R}$ where $\mathcal{R}$ is the following finitely generated graded algebra over $\mathbb{C}[t]$:
%The relative linear system $|m\tilde{\mL}|$ for $m\gg 1$ contracts the component $X$ of the central fibre $\tilde{\mX}_0$. More precisely, 
\begin{equation}\label{eq-mR}
\mathcal{R}=\bigoplus_{k=0}^{+\infty} H^0(\mathbb{C}, (\pi_2)_* (k \tilde{\mL}))=\bigoplus_{k=0}^{+\infty} H^0( \tilde{\mX}, k \tilde{\mL})=\bigoplus_{k=0}^{+\infty} \mathcal{R}_k,
\end{equation}
where $\tilde{\mL}=\pi_1^*L-E$.
The graded pieces $\mathcal{R}_k$ can be calculated in the same way as in \cite[Section 4]{RT06}:
\begin{eqnarray}\label{eq-mRk}
\mathcal{R}_k &=& H^0(\tilde{X}, k (\pi_1^*L-E))=H^0(X\times \bC, L^k\otimes (\mathcal{I}_D+(t))^k )\nonumber \\
&=&\bigoplus_{j=0}^{k-1} t^j H^0(X, L^k\otimes \mathcal{I}_D^{k-j}) \oplus t^k \mathbb{C}[t] H^0(X, L^k)
\nonumber\\
&=&\bigoplus_{j=0}^{k-1} t^j H^0(X, L^j)\oplus t^k \mathbb{C}[t] H^0(X, L^k).
\end{eqnarray}
In the last identity, we used $\mathcal{I}_D=\mathcal{O}_X(-D)\cong L^{-1}$. The central fibre is thus equal to:
\begin{equation}\label{eq-cX0}
\cX_0={\rm Proj}_{\bC}\left(\mathcal{R}/(t)\mathcal{R}\right)={\rm Proj}_{\bC}\left( \bigoplus_{k=0}^{+\infty} \mathcal{R}_k/(t)\mathcal{R}_k\right).
\end{equation}
From \eqref{eq-mRk}, we see directly that:
\begin{equation}\label{eq-cRk0}
\mathcal{R}_k/ (t)\mathcal{R}_k=\mathbb{C}\oplus \bigoplus_{j=1}^{k} t^j \frac{H^0(X, L^j)}{H^0(X, L^{j-1})}=\bigoplus_{j=0}^{k} t^j H^0(X, L^j)|_D.
\end{equation}
Here $H^0(X, L^j)|_D$ denotes the image of the restriction map $H^0(X, L^r)\rightarrow H^0(D, L^j|_D)$ for any $j\ge 0$. To see the the last identity, we 
consider the exact sequence of ideal sheaves:
\begin{equation}
0\longrightarrow L^{j-1}=L^j\otimes \mathcal{O}(-D) \longrightarrow L^j \longrightarrow L^j|_D \longrightarrow 0,
\end{equation}
and the corresponding long exact sequence:
\begin{equation}
0\longrightarrow H^0(X, L^{j-1})\longrightarrow H^0(X, L^j) \longrightarrow H^0(D, L^j|_D)\longrightarrow H^1(X, L^{j-1}).
\end{equation}
So indeed $H^0(X, L^j)/H^0(X, L^{j-1})=H^0(X, L^j)|_D$ for $j\ge 1$.

On the other hand, we have:
\[
\bar{C}(D, L)={\rm Proj}\bigoplus_{k=0}^\infty\left(\bigoplus_{j=0}^k  H^0(D, L^j|_D)\cdot t^{k-j}\right).
\]
Combining this with \eqref{eq-cX0} and \eqref{eq-cRk0}, we see that $\mX_0\cong \bar{C}(D, L)$ if $H^0(X, L^j)|_D=H^0(D, L^j|_D)$ for any $j\ge 1$ ($j=0$ case is automatic).
\end{proof}

 %H-J. Hein and C. Xu pointed to me the right condition ensuring the normality should be that:
%\begin{equation}\label{restmap}
%\phi_m: H^0(X, mL)\rightarrow H^0(D, mL)
%\end{equation}
%is surjective for any $m\ge 0$. 

For example, let $X$ be any Riemann surface of genus $\ge 1$ and $D=\{p\}$ be any point. Then $D$ is ample. In this special case, the central fibre $\mX_0$ is a singular curve whose normalization is $\mathbb{P}^1$.
Here the map $\psi_0=id: H^0(X, \mathcal{O}_X)\rightarrow H^0(p, \mathcal{O}_p)$. But $\psi_1=0: H^0(X, L_p)=\mathbb{C}\rightarrow H^0(\{p\},L_p|_{\{p\}})=\mathbb{C}$ because $\psi_1$ factors through the inverse of isomorphism $H^0(X, \mathcal{O}_X)=\mathbb{C}\stackrel{\cdot s_{\{p\}}}{\longrightarrow} H^0(X, L_p)$ by the assumption that $g(X)\ge 1$. In particular, $\psi_1$ is not surjective.
\begin{rem}
The above lemma was communicated to me by H-J. Hein. % after I told him the above Riemann surface example. 
One referee provided an even more explicit example to me: if $X$ is an elliptic curve and $p$ is a Weierstrass point, then by using the Weierstrass form, one can verify that the total space has a singularity of type $\tilde{E}_8$.
\end{rem}

On the other hand, from the exact sequence:
\[
0\rightarrow H^0(X, (m-1)L_D)\rightarrow H^0(X, mL_D)\rightarrow H^0(D, mL|_D)\rightarrow H^1(X, (m-1)L)\rightarrow \cdots,
\]
we see that $\psi_m$ is surjective if $H^1(X, (m-1)L)=0$ for all $m\ge 1$. In particular, this is satisfied in the Tian-Yau setting. Indeed, if $X$ is Fano and $m\ge 1$ then $H^1(X, (m-1)L)=H^1(X, \Omega_X^n\otimes\mathcal{O}_X(-K_X+(m-1)L))=0$ by the Nakano-Kodaira vanishing theorem. 

\subsection{Proof of Theorem \ref{thm-eqwt}}

From now on, we assume that we are in the situation that the above central fiber $\mX_0$, i.e. the strict transform of the exceptional divisor of the blow up, is normal and hence coincides with $\bar{C}(D, L)=:\bar{C}$. 
Let $\mD$ be the strict transform of $D\times\bC$ and $\mX^\circ=\mX\setminus \mD$. Because $\mD$ is a relatively ample divisor over $\bC$, we know that $\mX^\circ$ is a flat family of affine varieties.  In particular, we can define $\Ord_{\cX^\circ}$ as in Definition \ref{def-ord2}. Notice that $\cX^\circ_0=C(D, L)=:C$ and we can define the reduced Kodaira-Spencer class $\KS^\red_{\cX^\circ}\in \fT^1_C$. Since there is a natural $\bC^*$-action on $\fT^1_C$ , we can talk about the weight of $\KS^\red_{\cX^\circ}\in \fT^1_C$ and denote it by $w(\cX^\circ)$ (see Appendix \ref{App-cone}).
With these notations and combining the calculations from the previous subsection, we can derive the following
\begin{prop}\label{propamp}
Let $\cX\rightarrow \bB$ be the flat family constructed in the above section and assume $\cX_0=\bar{C}(D, N_D)$.
Let $D\hookrightarrow X$ be a $(k-1)$-comfortably embedded and $(\rho_{k-1}, \boldsymbol{\nu}_{k-1})$ be a $(k-1)$-comfortable pair. If $D$ is not $k$-splitting relative to $\rho_{k-1}$, then $\Ord(\cX^\circ)=k=-w(\cX^\circ)$. In particular, if $D$ is $(k-1)$-comfortably embedded and not $k$-splitting, then $\Ord(\cX^\circ)=k=-w(\cX^\circ)$.
\end{prop}
\begin{proof}

Let $(\rho_{k-1}, \boldsymbol{\nu}_{k-1})$ be a $(k-1)$-comfortable pair. Then by the proof of Proposition \ref{embvsdef}, we have a $(k-1)$-trivial atlas covering $\mW_0$. Without loss of generality, we can assume $\mW_0=\bar{C}\setminus K$ where $K$ is a strongly pseudo convex neighborhood of the vertex $o\in \bar{C}$. Then we also have a $(k-1)$-trivial atlas covering $\mW_0\setminus D=C\setminus K$. In particular, this atlas covers the annulus $Y=(C\setminus \overline{K}_{c_1})\cap \mathring{K}_{c_2}$. By Proposition \ref{prop-3eqwt}, we get $\Ord(\cX^\circ)\ge k$. 
Moreover from Proposition \ref{embvsdef}, we get a cohomology class $\theta_k\in H^1(L, \Theta_L)$ with weight $-k$, which is represented by a cocycle $\{(\theta_k)_{\beta\alpha}\}$. Proposition \ref{prop-3eqwt} yields that
\[
(\mu_1\circ \tau_U)(\KS^{(k)}_{\cX^\circ})=\theta_k|_Y
\]
where $Y$ again denotes the annulus and $\mu_1$, $\tau_U$ are from the diagram \eqref{eq-3morph}. But $\theta_k|_Y=\mu_1(\theta_k|_U)$. So thanks to the injectivity of both $\tau_U$ and $\mu_1$, we may reduce to proving that the class
$\vartheta_{k}:=\theta_k|_U\in H^1(U,\Theta_U)$ is not zero.

By Proposition \ref{embvsdef}, we know that $\mathfrak{T}_k(\theta_k)=\mathfrak{g}_k$ is the obstruction to $k$-splitting relative to $\rho_{k-1}$. So if the embedding is not $k$-splitting with respect to
$\rho_{k-1}$, then $\theta_k$ is non zero. Now the claim follows from the Lemma \ref{negeq}. 

\end{proof}

\begin{cor}\label{thmcor}
Assume ${\rm dim} D=n-1\ge 2$. If $D$ is $(k-1)$-comfortably embedded and not $k$-comfortably embedded, then the following holds:
\begin{enumerate}
\item
$\Ord(\cX^\circ)=k=-w(\cX^\circ)$, i.e. Theorem \ref{thm-eqwt} is true;
\item
For any $l< k$ and any $(l-1)$-th order lifting $\rho_{l-1}: \mathcal{O}_D\rightarrow \mathcal{O}_X/\mathcal{I}_D^{l}$, there exists a $(k-1)$-th order lifting $\rho_{k-1}: \mathcal{O}_D\rightarrow \mathcal{O}_X/\mathcal{I}_D^{k}$ such that $\phi_{k-1.l-1}(\rho_{k-1})=\rho_{l-1}$, where $\phi_{k-1,l-1}: \mathcal{O}_X/\mathcal{I}_D^{k}\rightarrow\mathcal{O}_X/\mathcal{I}_D^{l}$ is the natural map.
\end{enumerate}
\end{cor}
\begin{proof}
We first recall Remark \ref{psdiv}. If $\dim D\ge 2$ and $D$ is ample, $H^1(D, N_D\otimes\mathcal{I}_D^{k+1}/\mathcal{I}_D^{k+2})=H^1(D, L_D^{-k})=0$ for any $k\ge 1$ by Kodaira-Nakano vanishing. So there is no obstruction to $k$-comfortably embedded relative to any $k$-th order lifting. As a consequence, $k$-comfortable is equivalent to $k$-splitting for any $k\ge 0$, and is also equivalent to $k$-linearizable for all $k\ge 0$.

By the assumption, we know that $(X,D)$ is $(k-1)$-splitting but not $k$-splitting, and hence there exists a comfortable pair $(\rho_{k-1}, \boldsymbol{\nu}_{k-1})$ such that there is no $k$-th order lifting relative to $\rho_{k-1}$. So the first statement holds by Proposition \ref{propamp}. 

Suppose that for some $l<k$, there exists an $(l-1)$-th order lifting $\rho_{l-1}$ that can not be lifted to a $(k-1)$-order lifting. By choosing the maximal $l$ and using Remark \ref{psdiv}, we can assume there is a comfortable pair $(\rho_{l-1}, \boldsymbol{\nu}_{l-1})$ such that $\rho_{l-1}$ can not be lifted to an $l$-th order lifting. By Proposition \ref{propamp}, we get $w(X, D)=-l>-k$ which contradicts part 1.
\end{proof}
\begin{rem}
We will see in Proposition \ref{counterexample} that part 2 of the Corollary \ref{thmcor} is not necessarily true if $n=2$.
\end{rem}
\begin{lem}\label{negeq}
For $k\ge 1$, the natural restriction map induces an isomorphism $H^1(L, \Theta_L)(-k)\stackrel{\cong}{\rightarrow} H^1(U, \Theta_U)(-k)$. 
\end{lem}
\begin{proof}
This is already clear by the homogeneity argument as in the proof of Lemma \ref{basexact}. Indeed, we just need to construct an inverse of the natural morphism. 
Let 
$\theta_k\in H^1(U, \Theta_U)(-k)$. Then by the same argument as in the proof of Lemma \ref{basexact}, we can assume that $\theta_k$ is represented by a weight $(-k)$-cocyle:
\[
(\theta_k)_{\beta\alpha}=\sum_{r=1}^m b^r_{\beta\alpha}(w)\frac{\partial}{\partial w^r_\beta}+\sum_{p=m+1}^n c^p_{\beta\alpha}(w)\frac{\partial}{\partial w^p_\beta}.
\]
Since $\frac{\partial}{\partial w^r_\beta}$ (resp. $\frac{\partial}{\partial w^p_\beta}$) has weight $1$ (resp. $0$), we know that $b^r_{\beta\alpha}$ (resp. $c^p_{\beta\alpha}$) is homogeneous of degree $(k+1)$ (resp. $k$) in $w'=\{w^r_\beta\}$. Because $k\ge 1$, $\theta_k$ can be extended to become a cocycle $H^1(L, \Theta_L)(-k)$. This defines the inverse of the restriction morphism.
\end{proof}

\begin{rem}
We sketch a slightly more conceptual proof by using the Dolbeault cohomology.
%To prove these, we need to study $H^1(U, \Theta_U)(-k)$ and compare it with $H^1(L, \Theta_L)(-k)$. %By Andreotti-Grauert again, $\theta^{(k)}$ also gives a cocycle in $H^1(L, \Theta_L(-(j+1)D))$.
On the total space $L$, we have the exact sequence:
\begin{equation}
 0\rightarrow \pi_L^* L\rightarrow \Theta_{L}\rightarrow \pi_L^*\Theta_D\rightarrow 0.
\end{equation}
By restricting this exact sequence to $U=L\backslash D$, we have a similar exact sequence on $U$.
So we get commutative diagram of long exact sequences:
\begin{equation}\label{eq-CD46}
\xymatrix@R=0.8pc @C=0.8pc{
H^0(L, \pi_L^*\Theta_D)\ar[r]  \ar[d] & H^1(L, \pi_L^*L) \ar[d] \ar[r] & H^1(L, \Theta_L) \ar[r] \ar[d] & H^1(L, \pi_L^*\Theta_D) \ar[r] \ar[d] & H^2(L, \pi_L^*L) \ar[d]\\
H^0(U, \pi_U^*\Theta_D)\ar[r] & H^1(U, \pi_U^*L) \ar[r]         & H^1(U, \Theta_L) \ar[r]         & H^1(L, \pi_U^*\Theta_D) \ar[r] & H^2(U, \pi_U^*)
}
\end{equation}
%\[
%\dots\rightarrow H^1(L, \pi_L^*L)\rightarrow H^1(L, \Theta_L)\rightarrow H^1(L, \pi_L^*\Theta_D)\rightarrow H^2(L, \pi_L^*L)\rightarrow\cdots
%\]
For any $k\ge 0$, we have the weight-$(-k)$ pieces of the cohomology groups under the natural $\bC^*$-action:
%by using the projection formula (see \cite[Section 11]{Artin}) for $k\ge 1$ we have:
\[
H^p(L, \pi_L^*L)(-k)=H^p(D, L^{-k}), \; H^p(L, \pi_L^*\Theta_D)(-k)=H^p(D, \Theta_D\otimes L^{-k});
\]
\[
H^p(U, \pi_U^*L)(-k)=H^p(D, L^{-k}), \; H^p(U, \pi_U^*\Theta_D)(-k)=H^p(D, \Theta_D\otimes L^{-k}).
\] 
If we were to work in the algebraic category, the weight decomposition is directly obtained by using projection formula as in \cite[Section 11]{Artin}. Since we are working in the analytic category, we need to be more careful as we now explain.
Since the arguments to get the decompositions are the same, we just explain the first identity. Using the isomorphism between Dolbeault cohomology and sheaf cohomology, any cohomology class $\alpha\in H^p(L, \pi_L^*L)$ is representd by $\bar{\partial}$-closed $\pi_L^*L$-valued $(0,p)$-form denoted by $\eta$. For any point $p\in D$, we first choose local holomorphic coordinates $\{z^i, \xi\}$ where $\{z^i\}$ are holomorphic coordinates
on $D$ and $\xi$ is a linear coordinate along the fibre associated to a local trivializing holomorphic section $s$. By using the Fourier expansion along the circle $|\xi|={\rm constant}$ and extending to the whole $U$, one can show that % \footnote{I am grateful to Hans-Joachim Hein for helping with this.}
$\eta$ can be expressed as a convergent sum:
%\begin{eqnarray*}
%\eta&=&\sum_{I,m, |J|=p} a_I z^I \xi^m s d\bar{z}^J+\sum_{K, m', |L|=p-1} a_K z^K \xi^{m'} s d\bar{\xi} \wedge d\bar{z}^L\\
%&=&\sum_{I,m, |J|=p}a_I z^I \xi^m s d\bar{z}^J+\bar{\partial}\left(\sum_{K, m', |L|=p-1} a_K z^K \xi^{m'} s \bar{\xi} d\bar{z}^L  \right).
%\end{eqnarray*} 
\begin{eqnarray*}
\eta&=&\sum_{m\in \bN_0, |I|=p}[A'_{m,I}(z, |\xi|^2)\xi^m+A'_{\bar{m},I}(z, |\xi|^2)\bar{\xi}^m] s d\bar{z}^I\\
%&&+\sum_{m\in \bN_0, |J|=p-1,j}[B_{m,J}(z, |\xi|^2)\xi^{m-1}+B_{\bar{m},J}(z, |\xi|^2)\bar{\xi}^{m+1}] s d\bar{z}^j\wedge d\bar{z}^{J}\\
&&+\bar{\partial}\left(\sum_{m\in \bN_0, |J|=p-1}[B'_{m,J}(z, |\xi|^2)\xi^{m-1}+B'_{\bar{m},J}(z, |\xi|^2)\bar{\xi}^{m+1}]s d\bar{z}^J\right)\\
&=:&\eta'+\bar{\partial}\zeta.
\end{eqnarray*}
%Using the fact that $\eta$ is globally defined on the total space of $L$, it's easy to verify that $\zeta$ is a globally defined $\pi_L^*L$-valued $(0, p-1)$-form on $L$.
Furthermore, by using the fact that $\bar{\partial}\eta'=0$, one can see that $A'_{m,I}(z, |\xi|^2)$ and $A'_{\bar{m},I}(z, |\xi|^2)|\xi|^{2m}$ are constants in $\xi$. 
In particular, by smoothness, $A'_{\bar{m},I}=0$ for all $m\in \bN_0$. So we see that $\eta$ is $\bar{\partial}$-cohomologous to a $(0,p)$-form of the form:
\[
\eta'=\sum_{m\in \bN_0}\sum_{|I|=p} A'_{m,I}(z)\xi^m s d\bar{z}^I=: \sum_m \eta_m.
\]
$\mathbb{C}^*$-acts on $\eta'$ by
$t\circ \eta'=\sum_m t^{m-1} \eta'_m$. Using the $\bar{\partial}$-closedness of $\eta'$, it's easy to see that each component $\eta'_m$ is $\bar{\partial}$-closed. $\eta'$ is of weight $-k$ if and only if $\eta'=\eta'_{k+1}=\sum_{|I|=p}A'_{k+1,I}  \xi^{k+1} s d\bar{z}^I$ which represents a cohomology class $\alpha\in H^p(D, L^{-k})$.
%\[
%H^p(L, \pi_L^*L)=\bigoplus_{j=-1}^{+\infty} H^p(D, L^{-j}), \; H^p(L, \pi_L^*\Theta_D)=\bigoplus_{j=0}^{+\infty} H^p(D, \Theta_D\otimes L^{-j}).
%\]
%\[
%H^p(U, \pi_U^*L)=\bigoplus_{j=-\infty}^{+\infty} H^p(D, L^{-j}), \; H^p(U, \pi_U^*\Theta_D)=\bigoplus_{j=-\infty}^{+\infty} H^p(D, \Theta_D\otimes L^{-j}).
%\]
We can now extract the weight $(-k)$-part from \eqref{eq-CD46} to get exact sequences:
\begin{equation}
\xymatrix @R=0.8pc @C=0.8pc
{
H^0(D, \Theta_D\otimes L^{-k})\ar[r] \ar@{=}[d] & H^1(D, L^{-k}) \ar@{=}[d] \ar[r]^{\mathfrak{N}_k} & H^1(L, \Theta_L)(-k) \ar[r]^{\mathfrak{T}_k} \ar[d] & H^1(D, \Theta_D\otimes L^{-k}) \ar[r] \ar@{=}[d] & H^2(D, L^{-k}) \ar@{=}[d]\\
H^0(D, \Theta_D\otimes L^{-k})\ar[r] & H^1(D, L^{-k}) \ar[r]^{\mathfrak{N}_k^\circ}         & H^1(U, \Theta_U)(-k) \ar[r]^{\mathfrak{T}_k^\circ}         & H^1(D, \Theta_D\otimes L^{-k}) \ar[r] & H^2(D, L^{-k})
}
\end{equation}
The statement then follows from the 5-lemma. 
\end{rem}
\begin{rem}\label{remexacto}
If we rewrite the statement of Proposition \ref{embvsdef} by using the isomorphism of Lemma \ref{negeq}, then we get the exact sequence:
\begin{equation}\label{exacto}
H^1(D, L^{-k})\stackrel{\mathfrak{N}_k^\circ}{\longrightarrow} H^1(U, \Theta_U)(-k)\stackrel{\mathfrak{T}_k^{\circ}}{\longrightarrow} H^1(D, \Theta_D\otimes L^{-k})
\end{equation}
such that 1. $\mathfrak{T}_k^\circ(\vartheta^{(k)})=\mathfrak{g}_k$ is the obstruction to $k$-splitting; 2. If $\mathfrak{T}_k^\circ(\vartheta^{(k)})=0$, then there is a $k$-th order lifting $\rho_k$ and $\vartheta^{(k)}=\mathfrak{N}_k^\circ(\mathfrak{h}^{(k)})$ where $\mathfrak{h}^{(k)}$ is the obstruction to $k$-comfortably-embedding with respect to $\rho_k$.
\end{rem}
\subsection{2-dimensional examples and a remark on comfortable embedding}\label{comforem}

As mentioned in the introduction and recalled in Appendix \ref{linearizable}, Abate-Bracci-Tovena in \cite{ABT} gave a detailed study of various conditions of embedding: $k$-linearizable, $k$-splitting and $k$-comfortable embedding.  In order to talk about $k$-comfortable embedding, one needs to assume $k$-splitting (see Definition \ref{def-cft}). Under this assumption, we can study whether the embedding is comfortable with respect to any $k$-th order lifting.  
In \cite[Remark 3.4]{ABT}, the authors asked whether $k$-comfortable embedding with respect to one $k$-th order lifting implies $k$-comfortable embedding with respect to any other $k$-th order lifting. 
Here we give a simple example showing that the answer to this question is in general negative. 
%In Theorem \ref{thm-eqwt}, if ${\rm dim}D\ge 2$ (or equivalently ${\rm dim}X\ge 3$), then by Remark \ref{psdiv}, we know that, since $N_D$ is assumed to be ample, $D$ is $(m-1)$-comfortably embedded with respect to any $(m-1)$-splitting if and only if $D$ is $(m-1)$-splitting, if and only if $D$ is $(m-1)$-linearizable. However, if ${\rm dim}D=1$, we will see that these conditions have subtle differences.
%we will see that in general one can {\it not} replace the optimal order of comfortably embededness by that of linearizability in the {\it identity} $|w(X, D)|=m(X, D)$ which is the conclusion of Theorem \ref{thm-eqwt}. 

%In the last example of the last subsection, $M\cong X\backslash D$ with $X=\{\sum_{i=0}^{n+1}Z_i^2=0\}\subset \mathbb{P}^{n+1}$ and $D=X\cap \{Z_0=0\}$. By Theorem \ref{thm-eqwt}, $D$ is $1$-comfortably-embedded into $X$ but not $2$-comfortably-embedded into $X$. So if $n\ge 3$, then by Remark \ref{psdiv}, we know that $D$ is $1$-linearizable but not $2$-linearizable. 
\begin{prop}\label{counterexample}
The following is true for the diagonal embedding $D=\Delta(\mathbb{P}^1)\hookrightarrow X=\mathbb{P}^1\times\mathbb{P}^1$: 
\begin{enumerate}
\item[(i)] It is $k$-splitting for any $k\ge 1$.
\item[(ii)] The set of all 1st order liftings is parametrized by $\mathbb{C}$. So we can denote by $\rho_1^a$ the 1st order lifting corresponding to any $a\in\mathbb{C}$.
\item[(iii)] There exists a 2nd order lifting $\rho_2$ satisfying $\phi_{2,1}\circ\rho_2=\rho_1^a$ if and only if $a=0$.
\item[(iv)] The embedding is 1-comfortable with respect to $\rho_1^a$ if and only if $a=-1/2$.
\item[(v)] The embedding is 1-linearizable but not 2-linearizable. 
%\item[(vi)] The deformation of $\mathbb{P}^1\times\mathbb{P}^1$ to $\mathbb{P}(1,1,2)$ is of weight $-2$.
\end{enumerate}
\end{prop}
\begin{rem}
This diagonal embedding is $2$-splitting and $1$-comfortable, but the embedding is only $1$-linearizable. This does not contradict Theorem \ref{crilinear}, since the $1$-comfortable embedding is with respect to $\rho_{1}^{-1/2}$ which can not be lifted to a 2nd order lifting.
\end{rem}

\begin{proof}
Because there is a projection morphism onto the first factor $p_1: X=\mathbb{P}^1\times \mathbb{P}^1\rightarrow \mathbb{P}^1$, we see that there is a natural $k$-th order lifting $\rho_k: \mathcal{O}_D\rightarrow\mathcal{O}_X/\mathcal{I}_D^{k+1}$ given by 
$\phi_{\infty,k}\circ p_1^*\circ \Delta^*$, where $p_1^*: \mathcal{O}_D\rightarrow\mathcal{O}_X$ is the pull-back and $\phi_{\infty, k}: \mathcal{O}_X\rightarrow \mathcal{O}_X/\mathcal{I}_D^{k+1}$ is the natural quotient map. So the embedding is $k$-splitting for any $k\ge 1$. Since any embedding is 0-comfortable, we know that the embedding is 1-linearizable by Theorem \ref{crilinear}. So we get (i) and first half of (v).

We will quickly show that the the embedding is not comfortable with respect to the natural 1st order lifting $\rho_1$. We first construct an atlas near $D$. 
Choose the open covering of $\mathbb{P}^1\times\mathbb{P}^1$:
\[
\mathfrak{V}=\{U_i\times U_j; 1\le i,j\le 2\}.
\]
with (we denote $\mathbb{P}^1=\mathbb{C}\cup\{\infty\}$ with $|\infty|=+\infty$)
\[
U_1=\{z\in \mathbb{P}^1; |z|<2\}, U_2=\{z\in\mathbb{P}^1; |z|>1/2\}.
\]
Then $S=\Delta(\mathbb{P}^1)$ is covered by two open sets $\{V_i:=U_i\times U_i; i=1,2\}$, we define new coordinate functions by:
\[
\begin{array}{rcl}
V_{1}=\{(z, z')\in \mathbb{P}^1\times \mathbb{P}^1; |z|<2, |z'|<2\}&\rightarrow &\mathbb{C}^2\\
(z,z')&\mapsto & (y_1=z-z', z_1=z)\\
V_{2}=\{(z, z')\in \mathbb{P}^1\times\mathbb{P}^1; |z|>1/2, |z'|>1/2\}&\rightarrow &\mathbb{C}^2\\
(z,z')&\mapsto & (y_2=z^{-1}-z'^{-1}, z_2=z^{-1}).
\end{array}
\]
So we have $D\cap V_i=\{ y_i=0\}$. If $V'$ is a small neighborhood of $S=\Delta(\mathbb{P}^1)$ Then on the intersection $V_1\cap V_2\cap V'$, the transition functions are given by:
\begin{equation}\label{1to2tr}
y_2=-\frac{y_1}{z_1(z_1-y_1)}=-\frac{y_1}{z_1^2}-\frac{y_1^2}{z_1^3}+R_{3}, \quad z_2=z_1^{-1}.
\end{equation}
In the above expansion, we assume that $y_1$ is sufficiently small, and denote by $R_3$ a term $\in \mathcal{I}_D^{3}$. It's immediate to see that this atlas is adapted to the natural 1st order lifting $\rho_1$ where we have:
\[
\rho_1(z_1)=[z_1]_2 \mbox{ on } V_1\cap V', \quad \rho_1(z_2)=[z_2]_2 \mbox{ on } V_2\cap V'.
\]
The obstruction to 1-comfortable embedding is given by 
\begin{equation}\label{1cobs}
(\mathfrak{h}^{\rho_1}_1)_{21}=-\frac{[y_1^2]_3}{z_1^3}\frac{\partial}{\partial y_2}\in H^0(U_1\cap U_2, N_D\otimes\mathcal{I}_D^2/\mathcal{I}_D^3).
\end{equation}
Here we consider $\frac{\partial}{\partial y_2}$ and $\frac{\partial}{\partial y_1}$ as local generators of $N_D$, so that we have $\frac{\partial}{\partial y_2}=-z_1^2\frac{\partial}{\partial y_1}$ on $U_1\cap U_2$. We claim that $\mathfrak{h}_1^{\rho_1}$ represents a nonzero cohomology class in $H^1(D, N_D\otimes\mathcal{I}_D^2/\mathcal{I}_D^3)\cong H^1(\mathbb{P}^1, \mathcal{O}_{\mathbb{P}^1}(-2))=\mathbb{C}$. Otherwise, we can write:
\[
-\frac{[y_1^2]_3}{z_1^3}\frac{\partial}{\partial y_2}=a[y_1^2]_3\frac{\partial}{\partial y_1}-b[y_2^2]_3\frac{\partial}{\partial y_2} \mbox{ on } U_1\cap U_2
\]
where $a=a(z_1)$ is analytic in $z_1$ and $b=b(z_1^{-1})$ is analytic in $z_2=z_1^{-1}$. Using the change of coordinates, we arrive at an equation:
\[
-\frac{1}{z_1}=a(z_1)-\frac{b(z_1^{-1})}{z_1^2},
\]
which obviously has no solutions by looking at the Laurent expansion. So we get that $D\hookrightarrow X$ is not 1-comfortably embedded with respect to $\rho_1$.

Let's find all possible 1st order liftings, i.e. homomorphisms of sheaves of rings $\rho: \mathcal{O}_D=\mathcal{O}_X/\mathcal{I}_D\rightarrow \mathcal{O}_X/\mathcal{I}_D^2$ with $\phi_{1,0}\circ \rho=id$. On $U_1$, we can write $\rho(z_1)=[z_1+a(z_1)y_1]_2$ with $a(z_1)$ analytic in $z_1$
and $\rho(z_2)=[z_2+b(z_2)y_2]_2$ with $b(z_2)$ analytic in $z_2=z_1^{-1}$. Since $\rho$ is a homomorphism of sheaves of rings, we must have
\[
1=\rho(z_1 z_2)=[z_1z_2+a(z_1)z_1y_1+b(z_2)z_2y_2]_2=1+a(z_1)z_1[y_1]_2+b(z_2)z_2[y_2]_2 \mbox{ over } U_1\cap U_2.
\]
Since we have $[y_2]_2=-[y_1]_2z_1^{-2}$ by \eqref{1to2tr}, we get $(a(z_1)-b(z_2))z_1[y_1]_2=0$. So we must have that $a(z_1)=b(z_2)=a=$constant. Thus we get (ii). We will denote the corresponding
1st order lifting by $\rho_1^a$.

Now for any fixed 1st order lifting $\rho^a$, it's easy to find an atlas adapted to it. We simply need to make a coordinate change:
\begin{equation}\label{eq-cdchg}
\hat{z}_1=z_1+a y_1, \hat{y}_1=y_1 \mbox{ on } V_1; \quad \hat{z}_2=z_2+a y_2, \hat{y}_2=y_2 \mbox{ on } V_2. 
\end{equation}
We can calculate the new transition function:
\begin{equation}\label{eq-achange}
\hat{y}_2=-\frac{\hat{y}_1}{\hat{z}_1^2}-(2a+1)\frac{\hat{y}_1^2}{\hat{z}_1^3}+R_3, \quad \hat{z}_2=\hat{z}_1^{-1}-(a^2+a)\frac{\hat{y}_1^2}{\hat{z}_1^3}+R_3,
\end{equation}
where $R_3$ denotes terms in $\mathcal{I}_D^3$. So we see that the obstruction to 1-comfortable embedding with respect to $\rho_1^a$ is equal to $(2a+1)\mathfrak{h}_1^{\rho_1}$ (see \eqref{1cobs}). 
%\[
%\mathfrak{h}_1^{\rho_1}=\left\{(\mathfrak{h}_1^{\rho_1})_{21}=-\frac{[y_1^2]_3}{z_1^3}\frac{\partial}{\partial y_2}, (\mathfrak{h}_1^{\rho_1})_{12}=-(\mathfrak{h}^{\rho_1}_1)_{21}\right\}
%\]
From above we have seen that $H^1(D, N_D\otimes\mathcal{I}_D^2/\mathcal{I}_D^3)\cong \mathbb{C}$ is generated by $\mathfrak{h}_ 1^{\rho_1}$. So the embedding is comfortable with respect to $\rho_1^a$ if and only if $a=-1/2$. So we get (iii).

Furthermore, we can calculate the obstruction to existence of 2nd order lifting $\rho_2^a$ such that $\phi_{2,1}\circ \rho_2^a=\rho_1^a$:
\[
\left(\mathfrak{g}_2^{\rho_1^a}\right)_{21}=-a^2\frac{[y_1^2]_3}{z_1^3}\frac{\partial}{\partial z_2}\in H^0(U_1\cap U_2, \Theta_{D}\otimes\mathcal{I}_D^2/\mathcal{I}_D^3).
\]
By similar reasoning as before, we can see that $H^1(D, \Theta_D\otimes\mathcal{I}_D^2/\mathcal{I}_D^3)=H^1(\mathbb{P}^1, \Theta_{\mathbb{P}^1}\otimes \mathcal{O}_{\mathbb{P}^1}(-4))\cong \mathbb{C}$ is generated by the cohomology $\mathfrak{g}_2^{\rho_1^a}$ if and only if $a\neq 0$. So we get (iv). 

If the embedding is 2-linearizable, then it is 2-splitting and 1-comfortably with respect to the induced 1-splitting (see Theorem \ref{crilinear}). But from (ii)-(iv), we see that no such kind of 1-splitting exists. So we get second half of (v).

\end{proof}

\begin{rem}
By \eqref{eq-cdchg}, it's clear that the special value $a=-1/2$ corresponds to the (most) ``symmetric" coordinate atlas 
\begin{eqnarray*}
V_1\ni (z, z') &\mapsto& (z-z', \frac{1}{2}(z+z')=(\hat{y}_1, \hat{z}_1)\\
V_2\ni (z, z') &\mapsto& (z^{-1}-z'^{-1}, \frac{1}{2}(z^{-1}+z'^{-1}))=(\hat{y}_2, \hat{z}_2),
\end{eqnarray*}
for which the transition
functions are given by (see \eqref{eq-achange}):
\[
\hat{y}_2=-\frac{\hat{y}_1}{\hat{z}_1^2-\frac{1}{4}\hat{y}_1^2}=-\frac{\hat{y}_1}{\hat{z}_1^2}-\frac{1}{4}\frac{\hat{y}_1^3}{\hat{z}_1^4}+R_5, \quad \hat{z}_2=\frac{\hat{z}_1}{\hat{z}_1^2-\frac{1}{4}\hat{y}_1^2}=\frac{1}{\hat{z}_1}+\frac{1}{4}\frac{\hat{y}_1^2}{\hat{z}_1^3}+R_4.
\]
So this is indeed a $1$-comfortable atlas (see Theorem \ref{coordinate}).
\end{rem}

\begin{rem}
By Theorem \ref{ODmodule}, 1-comfortable embedding is equivalent to the splitting of the exact sequence:
\begin{equation}\label{exact2}
0\rightarrow \mathcal{I}_D^2/\mathcal{I}_D^3\rightarrow\mathcal{I}_D/\mathcal{I}_D^3\rightarrow \mathcal{I}_D/\mathcal{I}_D^2\rightarrow 0.
\end{equation}
This is a apriori sequence of sheaves of $\mathcal{O}_X/\mathcal{I}_D^2$-modules. $\mathcal{I}_D^2/\mathcal{I}_D^3$ and $\mathcal{I}_D/\mathcal{I}_D^2$ are natural $\mathcal{O}_D$-modules. 
%Fix any 1st order lifting $\rho^a$, $\mathcal{I}_D/\mathcal{I}_D^3$ becomes a sheaf of $\mathcal{O}_D$-module. The obstruction to $1$-comfortable embedding is the obstruction to the splitting of
%\eqref{exact2} as an exact sequence of sheaves of $\mathcal{O}_D$-modules. We can calculate the obstruction explicitly, which 
$\mathcal{I}_D/\mathcal{I}_D^3$ becomes a $\mathcal{O}_D$-module depending on the 1st order lifting (ring homomorphism) $\rho_1^a: \mathcal{O}_D\rightarrow \mathcal{O}_X/\mathcal{I}_D^2$. (iv) in Proposition \ref{counterexample} is equivalent to saying that \eqref{exact2} splits as an exact sequence of $\mathcal{O}_D$-modules thus obtained if and only if $a=-1/2$. This can also be verified directly using the expression: $\rho_1^a(z_1)=[z_1+a y_1]_2$ on $V_1$ and $\rho_1^a(z_2)=[z_2+a y_2]_2$ on $V_2$.
\end{rem}
\begin{rem}
If we denote by $w_i$ the fiber variables of $N_D$ satisfying $w_2=-z_1^{-2}w_1$, then using the notation in Lemma \ref{basexact}, we have: $\theta_1^a=\mathfrak{N}_1(\mathfrak{h}_1^{\rho_1^a})=0$ and $\mathfrak{T}_2(\mathfrak{\theta}_2^{-1/2})=\mathfrak{g}_2^{\rho_1^{-1/2}}\neq 0$, where
\[
(\theta_1^a)_{21}=-(2a+1)\frac{w_1^2}{z_1^3}\frac{\partial}{\partial w_2}=(2a+1)\left(\frac{1}{2}w_2\frac{\partial}{\partial z_2}-\frac{1}{2}w_1\frac{\partial}{\partial z_1}\right) \in H^0(\hat{U}_1\cap\hat{U}_2, \Theta_{N_D})(-1),
\]
and
\[
(\theta_2^{-1/2})_{21}=-\frac{1}{4}\frac{w_1^2}{z_1^3}\frac{\partial}{\partial z_2} \in H^0(\hat{U}_1\cap \hat{U}_2, \Theta_{N_D})(-2).
\]
\end{rem}
%\[
%N_D=H, -K_D=(n+1-2)H=(n-1)H \Longleftrightarrow N_D=H, -K_X=(n+2-2)H=nH.
%\]
%So $ \lambda=n^{-1}, \tau=\lambda^{-1}-1=n-1$ and $\beta=\frac{n-1}{n}$. 

%The intuitive reason for the existence of such AC Calabi-Yau metric is because that we have a conical Calabi-Yau metric on $N_D$ constructed by Calabi. on $N_D$ the metric has the explicit formula:
%Let $D=\{s_0=0\}$ where $s_0\in |D|=|-\lambda K_X|$. We can embed the degeration into $\mathbb{P}^N\times \mathbb{C}$ using $|mD|$ for $m$ sufficiently large. 
%$C=\bar{C}\backslash D$ is embedded into $\mathbb{C}^N=\mathbb{P}^N\backslash\{s_0^{\otimes m}=0\}$.
%\[
%H^i(D, L^{k})=H^i(D, K_D+(-K_D)+kL)=0, \mbox{ for } k\ge 0 \mbox{ and } i=1,2. 
%\]
%\[
%H^i(D, L^{-k})=H^{n-i}(D, K_D+L^{k})=0, \mbox{ for } k>0 \mbox{ and } n-i>0 (\Rightarrow n\ge 3).
%\]
%\[
%-K_D=(1-\lambda)H, L=\lambda H. 
%\]
%\[
%K_{\tC}=\pi^*K_{D}-D_\infty-D_0=p^*(K_\oC)+a D_0.
%\]
%\[
%-K_\oC=c D_\infty\Rightarrow -K_\oC|_{D_\infty}=c D_\infty|_{D_\infty}=c \lambda H\Rightarrow c=\lambda^{-1}.
%\]
%\[
%\oC\subset \mathbb{P}^N, \quad
%mD_\infty=\{S_\infty^{m}=0\}.
%\]
%$S_\infty^{m}\in |-m D_\infty|=|-m\lambda^{-1}H|$. 
%$S_\infty$ is the defining section of $D_{\infty}=\{S_\infty=0\}$. $D_\infty=\mathcal{O}_{\tC}(1)=\lambda^{-1}H$ is ample. 
%\[
%K_{L}=\pi^*K_{D}\otimes \pi^*L^{-1}=-\pi^*H.
%\]
%\[
%0\rightarrow\pi_L^*L\rightarrow\Theta_L\rightarrow \pi_L^*\Theta_D\rightarrow 0.
%\]
Notice that the central fiber of $\mX$ from the contracted deformation to the normal cone is $\bar{C}(\mathbb{P}^1, \mathcal{O}_{\mathbb{P}^1}(2))\cong \mathbb{P}(1,1,2)$. So by Proposition \ref{propamp}, we get the following corollary (See Example \ref{P1P1rate}).
\begin{cor}
The contracted deformation to the normal cone associated with $(\mathbb{P}^1\times\mathbb{P}^1, \Delta(\mathbb{P}^1))$ degenerates $\mathbb{P}^1\times\mathbb{P}^1$ to $\mathbb{P}(1,1,2)$. The weight of this deformation is $-2$. 
\end{cor}
Similarly we can deal with the case $D_2=\{Z_0^2+Z_1^2+Z_2^2=0\}\hookrightarrow X_2=\mathbb{P}^2$. For this, we notice that there is a 2-fold branched covering:
\[
\begin{array}{rcl}
p_2: \mathbb{P}^1\times\mathbb{P}^1&\longrightarrow& \mathbb{P}^2\\
([X_0, X_1], [Y_0, Y_1])& \mapsto & [X_0Y_0+X_1Y_1, \sqrt{-1}(X_0Y_0-X_1Y_1), \sqrt{-1}(X_0Y_1+X_1Y_0)].
\end{array}
\] 
The branch locus is exactly $\Delta(\mathbb{P}^1)$ with $p_2(\Delta(\mathbb{P}^1))=D_2$. Using this covering structure, it's easy to obtain two open sets $\{V_1, V_2\}$ covering $D_2$.
\[
\begin{array}{rcl}
V_1=(U_1\times U_1)/\mathbb{Z}^2 &\rightarrow& \mathbb{C}^2\\
(z, z') &\mapsto& \left(y_1=\frac{1}{4}(z-z')^2, z_1=\frac{1}{2}(z+z')\right)\\
V_2=(U_2\times U_2)/\mathbb{Z}^2 &\rightarrow& \mathbb{C}^2\\
(z, z') &\mapsto&\left(y_2=\frac{1}{4}(z^{-1}-z'^{-1})^2, z_2=\frac{1}{2}(z^{-1}+z'^{-1})\right)
\end{array}
\]
The transition function over $V_1\cap V_2$ is given by:
\[
y_2=\frac{y_1}{(z_1^2-y_1)^2}=\frac{y_1}{z_1^4}+\frac{2y_1^2}{z_1^6}+R_3, \quad z_2=\frac{z_1}{z_1^2-y_1}=\frac{1}{z_1}+\frac{y_1}{z_1^3}+R_2.
\]
So this atlas is a $0$-comfortable one. The associated $\theta_1\in H^1(D_2, N_{D_2})(-1)$ is represented by
\[
(\theta_1)_{21}=\frac{2w_1^2}{z_1^6}\frac{\partial}{\partial w_2}+\frac{w_1}{z_1^3}\frac{\partial}{\partial z_2} \in H^0(\hat{U}_1\cap \hat{U}_2, \Theta_{N_{D_2}})
\]
where $w_i$ are fiber variables of $N_{D_2}\cong \mathcal{O}_{\mathbb{P}^1}(4)$ satisfying $w_2=z_1^{-4}w_1$. So we have 
\[
(\mathfrak{g}_1)_{21}=\left(\mathfrak{T}_1(\theta_1)\right)_{21}=\frac{[w_1]_2}{z_1^3}\frac{\partial}{\partial z_2}\in H^0(U_1\cap U_2, \Theta_{D_2}\otimes\mathcal{I}_{D_2}/\mathcal{I}_{D_2}^2).
\]
In the \v{C}ech cohomology $\check{H}^1(\{U_1, U_2\}, \Theta_{D_2}\otimes\mathcal{I}_{D_2}/\mathcal{I}_{D_2}^2)$, any coboundary can be represented by
\[
a(z_1)[w_1]_2\frac{\partial}{\partial z_1}-b(z_2)[w_2]_2\frac{\partial}{\partial z_2}=\left(\frac{-a(z_1)}{z_1^2}-\frac{b(z_1^{-1})}{z_1^4}\right)[w_1]_2\frac{\partial}{\partial z_2}.
\]
Since $a(z_1)$ (resp. $b(z_1^{-1})$) is analytic in $z_1$ (resp. $z_1^{-1}$), the term in the bracket of the right hand side can not contain any $z_1^{-3}$-term.
So we see that $H^1(D_2, \Theta_{D_2}\otimes \mathcal{I}_{D_2}/\mathcal{I}_{D_2}^2)\cong H^1(\mathbb{P}^1, \mathcal{O}_{\mathbb{P}^1}(-2))\cong\mathbb{C}$ is generated by
$\mathfrak{g}_1\neq 0$. Because $\mathfrak{g}_1$ is the obstruction to 1-splitting (Proposition \ref{obsplit}), we obtain that the embedding is not even 1-splitting and hence not 1-linearizable. In this case, $\mX_0=\bar{C}(\mathbb{P}^1, \mathcal{O}_{\mathbb{P}^1}(4))\cong\mathbb{P}(1,1,4)$. So by Proposition 
\ref{propamp}, we obtain the following
result.
\begin{prop}\label{P2D2}
$D_2=\{Z_0^2+Z_1^2+Z_2^2=0\} \hookrightarrow \mathbb{P}^2$ is 0-linearizable. The contracted deformation to normal cone associated to $(\mathbb{P}^2, D_2)$ degenerates $\mathbb{P}^2$ to $\mathbb{P}(1,1,4)$. The deformation weight $w(X, D)$ is equal to $-1$. 
\end{prop}

\section{Applications to AC K\"{a}hler metrics}\label{sec-ACK}

In the first subsection, we explicitly compute the data of rotationally symmetric K\"{a}hler cone metrics on the affine cone. We also compare the norms with respect to smooth metric (living on the projective cone) and norms with respect to the cone metric near the infinity divisor. This allows us to get the estimate in Proposition \ref{prop-cftmap}. In the second subsection, we combine this estimate with Conlon-Hein's estimates in \eqref{CHestlambda} to get Corollary \ref{cor-estTY}. We then calculate several examples to illustrate our results. In particular, we can indeed recover numerical quantities in examples of \cite{CH1}.

\subsection{Proof of Proposition \ref{prop-cftmap}}\label{sec-conemetric}
%In the first subsection, we explicitly compute the data of rotationally symmetric K\"{a}hler cone metrics on the affine cone. We also compare the norms with respect to smooth metric (living on the projective cone) and norms with respect to the cone metric near the infinity divisor. This allows us to get the first estimate in Proposition \ref{prop-cftmap}. In the second subsection, we apply the previous construction to the Tian-Yau setting and prove the second estimate in Proposition \ref{prop-cftmap}. Combined with Conlon-Hein's estimates in \eqref{CHestlambda}, this allows us to get Corollary \ref{cor-estTY}. We then calculate several examples to illustrate our results. In particular, we recover numerical quantities in examples of \cite{CH1}.
First we review the K\"{a}hler cone metric on $C(D,L)$ given by the special Calabi ansatz $
\omega_{0}=\sddbar h^{\delta}$. %=\delta h^{\delta}\omega_D+\delta^2 h^{\delta}\frac{\nabla\zeta\wedge\overline{\nabla\zeta}}{|\zeta|^2}.
Then $\omega_0$ is a Riemannian cone metric on $C(D,L)$: 
\[
g=dr^2+r^2 g_Y,
\]
where $Y$ is the associated circle bundle over $D$. To see this, we consider the coordinate chart on $\mathbb{P}(L^{-1}\oplus \mathbb{C})$. Away from the infinity section $D_\infty$, we have coordinate chart given
by $(z,[\zeta_\alpha e_\alpha, 1])=(z, [e_\alpha, \zeta_\alpha^{-1}])=(z, [e_\alpha, \xi_\alpha])$.
Let %$s=\log(a|\zeta|^{-2})=\log(a|\xi|^2)=-\log h$, so that 
$h=|e_\alpha|_h^2|\zeta_\alpha|^2=a_{\alpha-}(z)|\zeta_\alpha |^2=(a_{\alpha+}(z)|\xi_\alpha|^2)^{-1}$. For simplicity, we will denote $\zeta=\zeta_\alpha$, $\xi=\xi_\alpha$, $a=a_{\alpha-}=a_{\alpha+}^{-1}$. Then we can calculate:
\begin{equation}\label{}
\omega_{0}=\sddbar h^{\delta}=\delta h^{\delta}\omega_D+\delta^2 h^{\delta}\frac{\nabla\zeta\wedge\overline{\nabla\zeta}}{|\zeta|^2}%\\
%&=&\lambda \delta e^{-\delta s} \omega_{KE}^D+\delta^2 e^{-\delta s} \frac{\nabla\xi\wedge\overline{\nabla\xi}}{|\xi|^2}\\
%&\sim& |\xi|^{-2\delta}\pi^*\omega_D+|\xi|^{-2(1+\delta)}\nabla\xi\wedge\overline{\nabla\xi}.
=\delta h^\delta\omega_D+\delta^2 h^{\delta}\frac{\nabla\xi\wedge \overline{\nabla\xi}}{|\xi|^2},
\end{equation}
where $\omega_D=\sddbar \log h$ is a smooth K\"{a}hler metric on $D$, and 
we have used vertical and horizontal frames:
\[
dz^i, \nabla \zeta=d\zeta+\zeta a^{-1}\partial a\quad \stackrel{dual}{\Longleftrightarrow}\quad \nabla_{z^i}=\frac{\partial}{\partial z^i}-a^{-1}\frac{\partial a}{\partial z^i}\zeta\frac{\partial}{\partial \zeta}, \frac{\partial}{\partial \zeta}.
\]
Under the $\{z,\xi\}$ coordinate, we have similarly:
\[
dz^i, \nabla\xi=d\xi-\xi a^{-1}\partial a=-\zeta^{-2}\nabla\zeta\quad \stackrel{dual}{\Longleftrightarrow} \quad \nabla_{z^i}=\frac{\partial}{\partial z^i}+ a^{-1}\frac{\partial a}{\partial z^i}\xi\frac{\partial}{\partial \xi}, \frac{\partial}{\partial\xi}=-\zeta^2\frac{\partial}{\partial\zeta}.
\]
To write the metric into a metric cone, we write $\zeta=\tilde{\rho} e^{i\theta}$. Then 
\[
\nabla\zeta=d\zeta+\zeta a^{-1}\partial a= e^{i\theta} (d\tilde{\rho}+i\tilde{\rho} d\theta+\tilde{\rho} a^{-1}\partial a)=e^{i\theta}(d\tilde{\rho}+i \tilde{\rho} (d\theta-i a^{-1}\partial a)).
\]
So if we let $r=h^{\delta/2}=(a(z)|\zeta|^2)^{\delta/2}$ and $\nabla\theta=d\theta-J a^{-1}d a$, then it's easy to verify that the corresponding metric tensor is given by:
\[
g_{\omega_0}=dr^2+ r^2 (\delta g_{\omega_D}+\delta^2 \nabla\theta\otimes \nabla\theta).
\]
Note that $\nabla\theta$ is nothing but the connection form on the unit $S^1$-bundle in $L^{-1}$. %$\tau=(1-\lambda)/\lambda\Rightarrow \delta=\frac{\lambda^{-1}-1}{n}$.
Now we compare the norm of tensors on $U=L\backslash D$ with respect to two metrics $\omega_0$ and $\tilde{\omega}_0$, where $\tilde{\omega}_0$ is any smooth K\"{a}hler metric on a neighborhood of $D$ in $L$. For example, we can take 
\[
\tilde{\omega}_0=\pi_L^*\omega_D+\epsilon \sddbar (a_{+}(z)|\xi|^2)
\] 
for small $\epsilon>0$. Suppose
$\Phi$ is a tensor of type $(p=p_h+p_v,q=q_h+q_v)$, i.e.
 \[
 \Phi\in (T_h^*X)^{\otimes p_h}\otimes (T_v^*X)^{\otimes p_v}\otimes (T_hX)^{\otimes q_h}\otimes (T_vX)^{\otimes q_v}.
 \]
 Then, by noticing $h^{\delta/2}\sim |\xi|^{-\delta}$, we have
\begin{equation}\label{difscaling}
\frac{|\Phi|_{\omega_0}}{|\Phi |_{\tilde{\omega}_0}}\sim |\xi|^{\delta p_h+ (\delta+1) p_v-\delta q_h-(\delta+1)q_v}.
\end{equation}
%Let $s=\log(a|\xi|^2)=-\log h$.
%\begin{eqnarray*}
%\omega_{CY}&=&\sddbar h^{\delta}=\delta h^{\delta}\tilde{\omega}+\delta^2 h^{\delta}\frac{\nabla\zeta\wedge\overline{\nabla\zeta}}{|\zeta|^2}\\
%&=&\lambda \delta e^{-\delta s} \omega_{KE}^D+\delta^2 e^{-\delta s} \frac{\nabla\xi\wedge\overline{\nabla\xi}}{|\xi|^2}\\
%&\sim& |\xi|^{-2\delta}\pi^*\tilde{\omega}+|\xi|^{-2(1+\delta)}\nabla\xi\wedge\overline{\nabla\xi}.
%\end{eqnarray*}
%\[
%dz^i, \nabla\xi=d\xi+\xi a^{-1}\partial a\Longleftrightarrow \nabla_{z^i}=\frac{\partial}{\partial z^i}- a^{-1}\frac{\partial a}{\partial z^i}\xi\frac{\partial}{\partial \xi}, \frac{\partial}{\partial\xi}.
%\]
%\begin{eqnarray*}
%\omega&=&\pi^*\tilde{\omega}+\sddbar P(s)=(1-\epsilon P_s)\pi^*\tilde{\omega}+P_{ss}\frac{\nabla\xi\wedge\overline{\nabla\xi}}{|\xi|^2}\\
%&=& f_1 \pi^*\tilde{\omega}+f_2\nabla\xi\wedge\overline{\nabla\xi}.
%\end{eqnarray*}
%$P(s)\sim e^{s}\Rightarrow P_s\sim e^{s}, P_{ss}\sim e^{s}\sim |\xi|^2$.
%\[
%\phi=\phi^{h}_{v}+\phi^{v}_{h}+\phi^{h}_{h}+\phi^{v}_v.
%\]
%$r^{-\eta}\sim |\xi|^{\delta\eta}$ ( $\delta\eta=-w$)
In particular, we get :
\begin{lem}\label{lem-tilomega}
If $\Phi$ is tensor of type $(1,1)$,  then 
\[
|\Phi^{h}_{v}|_{\omega_0}\sim |\Phi^{h}_v|_{\tilde{\omega}_0} |\xi|, \quad
%$|\xi|^{\delta\eta}\sim 
|\Phi^{v}_{h}|_{\omega_0}\sim |\Phi^{v}_h|_{\tilde{\omega}_0} |\xi|^{-1}, \quad
%$|\xi|^{\delta\eta}\sim 
|\Phi^{v}_{v}|_{\omega_0}\sim |\Phi^{v}_v|_{\tilde{\omega}_0}, \quad |\Phi^{h}_h|_{\omega_0}=|\Phi^{h}_h|_{\tilde{\omega}_0}.
\]
\end{lem}
As a consequence, under the assumption that the embedding $D\hookrightarrow X$ is $(k-1)$-comfortable, we combine Lemma \ref{lem-tilomega} with estimates \eqref{eq-improvePhi} to get:
\begin{equation}\label{eq-C0Phi}
|\Phi|_{\omega_0}\le C_0 |\xi|^{k}\sim C_0 r^{-\frac{k}{\delta}}. 
\end{equation}
%So if  $|\Phi|_{\omega_0}\sim |\xi|^{\eta}$, then we have
%\begin{equation}\label{blowuporder}
%|\Phi^{h}_v|_{\tilde{\omega}_0}\sim |\xi|^{\eta-1}, \quad
%$|\xi|^{\delta\eta}\sim 
%|\Phi^{v}_h|_{\tilde{\omega}_0}\sim |\xi|^{\eta+1}, \quad
%|\Phi^{v}_v|_{\tilde{\omega}_0}\sim |\xi|^{\eta}, \quad |\Phi^{h}_h|_{\tilde{\omega}_0}\sim |\xi|^{\eta}.
%\end{equation}

Next we compare the Christoffel symbols of the two metrics, which will be useful for converting the estimate of covariant derivatives with respect to $\omega_0$ to that with respect to $\tilde{\omega}_0$. See \eqref{covtopar}-\eqref{compdecay}. To simplify the calculation, we can choose the coordinate $\{z_\alpha^i\}$ on $D$ and holomorphic frame such that
\[
g^D_{i\bar{j}}(0)=\omega_D(\partial_{z_\alpha^i},\partial_{z_\alpha^j})(0)=\delta_{ij}, (\partial_{z_\alpha^k} g^D_{i\bar{j}})(0)=0; \quad (\partial_{z_\alpha^i} a)(0)=0, (\partial_{z_\alpha^i}\partial_{z_\alpha^j}a)(0)=0.
\]
Denote by the index 0 the coordinate corresponding to $\xi=\xi_\alpha$. Then the components of the metric tensor associated with $\omega_0$ are given by:
\[
g_{i\bar{j}}=\delta a^{\delta} |\xi|^{-2\delta}\delta_{ij}, g_{0\bar{0}}=\delta^2 a^{\delta}|\xi|^{-2(\delta+1)}, g_{0\bar{j}}=g_{j\bar{0}}=0. 
\]
So it's easy to calculate that:
\[
|dz_{\alpha}^i|_{\omega_0}=\delta^{-1/2}a^{-\delta/2}|\xi|^{\delta}\sim \frac{1}{|\xi|^{-\delta}},\quad |d\xi|_{\omega_0}=\delta^{-1}a^{-\delta/2}|\xi|^{(\delta+1)}\sim \frac{|\xi|}{|\xi|^{-\delta}}.
\]
\[
\Gamma_{ij}^k=\Gamma_{ij}^0=\Gamma_{i0}^0=\Gamma_{00}^i=0, \quad \Gamma_{i0}^j=-\frac{\delta}{\xi}\delta_{ij},\quad \Gamma_{00}^0=-\frac{\delta+1}{\xi}.
\]
In other words,
\[
\nabla \partial_{z_\alpha^i}=-\frac{\delta}{\xi}d\xi \otimes \partial_{z_\alpha^i}, \nabla \partial_\xi=-\frac{\delta+1}{\xi}dz_\alpha^i\otimes\partial_{z_\alpha^i}-\frac{\delta+1}{\xi}d\xi\otimes\partial_\xi.
\]
\[
\nabla dz_\alpha^i=-\frac{\delta}{\xi}(d\xi\otimes dz_\alpha^i+dz_\alpha^i\otimes d\xi),\quad \nabla d\xi=-\frac{\delta+1}{\xi}d\xi\otimes d\xi.
\]
So we see that
\begin{equation}\label{dernorm1}
|\nabla_{\omega_0} \partial_{z_\alpha^i}|_{\omega_0}\le C\sim \frac{|\partial_{z_\alpha^i}|_{\omega_0}}{|\xi|^{-\delta}}\sim \frac{|\partial_{z_\alpha^i}|_{\omega_0}}{r},\quad |\nabla_{\omega_0}\partial_{\xi}|_{\omega_0}\le C|\xi|^{-1}\sim \frac{|\partial_\xi|_{\omega_0}}{|\xi|^{-\delta}}\sim \frac{|\partial_{\xi}|_{\omega_0}}{r}.
\end{equation}
\begin{equation}\label{dernorm2}
|\nabla_{\omega_0} dz_\alpha^i|_{\omega_0}\le C|\xi|^{2\delta}\sim \frac{|dz_\alpha^i|_{\omega_0}}{|\xi|^{-\delta}}\sim \frac{|dz^i_\alpha|_{\omega_0}}{r},\quad |\nabla_{\omega_0} d\xi|_{\omega_0}\le C|\xi|^{1+2\delta}\sim \frac{|d\xi|_{\omega_0}}{|\xi|^{-\delta}}\sim \frac{|d\xi|_{\omega_0}}{r}.
\end{equation}
The above estimates imply that each time we take a covariant derivative with respect to $\omega_0$, we get an extra decay factor $|\xi|^{\delta}\sim r^{-1}$. So by induction which starts from \eqref{eq-C0Phi}, we get the wanted estimate in Proposition \ref{prop-cftmap}:
\begin{equation}
|\nabla^j_{\omega_0}\Phi|_{\omega_0}\le C_j |\xi|^{k+j\delta}\sim C_j r^{-\frac{k}{\delta}-j} \text{ for any } j\ge 0.
\end{equation}

\subsection{Asymptotical rates of Tian-Yau's Ricci-flat metrics}\label{sec-TY}
In this subsection, we explain how to get Corollary \ref{cor-estTY}. 
First we recall the Calabi-Yau cone metric on $C:=C(D, L)$ in the case when $K_D^{-1}=\mu L$ for $\mu>0$ and $D$ has a K\"{a}hler-Einstein metric $\omega_D=\omega_D^{\rm KE}$ such that 
$Ric(\omega_D^{{\rm KE}})=\mu\cdot \omega_D^{{\rm KE}}$. In this case, note that the Hermitian metric $h$ satisfies $\sddbar\log h=\omega_D^{{\rm KE}}$.
To find the Calabi-Yau cone metric, it's straightforward to calculate that:
\[
Ric(\omega_0)=-\sddbar\log\omega_0^n=(-n \delta+\mu)\pi_L^*\omega_D^{{\rm KE}}, 
\]
%\[
%\frac{d r}{d\tilde{\rho}}=\delta \tilde{\rho}^{\delta-1}\Rightarrow r=\tilde{\rho}^{\delta}.
%\]
where $n=\dim D+1$. So we get the exponent for the Calabi-Yau cone metric:
\begin{equation}\label{Calabicone}
-K_D=\mu N_D\Longrightarrow \delta=\frac{\mu}{\dim D+1}.
\end{equation}
Now assume that $X$ is a Fano manifold of dimension $n$ %$n\ge 3$ 
and $D$ is a smooth divisor such that $-K_X\sim \alpha D$ with $\mathbb{Q}\ni\alpha>1$. By adjunction formula, we get $-K_D=-K_X|_D-[D]=(\alpha-1)[D]=(1-\alpha^{-1})K_X^{-1}$ is still ample, and so $D$ is also a Fano manifold. Assuming that $D$ has a K\"{a}hler-Einstein metric, Tian-Yau \cite{TiYa} constructed an Asymptotical Conical (AC) Calabi-Yau K\"{a}hler metric $\omega_{\rm TY}$ on $X\backslash D$ whose metric tangent cone at infinity is the conical Calabi-Yau metric on $C(D,N_D)$ discussed above with the exponent $\delta=\frac{\alpha-1}{n}$. 
More precisely, there is a diffeomorphism $\phi_K: C(D, N_D)\backslash B_R(\underline{o})  \rightarrow (X\backslash D)\backslash K$ such that
\begin{equation}\label{eq-TYdecay}
\|\nabla_{\omega_0}^j(\phi_K^*(\omega_{\rm TY})-\omega_0)\|_{C^{0}}\le C r^{-\lambda-j} \mbox{ for } j\ge 0.
\end{equation}
Here $K$ is a compact set in the noncompact manifold $M:=X\backslash D$ and $B_R(\underline{o})$ is the ball of radius $R$ around the vertex $\underline{o}$ of the metric cone.  
 
 A natural problem is to determine the optimal order (i.e. the number $\lambda$ in \eqref{eq-TYdecay}) of such AC Calabi-Yau metric. This issue was studied in detail in Cheeger-Tian \cite{ChTi} and Conlon-Hein (\cite{CH1}, \cite{CH2}). 
Conlon-Hein \cite{CH1} studied the estimates on solutions to the corresponding complex Monge-Amp\`{e}re equation for Calabi-Yau metrics. If we denote by $\mathfrak{k}$ is K\"{a}hler class represented by $\omega_{\rm TY}$, then their estimate of the optimal rate is as follows (see \cite{CH1}, and \cite[Remark 1.2]{CH3}):
\begin{equation}\label{CHestlambda}
\lambda_{\max}\ge 
\left\{
\begin{array}{ll}
\min (2n, \lambda_1), &\mbox{ if } \mathfrak{k}\in H^2_c(M); \\
\min (2, \lambda_1), &\mbox{ if } \mathfrak{k}\in H^2(M).
\end{array}
\right.
\end{equation}
Here $\lambda_1$ is any number satisfying the following condition: there exists a diffeomorphism $F_K: C(D, N_D)\backslash B_R(\underline{o})\rightarrow M\backslash K$ such that 
\begin{equation}\label{pbO1}
\|\nabla_{\omega_0}^{j}(F_K^*\Omega-\Omega_0)\|_{\omega_{0}}\le C r^{-\lambda_1-j} \mbox{ for any } j\ge 0,
\end{equation}
where $\Omega$ (resp. $\Omega_0$) is the multi-valued meromorphic volume form on $X$ (resp. $\ov{C}(D, N_D)$) that is non-vanishing holomorphic on $M=X\backslash D$ (resp. $C(D, N_D)$) and has pole of order $\alpha$ along $D$. Conlon-Hein \cite{CH1} also showed that the condition \eqref{pbO1} is equivalent to the following condition:
 \begin{equation}\label{pbJ1}
\|\nabla_{\omega_0}^{j}(F_K^*J-J_0)\|_{\omega_{0}}\le C r^{-\lambda_1-j} \mbox{ for any } j\ge 0,
\end{equation}
where $J$ (resp. $J_0$) is the complex structure on $M$ (resp. $C(D, N_D)$). So we see that $\lambda_1$ essentially measures the difference between the complex
structure of $M\backslash K$ and $C(D, N_D)\backslash B_R(\underline{o})$. Equivalently we are indeed comparing the complex structure on the (punctured) neighborhood of $D$ inside $X$ and the complex structure of (punctured) neighborhood of $D$ inside $N_D$. 

Now assuming $D$ is $(k-1)$-comfortably embedded, the diffeomorphism from Proposition \ref{prop-cftmap} (constructed in Section \ref{consdiff}) satisfies \eqref{pbJ1} with $\lambda_1=\frac{k}{\delta}$.
By the above discussion, we indeed get Corollary \ref{cor-estTY} by using the estimates of Conlon-Hein.
 
\begin{exmp}\label{P1P1rate}
\begin{enumerate}
\item
$(X, D)\cong (\mathbb{P}^1\times\mathbb{P}^1, \Delta(\mathbb{P}^1))$. In this case, $\omega_{\rm TY}$ coincides with the Eguchi-Hanson metric. $\alpha=2$, $n=2$, $\delta=(\alpha-1)/n=1/2$. By Proposition \ref{counterexample}, $D$ is $1$-comfortably embedded (and $1$-linearizable) so that $k=2$. So $\lambda=\frac{k}{\delta}=4$.
\item
$(X, D)\cong (\mathbb{P}^2, \{Z_0^2+Z_1^2+Z_2^2=0\})\cong (\bP^1\times\bP^1, \Delta(\bP^1))/\mathbb{Z}_2$. In this case, $\omega_{\rm TY}$ is the Euguchi-Hanson metric$/\mathbb{Z}_2$. $\alpha=\frac{3}{2}$, $n=2$, $\delta=(\alpha-1)/n=1/4$. By Proposition \ref{P2D2}, $D$ is $0$-comfortably embeded (and $0$-linearizable) so that $k=1$. So $\lambda=\frac{k}{\delta}=4$.
\end{enumerate}
\end{exmp}

\begin{exmp}\label{calexmp}
We consider Pinkham's construction of sweeping out the cone, see \cite[page 46]{Pink}.
Assume $D^{n-1}\subset \mathbb{P}^{N-1}$ is a smooth complete intersection:
\[
D=\bigcap_{i=1}^m\{F_i(Z_1,\cdots, Z_N)=0\}\subset\mathbb{P}^{N-1},
\]
where $m=N-n$ and $F_i(Z_1,\cdots, Z_{N})$ is a (generic) homogeneous polynomial of degree $d_i>0$. Denote the affine cone over $D$ in $\mathbb{C}^N$ and projective cone over $D$ inside $\mathbb{P}^N$ by
\[
C(D,H)=\bigcap_{i=1}^m \{F_i(z_1,\cdots, z_N)=0\}\subset \mathbb{C}^{N}.
\] 
\[
\ov{C}(D,H)=\bigcap_{i=1}^m \{F_i(Z_1,\cdots,Z_N)\}\subset\mathbb{P}^N.
\]
Notice that since we have assumed that $D$ is a complete intersection, 
it's then known that $D$ is projectively normal in $\mathbb{P}^{N-1}$ which implies that its projective cone inside $\mathbb{P}^N$ is normal and hence coincides with its normalization $\bar{C}(D, H)$. 

Now assume $G_i(Z_0, Z_1, \cdots, Z_N)$ is a generic homogeneous polynomial of degree $e_i$ with $e_i<d_i$ for each $i=1, \cdots, m$. In particular $G_i(1, z_1, \cdots, z_N)$ a polynomial of degree $e_i$. We construct a degeneration:
\begin{equation}\label{eq-exmX}
\mX=\bigcap_{i=1}^m\{F_i(Z_1,\cdots, Z_{N})+(t Z_{0})^{d_i-\deg G_i}G_i(t Z_0, Z_1,\cdots, Z_{N}) =0\}\subset \mathbb{P}^N\times \mathbb{C}.
\end{equation}
By the ``generic" assumption, $X=\mX_1$ is smooth. This degenerates the variety $X=\mX_1\subset \mathbb{P}^N$ to $\ov{C}(D,H)$. 
In fact, $\mX$ is a degeneration of $\mX_1$ generated by the one parameter subgroup of projective transformations:
\[
[Z_0,Z_1,\cdots, Z_N]\rightarrow [t^{-1} Z_0, Z_1,\cdots, Z_N].
\]
Away from $\{Z_0=0\}$, we have the deformation of $C(D,H)$:
\begin{equation}\label{eq-eqci}
\mX^{\circ}=\bigcap_{i=1}^m \{F_i(z_1,\cdots, z_N)+t^{d_i-\deg G_i} G_i(t, z_1,\cdots, z_N)=0\}\subset \mathbb{C}^N\times\mathbb{C}.
\end{equation}
From the Digression \ref{dig-DNC}, the degeneration $\mX$ coincides with the family obtained by first blowing $D\times \{0\}$ and $X\times \bC$ and then blowing down the strict transform of $X\times\{0\}$ as in the introduction. 
%This can be readily seen from the graph construction for the deformation to the normal cone, which was recalled in Section \ref{secamp}. 
Now using the representation of $\cX$ in \eqref{eq-exmX},  we see that $\cX$ can be obtained by applying the above construction to the case $X=\cX_1$, $Y'=\bar{C}(X, H)\subset \bar{C}(\bP^{N}, H)=\bP^{N+1}$, and $D=\{Z_0=0\}\cap X$.
The coincidence of $\bar{C}(D, H)$ with the central fibre from the contracted deformation to the normal cone can also be verified directly by using Lemma \ref{normalem} and the projective normality of $D$. 
%Indeed, because $D$ is a complete intersection, it's well known that $D$ is projectively normal in $\mathbb{P}^{N-1}$ which means that its affine cone $\mX_0=\bigcap_{i=1}^m\{F_i(Z_1,\dots, Z_N)=0\}\subset \mathbb{C}^N$ is normal.
%Since the normal bundle $N_D$ coincides with the hyperplane bundle $H$, we know that $C(X, N_D)$ is normal. By our assumption of the numerical relation (i.e. $\alpha\ge 2$ as above) and Remark \ref{normal}, $\mX_0$ is exactly the central fiber sitting in the deformation to the normal cone construction.

By adjunction formula, we know that $-K_{\mX_1}=(N+1-\sum_{i=1}^m d_i)H$ and $-K_{D}=(N-\sum_{i=1}^m d_i)H$.  Consider the hyperplane section $D=\mathcal{D}_1=\mX_1\cap \{Z_0=0\}\subset\mX_1$. Then if we assume $\sum_{i=1}^m d_i\le  N-1$, we are in the above Tian-Yau's setting with $\alpha:=N+1-\sum_{i=1}^m d_i\ge 2$. 

%By Appendix \ref{appcn}, $\mathbf{T}^1_C$ can be calculated as a quotient ring. By \eqref{eq-eqci} and Example \ref{exmpclint}, we get 
%\[
%\Ord(\cX^\circ/\bC)=\min_{i=1}^m\{d_i-e_i\}=-w(\cX^\circ/\bC)=:-w.
%\]
%Without of loss of generality, we can assume $e_1>e_2>\cdots>e_m$. Under the morphism Then we have:
%\[
%\KS^{(d_1-e_1)}_{\cX^\circ/\bC}=[G_1] \in \fT^1_C(-(d_1-e_1)).
%\]
%\begin{rem}
%then the classifying map $\mathbf{I}_{\mX^{\circ}}$ defined in the introduction vanishes to the order $|w|$, and we see that the reduced Kodaira-Spencer class satisfies:
%\[
%{\bf KS}_{\mX^{\circ}}^{\rm red}=\left.\frac{1}{|w|!}\frac{d^{|w|}}{dt^{|w|}}\right|_{t=0}{\bf I}_{\mX^{\circ}}(t)={\scriptstyle \left.\left[\left\{\binom{d_i-e_i}{|w|}t^{d_i-e_i-|w|}G_i(1, z_1, \cdots, z_N)\right\}_{i=1}^m\right]\right|_{t=0}}=:[\mathcal{G}]\in \bigoplus_{i=1}^m {\bf T}^1(-(d_i-e_i)).
%\]
%\end{rem}

By Appendix \ref{appcn}, $\mathbf{T}^1_C$ can be calculated as a quotient ring. 
As in Example \ref{exmpclint}, consider the class
\[
[\mathcal{G}]:=\left[\sum_i G_i(1, z_1, \cdots, z_n)\right]\in \bigoplus_{i=1}^m \fT^1_C(-(d_i-e_i)),
\]
where $[\cdot]$ denotes the quotient morphism (see \eqref{eq-T1Cj}):
\[
H^0(U, N_U)\rightarrow \fT^1_C=\frac{H^0(U, N_U)}{H^0(U, \Theta_{\bC^N}|_U)}=\bigoplus_{j=-\infty}^{+\infty}\frac{\bigoplus_{i=1}^{m} H^0(D, (d_i+j) H)}{{\rm Jac}(H^0(D, (j+1)H)^{\oplus N})}.
\]
Notice the right-hand-side is actually finite dimensional (see \cite{Schl1, Artin}). Now if we assume that $[\mathcal{G}]$ in $\mathbf{T}^1_C$ is nonzero, then the reduced Kodaira-Spencer class $\KS^{\red}_{\cX^\circ}$ is the maximal weight piece of $[\mathcal{G}]$ 
and the weight of deformation $w(\cX^\circ)$ of ${\bf KS}^{\rm red}_{\mX^{\circ}/\bB}$ is equal to the weight of $[\mathcal{G}]$.

Without loss of generality we can assume $e_1>e_2>\cdots>e_m$ so that $\min_{i=1}^m\{d_i-e_i\}=d_1-e_1$. Then in general, $w:=w(\cX^\circ)\le -(d_1-e_1)$ which could be a strict inequality (see example item 3 of ordinary double point below). The equality holds if $[G_1]\neq 0\in \fT^1_C(-(d_1-e_1))$. 
%If $e_i\le d_i-2$, then we get $w=-\min_{i=1}^m\{d_i-e_i\}$ (see Example \ref{exmpclint} in Appendix \ref{appcn}). 
%Now, we assume $n\ge 3$. 
If we assume furthermore that $n\ge 3$, then by Theorem \ref{thm-eqwt}, we know that the divisor $D$ is $(|w|-1)$-comfortably embedded into $X$ (but not $|w|$-comfortably embedded into $X$).

So by the above calculation, we see that the asymptotic rate 
of holomorphic form is given by 
\[
\lambda=\frac{|w|}{\delta}=\frac{n |w|}{\alpha-1}.
\]
If furthermore $e_i\le d_i-2$, then
\[
\lambda=\frac{|w|}{\delta}=\frac{n |w|}{\alpha-1}=\frac{n\cdot \min_{i=1}^m\{ d_i-e_i\} }{N-\sum_{i=1}^m d_i}.
\]
In this way, we can indeed give an algebraic interpretation of the corresponding calculations in \cite{CH1}.
\begin{enumerate}
\item (\cite[Example 1]{CH1}). Smoothing of the cubic cone:
\[
C=\left\{z\in\mathbb{C}^4; \sum_{i=1}^{4}z_i^3=0 \right\} \leadsto M=\left\{z\in\mathbb{C}^4; \sum_{i=1}^4 z_i^3=\sum_{i,j} a_{ij} z_iz_j+\sum_{k} a_k z_k+\epsilon\right\}.
\]
where $a_{ij}$, $a_i$, $\epsilon$ are small (generic) constants. We have 
\[
\mathbf{T}_C^1=\frac{\mathbb{C}[z_1,\cdots, z_4]}{\langle z_1^2,\cdots, z_4^2\rangle}=\bigoplus_{\nu=-3}^{1}{\bf T}^1_C(\nu).
\]
With the earlier notation, $G(Z_0, \cdots, Z_4)=\sum_{i,j}a_{ij}Z_iZ_j+\sum_{k}a_k Z_k Z_0+\epsilon Z_0^2$ with 
\[
[\mathcal{G}]=[\sum_{ij}a_{ij} z_iz_j+\sum_k a_k z_k+\epsilon]\in
\mathbf{T}^1_C(-1)+\mathbf{T}^1_C(-2)+\mathbf{T}^1_C(-3).
\]
Note that we assume $a_{ij}$, $a_k$ are generic if they are not zero. So we get
\[
\def\arraystretch{1.4}
\begin{array}{l|c|c|c}
&\KS^\red_{\cX^\circ}& w(\cX^\circ) & \lambda \\ 
\hline
a_{ij}=a_k=0 & [\sum_{i,j} a_{ij}z_iz_j] & -3 & \frac{3\cdot 3}{4-3}=9 \\
\hline
a_{ij}=0, a_k\neq 0 & [\sum_k a_k z_k] & -2 & \frac{3\cdot 2}{4-3}=6\\
\hline
a_{ij}\neq 0 & [\epsilon] & -1 & \frac{3\cdot 1}{4-3}=3
\end{array}
\]

\item (\cite[Example 2]{CH1}). Smoothing of the complete intersection:
\[
C=\left\{ z\in\mathbb{C}^5; f_1=\sum_{i=1}^5 z_i^2=0, f_2=\sum_{i=1}^5\eta_i z_i^2=0\right\}\leadsto M=\left\{z\in\mathbb{C}^5; f_1(z)=f_2(z)=\epsilon\right\}.
\]
Here $\eta_i$ are distinct complex numbers. We have:
\[
\mathbf{T}^1_C=\frac{\mathbb{C}[z_1,\cdots,z_5]^{\oplus 2}}{{\rm Im}
\left(\begin{array}{ccc}
z_1&\cdots &z_5\\
\eta_1 z_1&\cdots&\eta_5 z_5
\end{array}\right)
}={\bf T}_C^1(-2).
\]
Because the images of $\mathcal{G}=(-\epsilon,-\epsilon)$ is not zero inside $\mathbf{T}^1_C$, we have $
\lambda=\frac{3\cdot 2}{5-2-2}=6.
$
\item (\cite[Example 3]{CH1}). Smoothing of the ordinary double point:
\[
C=\left\{z\in\mathbb{C}^{n+1}; \sum_{i=1}^{n+1}z_i^2=0 \right\} \leadsto M=\left\{z\in\mathbb{C}^{n+1}; \sum_{i=1}^{n+1} z_i^2=\sum_{i=1}^{n+1}a_i z_i +\epsilon\right\}.
\]
\[
\mathbf{T}^1_C=\frac{\mathbb{C}[z_1,\cdots,z_{n+1}]}{\langle z_1,\cdots, z_{n+1}\rangle}={\bf T}^1_C(-2).
\]
$G(Z_0,\cdots, Z_{n+1})=\sum_{i=1}^{n+1} a_i Z_i+ \epsilon Z_0$. So $[G(1,z_1,\cdots, z_n)]=[\sum_{i=1}^{n+1}a_iz_i+\epsilon]=[\epsilon]$ is of weight $-2$. So we have $
\lambda=\frac{n\cdot 2}{n+1-2}=\frac{2n}{n-1}$. 

Note that if $n=2$, then $D\hookrightarrow X$ is isomorphic to $\Delta(\mathbb{P}^1)\hookrightarrow \mathbb{P}^1\times \mathbb{P}^1$ where $\Delta: \mathbb{P}^1\rightarrow \mathbb{P}^1\times\mathbb{P}^1$ is the diagonal embedding which was studied in Section \ref{comforem}. The identification is easily constructed:
\begin{eqnarray*}
(\mathbb{P}^1\times\mathbb{P}^1, \Delta(\mathbb{P}^1))&\longrightarrow& (X, D)=\left(\{Z_0^2+Z_1^2+Z_2^2+Z_3^2=0\}, \{Z_0=0\}\cap X\right)\\
([X_0,X_1], [Y_0, Y_1])&\mapsto& \scriptstyle{[X_0Y_1-X_1Y_0, \sqrt{-1}(X_0Y_1+X_1Y_0), (X_0Y_0+X_1Y_1),\sqrt{-1}(X_0Y_0-X_1Y_1)]}.
\end{eqnarray*}

\end{enumerate}
%Assume $F(z_1,\cdots, z_n)$ is a homogeneous degree 
%\[
%\mathbb{C}^{n+1}\supset\left\{\sum_{i=1}^{n+1} z_i^2=t\right\}\subset \left\{\sum_{i=1}^{n+1} Z_i^2=t Z_0^2\right\}\subset\mathbb{P}^{n+1}.
%\]
\end{exmp}

\begin{dig}\label{dig-DNC}
Here we recall an equivalent description of deformation to normal cone by using MacPherson's graph construction. Let $s_D$ denote the canonical holomorphic section of $L=L_D$ with $D=\{s_D=0\}$. We can identify $X$ with the graph of $s_D$ as a subvariety of $Y=\mathbb{P}(L\oplus \mathbb{C})$: $\mathcal{X}_1=\{(p, [s_D(p), 1]); p\in X\}$. We then use the natural $\mathbb{C}^*$-action on $Y$ to get a family of subvarieties of $Y$: $\mathcal{X}_t=\{p, [t^{-1} s_D(p), 1]; p\in X\}$. For $t\neq 0$, $\mX_t\cong X$. As $t\rightarrow 0$, $\mX_t$ converges to a subscheme $\tilde{\mX}_0$ of $Y$ which is nothing but the union of 
$X$ with $E$. Alternatively, there is a rational map
\begin{eqnarray*}
\Psi: X\times \bC \dashrightarrow \bP(L\oplus \bC), \quad (p, t)\mapsto (p, [t^{-1} s(p), 1])=(p, [s(p), t]).
\end{eqnarray*}
Notice the indeterminacy locus of $\Psi$ is exactly $D\times\{0\}=\{s=0\}\times \{0\}$. So $\tilde{\cX}=Bl_{S\times\{0\}}(X\times\bC)$ is the graph $\Gamma_\Psi$ of $\Psi$, i.e. the closure of the graph of $\Psi:( X\times \bC)\setminus (D\times \{0\})\rightarrow \bP(L\oplus\bC)$. 

Figure \ref{DNC} is an illustration of deformation to the normal cone using the graph construction ($S=D$). Notice that the two pairs of opposite sides of the boundary in the figure are glued according to the direction of arrows and the total space $\tilde{\cX}$ should be taken as the disjoint union of $\cX_t$ in the figure. See also \cite[Remark 5.1.1, Section 5.1]{Fult}. 
\begin{figure}[h]
\centering
\includegraphics[height=4.2cm]{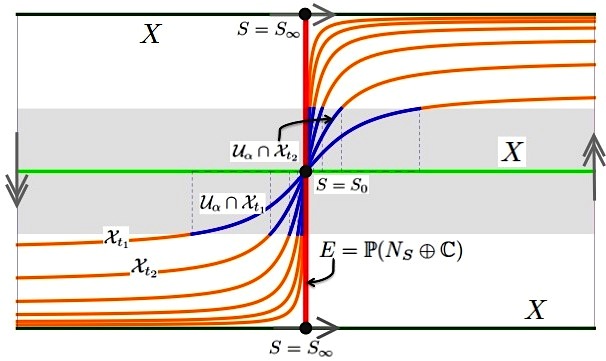}
\caption{Deformation to the normal cone: graph construction}\label{DNC}
\end{figure}
To get $\cX$ from $\tilde{\cX}$, we can use the similar construction, just by replacing $Y=\mathbb{P}(L\oplus\mathbb{C})$ by the projective cone $\bar{C}(X, L)=:Y'$ which is obtained from $Y$ by contracting the infinity divisor $X_\infty$. %Similar discussion already appeared in \cite[Proof of Proposition 5.1]{CH1}.

\end{dig}

\begin{rem} 
As pointed out by the referee, for the above examples of complete intersections the result of Theorem \ref{thm-eqwt} may not be surprising since we have explicit expressions:
\[
\mX_1=\bigcap_{i=1}^m \{F_i(Z_1,\cdots, Z_N)+Z_0^{d_i-e_i} G_i(Z_0, Z_1, \cdots, Z_N)=0\}\subset\mathbb{P}^N.
\] 
Noting that $|w|=\min_{i=1}^m\{d_i-e_i\}$, it's immediate that 
\[
\mathcal{O}_{\mX_1}/\mathcal{I}_D^{|w|}\cong\mathcal{O}_{\mX_0}/\mathcal{I}_D^{|w|},
\]
using the fact that $\mathcal{I}_D\left(U_{\{Z_i\neq 0\}\cap \mX_1}\right)=\left(\left\langle \frac{Z_0}{Z_i}\right\rangle+\mathcal{I}_{\mX_1}\right)/\mathcal{I}_{\mX_1}$. In other words, $(\mX_1, D)$ is $(|w|-1)$-linearizable. Then by Remark \ref{psdiv}, when $n\ge 3$, we know that $D$ is $(|w|-1)$-comfortably embedded. So we get $m(X, D)\ge |w|$. Note that the conclusion in Theorem \ref{thm-eqwt} is however stronger, saying that this is an equality for the more general case without using such explicit defining equations.
\end{rem}

\section{Analytic compactification }

In this section, we will prove Theorem \ref{thm-anacpt}. We will first sketch a proof following the strategy of the classical work of Newlander-Nirenberg in \cite{NeNi} that is modified to adapt to the setting of weighted spaces. Then we will write down the detailed estimates by imitating the corresponding estimates in \cite{NeNi}.

\subsection{Reduction of Theorem \ref{thm-anacpt} to Proposition \ref{prop-coordinates}}\label{sec-thmred}
We refer to section \ref{sec-conemetric} for the background. Denote $U=L\backslash D$. Denote the standard complex structure on $U$ by $J_0$. Assume that we have a complex structure $J$ on some neighborhood $U_\epsilon$ of $D$. Denote $\Phi=J-J_0$. %As in section \ref{consdiff}, 
We assume the index $v\in \{1,\overline{1}\}$ associates to the fiber variable $\xi=z^1_\alpha$, $h\in \{2,\dots, n, \overline{2}, \cdots, \overline{n}\}$ associates to the base variables $\{z^2_\alpha, \cdots, z^{n}_\alpha\}$. By abuse of notations, we decompose $\Phi$ into four types of components:
\begin{equation}\label{eq-J-J0}
\Phi=\Phi_{v}^{h}+\Phi_{h}^{v}+\Phi_{v}^{v}+\Phi_{h}^{h}=\phi_{v}^h d z^v\otimes \partial_{z^h}+\phi_{h}^{v}dz^h\otimes \partial_{z^v}+\phi_{v}^v dz^v\otimes \partial_{z^v}+\phi_{h}^{h}dz^h\otimes\partial_{z^h}.
\end{equation}
We assume $\Phi$ satisfies $|\nabla^j\Phi |_{\omega_0}\le C |r|^{-\lambda-j}\sim |\xi|^{\delta(\lambda+j)}$.  We first need to transform this estimate to the corresponding estimate with respect to $\tilde{\omega}_0$. For this, note that we know the basic tensors satisfy \eqref{dernorm1} and \eqref{dernorm2}. So we can equivalently assume $\Phi$ satisfies:
\begin{equation}\label{covtopar}
|(\partial_{z^v}^{j_1}\partial_{z^h}^{j_2}\Phi)\otimes (dz^v)^{\otimes j_1}\otimes (dz^h)^{\otimes j_2}|_{\omega_0}\le C|r|^{-\lambda-j}=C|\xi|^{\delta(\lambda+j)}.
\end{equation}
Recall the norm in Section \ref{sec-conemetric}:
\[
|dz^v|_{\omega_0}\le C|\xi|^{\delta+1}, |dz^h|_{\omega_0}\le C|\xi|^{\delta}\Longrightarrow |(dz^v)^{\otimes j_1}\otimes dz^h)^{\otimes j_2}|_{\omega_0}\le |\xi|^{j_1(\delta+1)+j_2 \delta}=|\xi|^{\delta j+j_1}.
\]
Also we have:
\[
|dz^v\otimes \partial_{z^h}|_{\omega_0}\le C|\xi|, |dz^h\otimes \partial_{z^v}|_{\omega_0}\le C|\xi|^{-1}, |dz^v\otimes \partial_{z^v}|_{\omega_0}\le C, |dz^h\otimes \partial_{z^h}|_{\omega_0}\le C.
\]
By these inequalities, it's easy to see that:
\begin{equation}\label{compdecay}
|\partial_{z^v}^{j_1}\partial_{z^h}^{j_2}\phi_v^h|\lesssim |\xi|^{\lambda \delta-1-j_1}, |\partial_{z^v}^{j_1}\partial_{z^h}^{j_2}\phi_h^v|\lesssim |\xi|^{\lambda\delta+1-j_1}, 
|\partial_{z^v}^{j_1}\partial_{z^h}^{j_2}\phi_v^v|\lesssim |\xi|^{\lambda \delta-j_1}, |\partial_{z^v}^{j_1}\partial_{z^h}^{j_2}\phi_h^h|\lesssim |\xi|^{\lambda\delta-j_1}.
\end{equation}
%|\partial_{z^v}^{j_1}\partial_{z^h}^{j_2}\phi_{h}^v|, |\partial_{z^v}^{j_1}\partial_{z^h}^{j_2}\phi_h^v|, |\partial_{z^v}^{j_1}\partial_{z^h}^{j_2}\phi_v^v|

\begin{prop}\label{prop-coordinates}
Fix $\eta\in \mathbb{R}_{>0}\setminus\mathbb{N}$. Let $J_0$ denote the standard complex structure on $\bB^*\times \bB^{n-1}$. Assume that $J$ is an integrable almost complex structure on $\bB^*\times \bB^{n-1}$ and the tensor $\Phi=J-J_0$ is decomposed into four types of components:
\begin{equation}\label{eq-J-J0}
\Phi=J-J_0=\Phi_{v}^{h}+\Phi_{h}^{v}+\Phi_{v}^{v}+\Phi_{h}^{h}=\phi_{v}^h d z^v\otimes \partial_{z^h}+\phi_{h}^{v}dz^h\otimes \partial_{z^v}+\phi_{v}^v dz^v\otimes \partial_{z^v}+\phi_{h}^{h}dz^h\otimes\partial_{z^h},
\end{equation}
where the index $v\in \{1,\overline{1}\}$ is associated to the first variable $z^1$, and $h\in \{2,\dots, n, \overline{2}, \cdots, \overline{n}\}$ is associated to the variables $\{z^2, \cdots, z^{n}\}$.
%For any number $\nu\in (0,+\infty)\backslash \mathbb{N}$ with $\nu<\eta$ and $m:=\lceil \nu\rceil=\lceil \eta\rceil$, 
Assume that there exists a constant $C$ such that
for any $j_1+j_2\le  2n+1$ %, there exists a uniform constant $C=C_{j_1, j_2}$ independent of
and all $(z_1, z_2, \cdots, z_n)\in \bB^*\times\bB^{n-1}$ it holds:
\begin{equation}\label{eq-Phidecay}
|\partial_{z^v}^{j_1}\partial_{z^h}^{j_2}\phi_v^h|\le C |\xi|^{\eta-1-j_1}, |\partial_{z^v}^{j_1}\partial_{z^h}^{j_2}\phi_h^v|\le C |\xi|^{\eta+1-j_1}, 
|\partial_{z^v}^{j_1}\partial_{z^h}^{j_2}\phi_v^v|\le C |\xi|^{\eta-j_1}, |\partial_{z^v}^{j_1}\partial_{z^h}^{j_2}\phi_h^h|\le C |\xi|^{\eta-j_1}.
\end{equation}
Denote $m=\lceil \eta\rceil$. Then for sufficiently small $R>0$, there exist $J$-holomorphic coordinates $\zeta=(\zeta_1, \zeta_2, \cdots, \zeta_n): \bB^*_R\times \bB^{n-1}_R\rightarrow \bB^*_{2R}\times \bB^{n-1}_{2R}$ and a constant $C'$
such that for any $j_1+j_2\le 2n+1$ it holds:
\[
\left|\partial^{j_1}_{z^v}\partial^{j_2}_{z^h}\left( \zeta^1-z^1(\zeta)\right)\right|\le C' |\zeta^1|^{m+1-j_1}; \quad \left|\partial_{z^v}^{j_1}\partial_{z^h}^{j_2}(\zeta^k-z^k(\zeta))\right|\le C' |\zeta^1|^{m-j_1}, 2\le k\le n.
\]

\end{prop}

\begin{rem}\label{cmpHHN}
%In \cite{HHN}, the authors proved an analytic compactification result in the asymptotically cylindrical Calabi-Yau case. Their compactification result depends on classifying asymptotical models of the asymptotically cylindrical Calabi-Yau metrics. In the asymptotically conical case, the classification of models at infinity is not clear at present. So here we just concentrate in a case when the model at infinity is known. In this sense, t
The result obtained here is a counterpart of \cite[Theorem 3.1]{HHN} in our different asymptotically conical setting.  
%More precisely, the proof in \cite{HHN} consists several steps. 
In the proof of \cite[Theorem 3.1]{HHN}, the authors used gauge fixing and used result of Nijenhuis-Woolf \cite{NiWo}. See \cite{CH3} for a different proof following similar argument as in \cite{HHN}.
%The Gauge Fixing sketched in Step 1 in \cite[Section 3.2]{HHN} corresponds in our setup to solving the 1st and 3rd equation in the system \eqref{zzetaeqcomp}. 
We aim to give a more direct proof by following the fundamental work of Newlander-Nirenberg. One should also be able to adapt the work of Nijenhuis-Woolf \cite{NiWo}, Malgrange \cite{Malg} to the current setting to prove the compactification (extension) of the complex structures considered here. \end{rem}

\begin{rem}\label{rem-NNreg}
If we assume $\eta>1$, then the existence of such coordinates follows from the work of \cite{HiTa}. However even in this case, Proposition \ref{prop-coordinates} provides more information (weighted estimates), which is needed to read out the embedding order of the divisor at infinity. 
\end{rem}

In the remainder of Section \ref{sec-thmred}, we will sketch the proof of Proposition \ref{prop-coordinates} and show how Theorem \ref{thm-anacpt} follows from it. Section \ref{sec-NNproof} contains the technical details of the proof of Proposition \ref{prop-coordinates}.

The $(0,1)$ vector under the new complex structure $J$ is given by 
\[
\frac{1}{2}(1+\sqrt{-1} J)\frac{\partial}{\partial \ov{z}^i}=\frac{\partial }{\partial \ov{z}^i}+\frac{\sqrt{-1}}{2}\phi_{\ov{i}}^{\bar{j}}\frac{\partial}{\partial\ov{z}^j}+\frac{\sqrt{-1}}{2}\phi_{\ov{i}}^{k}\frac{\partial}{\partial z^k}.
\]
Denote $\eta=\lambda\delta$ and $\rho=|\xi|=|z^1|$. Then from \eqref{compdecay}, we can write:
\begin{equation}\label{phiord}
\left(\phi_{\ov{i}}^{\ov{j}}\right)= \left(
\begin{array}{cc}
O(\rho^{\eta})_{1\times 1}& O(\rho^{\eta+1})_{1\times (n-1)}\\
O(\rho^{\eta-1})_{(n-1)\times 1}& O(\rho^{\eta})_{(n-1)\times (n-1)}
\end{array}
\right).
\end{equation}
We have the same order estimates for $\left(\phi^k_{\ov{i}}\right)$.
When $\rho$ is sufficiently small, the matrix $\left(\delta_{\ov{i}}^{\ov{j}}+\frac{\sqrt{-1}}{2}\phi^{\ov{j}}_{\ov{i}}\right)$ is invertible. It's easy to get the order estimates: 
\begin{equation}\label{invertible}
\left(a^{k}_{\ov{i}}\right):=-\left(\delta_{\ov{i}}^{\ov{j}}+\frac{\sqrt{-1}}{2}\phi^{\ov{j}}_{\ov{i}}\right)^{-1}\left(\frac{\sqrt{-1}}{2}\phi^{k}_{\ov{j}}\right)= \left(
\begin{array}{cc}
O(\rho^{\eta})_{1\times 1}& O(\rho^{\eta+1})_{1\times (n-1)}\\
O(\rho^{\eta-1})_{(n-1)\times 1}& O(\rho^{\eta})_{(n-1)\times (n-1)}
\end{array}
\right).
\end{equation}

To get an analytic compactification of the complex structure $J$, we want to solve for a map $z: \bB_R^n\rightarrow \bB_{2R}^n\subset\mathbb{C}^n$ where $\bB_R^n=\{(\zeta^1,\cdots, \zeta^n)\in\mathbb{C}^n; |\zeta^j|\le R\}$, such that $z$ is a homeomorphism onto the image and is holomorphic with respect to $J_0$ and $J$.  For the map $z$ to be holomorphic, $dz(\partial/\partial \bar{\zeta}^l)$ should be
a $(0,1)$-vector for any $l\ge 1$. It's easy to see that $z^i=z^i(\zeta)$ must solve the following equations:
%\[
%(J_0+\psi)\left(\frac{\partial z^i}{\partial\overline{\zeta}^j}\frac{\partial}{\partial z^i}+\frac{\partial\overline{z}^i}{\partial\overline{\zeta}^j}\frac{\partial}{\partial\overline{z}^j}\right)=-\sqrt{-1}\left(\frac{\partial z^i}{\partial\overline{\zeta}^j}\frac{\partial}{\partial z^i}+\frac{\partial\overline{z}^i}{\partial\overline{\zeta}^j}\frac{\partial}{\partial\overline{z}^j}\right).
%\]
%This is equivalent to 
%\[
%2\sqrt{-1}\frac{\partial z^i}{\partial}
%\] 
\begin{equation}\label{zzetaeq}
\frac{\partial z^i}{\partial \ov{\zeta}^l}+\sum_{p=1}^n a^{i}_{\ov{p}}(z)\frac{\partial \ov{z}^p}{\partial\ov{\zeta}^l}=0,\quad i, l=1,\dots, n.
\end{equation}

We first recall the important homotopy operator in \cite{NeNi}. For a vector of $n$ complex-valued functions $F=(f_{\ov{1}},\cdots, f_{\ov{n}})$, denote (\cite[(2.5)]{NeNi}):
\[
\mathbb{T}F=\sum_{s=0}^{n-1}\frac{(-1)^s}{(s+1)!}\sum{}'\;  T^{j_1}\ov{\partial}_{j_1}\dots T^{j_s}\ov{\partial}_{j_s}\cdot T^k f_{\ov{k}}.
\]
where $\sum{}'$ denote the summation over all $(s+1)$-tuples with $j_1,\cdots, j_s, k$ distinct, and
\begin{eqnarray*}
T^{1}f(\zeta)&=&\frac{1}{2\pi i}\iint_{0<|\tau|<R}\frac{f(\tau,\zeta^2,\cdots,\zeta^n)}{\zeta^1-\tau}d\tau d\bar{\tau},\\
T^{j} f(\zeta)&=&\frac{1}{2\pi i}\iint_{|\tau|<R}\frac{f(\zeta^1,\cdots, \zeta^{j-1}, \tau, \zeta^j, \cdots, \zeta^n)}{\zeta^j-\tau}d\tau d\bar{\tau}, \mbox{ for } j\ge 2.
\end{eqnarray*}
For fit our setting, we need to modify $T^1$. First choose $N=\lceil \eta\rceil$. Then we define (see \eqref{eq-tildeT} and Lemma \ref{toresch}):
\begin{eqnarray*}
\tilde{T}^1 f(\zeta)&=&T^1 f(\zeta^1,\zeta^2,\cdots,\zeta^n)-T^1 f(0,\zeta^2,\cdots, \zeta^n)-\sum_{k=1}^{N-1} (T^1f)^{(k)}(0, \zeta^2, \cdots, \zeta^n)\frac{\zeta^k}{k!},\\
\tilde{T}^j f(\zeta)&=&T^jf(\zeta), \mbox{ if } j\ge 2.
\end{eqnarray*}

Then by Lemma \ref{toresch} and Lemma \ref{potential}, these operators are well defined for functions $f$ such that $f\sim O(|\zeta^1|^{\eta-1})$ and satisfy (see \cite[page 775]{Chern}) the following identities on $\bB_R^*\times\bB_R^{n-1}$:
\begin{equation}\label{commute}
\ov{\partial}_j \tilde{T}^j f= f, j=1, \cdots, n; \quad \mbox{ and } \quad \ov{\partial}_j\tilde{T}^k f=\tilde{T}^k\ov{\partial}_j f, \mbox{ for } j\neq k.
\end{equation}
Then we define 
\[
 \widetilde{\mathbb{T}}F(\zeta)=\sum_{s=0}^{n-1}\frac{(-1)^s}{(s+1)!}\sum{}'\;  \tilde{T}^{j_1}\ov{\partial}_{j_1}\dots \tilde{T}^{j_s}\ov{\partial}_{j_s}\cdot \tilde{T}^k f_{\ov{k}}.
\]
Then using relation \eqref{commute} to manipulate, we can easily get the following formula which is a variation of the formula in cf. \cite[2.6]{NeNi} by replacing the operator $T^j$ by $\tilde{T}^j$.
\begin{equation}\label{partialrel}
\ov{\partial}_{j}\tilde{\mathbb{T}}F-f_{\ov{j}}=\sum_{s=0}^{n-2}\frac{(-1)^s}{(s+2)!}\sum{}^{j} \; \tilde{T}^{j_1}\ov{\partial}_{j_1}\cdots \tilde{T}^{j_s}\ov{\partial}_{j_s}\cdot \tilde{T}^k(\ov{\partial}_j f_{\ov{k}}-\ov{\partial}_k f_{\ov{j}}).
\end{equation}
where $\sum{}^j$ denotes the summation over all $(s+1)$-tuples with $j_1, \dots, j_s, k$ distinct and different from $j$. From \eqref{zzetaeq}, we will denote
\begin{equation}\label{dbaright}
f^{i}_{\ov{l}}=-\sum_{p=1}^n a^i_{\ov{p}}(z)\frac{\partial\ov{z}^p}{\partial\ov{\zeta}^l}, \quad F^i=(f^i_{\ov{1}}, f^i_{\ov{2}},\dots, f^i_{\ov{n}})=\sum_{l=1}^n f^i_{\ov{l}}d\ov{\zeta}^l.
\end{equation}
Denote also $\mathfrak{z}^i(\zeta)=z^i(\zeta)-\zeta^i$. We then want to transform equations \eqref{zzetaeq} into:
\begin{equation}\label{fixedpteq}
z^i=\zeta^i+\widetilde{\mathbb{T}}(F^i(z))\Longleftrightarrow \mathfrak{z}^i=\widetilde{\mathbb{T}}(F^i(\zeta+\mathfrak{z})) \Longleftrightarrow \mathfrak{z}=\mathfrak{J}[\mathfrak{z}].
\end{equation}
We will show in Lemma \ref{actsol} that the solution to this equation with the appropriate control is indeed the solution to \eqref{zzetaeq}. To get solutions to the system \eqref{zzetaeq} with required order estimates, we would like to prescribe asymptotically behaviors:
\begin{equation}\label{asymp1}
z^1=\zeta^1+O(\rho^{1+\eta}),\; z^j=\zeta^j+O(\rho^{\eta})\;\;
\Longleftrightarrow\;\;
\mathfrak{z}^1=O(\rho^{1+\eta}),\; \mathfrak{z}^k = O(\rho^{\eta}).
\end{equation}
Here and in the following, we still denote $\rho=|\zeta^1|$ since $|\zeta^1|$ and $|z^1|$ is comparable with this prescription. If we denote $h$ the index $\{2,\cdots, n\}$, 
then the precise meaning of \eqref{asymp1} is the following
\begin{equation}\label{eq-usualHolder}
\left|\partial_{\zeta^1}^{l_1}\partial_{\zeta^h}^{l_2} (z^1-\zeta^1)\right|\le C(l_1,l_2) \left|\zeta^1\right|^{1+\eta-l_1},\quad \left|\partial_{\zeta^1}^{l_1}\partial_{\zeta^h}^{l_2}(z^h-\zeta^h)\right|\le C(l_1,l_2)\left|\zeta^1\right|^{\eta-l_1}, 
\mbox{ for all } l_1, l_2\ge 0.
\end{equation}
However to carry out the argument in \cite{NeNi}, we need first define the space of functions which have only ``mixed" higher order derivatives. So we will first consider the functions $\{z^i; i=1,\dots, n\}$ satisfying:
\begin{equation}\label{mixasymp1}
z^1=\zeta^1+\tilde{O}(\rho^{1+\eta}),\; z^j=\zeta^j+\tilde{O}(\rho^{\eta})\;\;
\Longleftrightarrow\;\;
\mathfrak{z}^1=\tilde{O}(\rho^{1+\eta}),\; \mathfrak{z}^k = \tilde{O}(\rho^{\eta}),
\end{equation}
which means the following estimates hold:
\begin{equation}
\left|\partial'^{l_1}_{\zeta^1}\partial'^{l_2}_{\zeta^h} (z^1-\zeta^1)\right|\le C(l_1,l_2) \left|\zeta^1\right|^{1+\eta-l_1},\quad \left|\partial'^{l_1}_{\zeta^1}\partial'^{l_2}_{\zeta^h}(z^h-\zeta^h)\right|\le C(l_1,l_2)\left|\zeta^1\right|^{\eta-l_1},
\end{equation}
where $\partial'$ means we don't allow repeated derivatives with respect to any single variable (see section \ref{sec-multvar}).

Under this prescription, by using \eqref{dbaright} and the asymptotic behavior of $a^{i}_{\ov{p}}$, we first show (Lemma \ref{actF}) that
\begin{equation}\label{eq-rough}
\begin{array}{l}
(f^1_{\bar{1}}, f^1_{\ov{m}})
%\sim
%\left(
%\begin{array}{cc}
%r^{\eta}& r^{2\eta+1}\\
%r^{2\eta}& r^{\eta+1}
%\end{array}
%\right)
= (\tilde{O}(\rho^\eta+\rho^{2\eta}), \tilde{O}(\rho^{2\eta+1}+\rho^{\eta+1}) 
= (\tilde{O}(\rho^\eta), \tilde{O}(\rho^{\eta+1}), \\
(f^j_{\bar{1}}, f^j_{\ov{m}})
%\sim\left(
%\begin{array}{cc}
%r^{\eta-1}& r^{2\eta}\\
%r^{2\eta-1}& r^{\eta}
%\end{array}
%\right)
= (\tilde{O}(\rho^{\eta-1}+\rho^{2\eta-1}), \tilde{O}(\rho^{2\eta}+\rho^{\eta}))= (\tilde{O}(\rho^{\eta-1}), \tilde{O}(\rho^{\eta})).
\end{array}
\end{equation}
Then we show that (Lemma \ref{actTT}):
\[
\widetilde{\mathbb{T}}[F^1] = \tilde{O}(\rho^{\eta+1}), \quad \widetilde{\mathbb{T}}[F^k] = \tilde{O}(\rho^{\eta}) \mbox{ for } k\ge 2.
\]
This is compatible with the prescription in \eqref{mixasymp1} and should allow us to use the arguments in \cite{NeNi} to solve the system \eqref{fixedpteq}. However, to use the contraction-iteration principle (see Lemma \ref{relax}), we have to relax asymptotic behaviors in \eqref{mixasymp1} a little bit by replacing $\eta$ by a $\nu$ satisfying  
\begin{equation}
0<\nu<\eta, \quad  \lceil \nu\rceil =\lceil \eta\rceil \quad \text{ and }\quad \nu\not\in \mathbb{N}.
\end{equation}
%Note the above condition in particular implies $\nu\not\in (0, \eta)\setminus \mathbb{N}$. 
Although replacing $\eta$ by $\nu$ might seem a loss of derivative, we will gain this $\epsilon$ back using the analyticity of transition functions. 

More precisely, in the next subsection, we will introduce weighted multiple H\"{o}lder norm $\|\cdot\|_{n+n\alpha,(\nu+1,\nu)}$ and show in Theorem \ref{yasuo} that, for any $\fz$, $\tilde{\fz}$ satisfying that when $R$ is sufficiently small and
$\left\|\fz\right\|_{n+n\alpha,(\nu+1,\nu)}\le R$, $\left\|\tilde{\fz}\right\|_{n+n\alpha,(\nu+1,\nu)}\le R$, then the following estimates hold:
\begin{enumerate}
\item
\begin{equation}\label{yasuo1}
\left\|\mathfrak{J}[\mathfrak{z}]\right\|_{n+n\alpha,(\nu+1,\nu)}\le R.
\end{equation}
\item
\begin{equation}\label{yasuo2}
\left\|\mathfrak{J}[\mathfrak{z}]-\mathfrak{J}[\tilde{\mathfrak{z}}]\right\|_{n+n\alpha, (\nu+1,\nu)}\le \frac{1}{2}\left\|\mathfrak{z}-\tilde{\mathfrak{z}}\right\|_{n+n\alpha,(\nu+1,\nu)}.
\end{equation}
\end{enumerate}
By standard iteration, there is a unique solution to the system \eqref{fixedpteq} such that:
\begin{equation}\label{asymp2}
\mathfrak{z}^1= \tilde{O}(\rho^{1+\nu}),\; \mathfrak{z}^j= \tilde{O}(\rho^{\nu}), \mbox{ or equivalently }
z^1=\zeta^1+\tilde{O}(\rho^{1+\nu}),\; z^j=\zeta^j+\tilde{O}(\rho^{\nu}).
\end{equation}

In the following $\bB_R=\{\zeta\in\mathbb{C}; |\zeta|\le R\}$ denotes the closed disc of radius R with center 0, and
$\bD_R^{*}=\{\zeta\in\mathbb{C}; 0<|\zeta|\le R\}$ denotes the punctured closed disc. We need to show that the map $\zeta\mapsto z$ gives a coordinate chart for $\zeta\in\bD_R^n$ when $R$ is sufficiently small.  First note that $\{z^i(\zeta)\}$ is identity for $\zeta^1=0$ and is H\"{o}lder continuous on $\{\zeta^1=0\}$.  Secondly on $\mathbb{U}_R=\bD_R^*\times\bD_R^{n-1}$, consider the Jacobian
\[
\mathbb{J}=\left(
\frac{\partial (z^i, \ov{z}^i)}{\partial (\zeta^j, \ov{\zeta}^j)}
\right).
%=\left(\begin{array}{cc}
%1+O(\rho^{\nu})& O(\rho^{\nu+1})\\
%O(\rho^{\nu-1})& 1+O(\rho^{\nu})
%\end{array}
%\right).
\]
By the similar argument for obtaining \eqref{invertible}, it's easy to see that $\mathbb{J}$ is invertible if $R$ is very small. So on $\mathbb{U}_R$, $\zeta\mapsto z$ is a local diffeomorphism to its image. We just need to show that it's an injective map and hence a homeomorphism. 

To do this, we decompose the coordinate change in \eqref{asymp1} into two steps. First we let
\begin{equation}
y^1=z^1(\zeta)=\zeta^1+\tilde{O}(|\zeta^1|^{1+\nu}),\quad y^k=\zeta^k \mbox{ for } k\ge 2.
\end{equation}
Since the Jacobian matrix is invertible and $C^{\nu}$, the map is a $C^{1,\nu}$-diffeomorphism and is clearly a change of coordinates. We can express $\zeta$ in terms of $y$ to get:
\[
\zeta^1=y^1+\tilde{O}(|y^1|^{1+\nu}),\quad \zeta^k=y^k \mbox{ for } k\ge 2.
\]
Now we can write the map in \eqref{asymp1} as:
\[
z^1=y^1, \quad z^k=y^k+\tilde{O}(|y^1|^\nu) \mbox{ for } k\ge 2.
\]
We just need to show this is injective. We assume $z(y)=z(\tilde{y})$. Then 
$y^1=\tilde{y}^1$, and $z^j(y)=z^j(\tilde{y})$. On the slice $y^1=\tilde{y}^1$, we connect $y$ and $\tilde{y}$ by $y_t=(1-t)y+t\tilde{y}$, then we have
\begin{eqnarray*}
0=\|z(\tilde{y})-z(y)\|&=&\sum_{j=1}^n \left|\int_0^1 \sum_{k=1}^n (\partial_{y^k} z^j)(y_t)\cdot (\tilde{y}^k-y^k) dt\right|\\
&=&\sum_{j=2}^n \left|\int_0^1 \sum_{k=2}^n (\delta^{j}_{k}+\tilde{O}(|y^1|^{\nu}))(\tilde{y}^k-y^k) dt \right|\\
&\ge& C(1-R^\nu)\|\tilde{y}-y\|.
\end{eqnarray*}
So if $R$ is sufficiently small, we indeed have $\tilde{y}=y$. 

To get all higher order estimates for the functions as stated in Proposition \ref{prop-coordinates}, i.e.
\begin{equation}\label{asymp3}
\mathfrak{z}^1= O(\rho^{1+\nu}),\; \mathfrak{z}^j=O(\rho^{\nu}), \mbox{ or equivalently }
z^1=\zeta^1+O(\rho^{1+\nu}),\; z^j=\zeta^j+O(\rho^{\nu}),
\end{equation}
we need to apply similar arguments as in \cite[6]{NeNi}
involving regularity theorems for elliptic equations \eqref{zzetaeq}.%\footnote{I am grateful for H.-J. Hein for pointing out this to me} 
Since this part of argument is now standard, we will be brief and refer to \cite[6]{NeNi} for references on differentiability theorems. By \eqref{asymp2}, we know that $z^1,\dots, z^n$ are $C^{1+\alpha}_{1+\nu, \nu}$ functions of $\zeta^j, \bar{\zeta}^j$ under the weighted H\"{o}lder norm. Because \eqref{zzetaeq} is first order elliptic,  we infer that $z^k$ are $C^{2+\alpha}_{1+\nu,\nu}$ with respect to the variables $\zeta^j, \bar{\zeta}^j$. Combining this with the ``mixed" second derivatives from \eqref{asymp2}, we see that $z^k$ are of class $C^{2+\alpha}_{1+\nu, \nu}$ whose norm is defined using all derivatives (including repeated derivatives) with respect to $\zeta^1,\dots, \zeta^n$. The higher order regularity follows from differentiating equations \eqref{zzetaeq} and improving the derivatives by standard bootstrapping argument. See \cite{CH3} for a different proof of the higher order estimates using gauge fixing method and regularity theorems.
%\begin{eqnarray*}
%|z^1(\tilde{\zeta})-z^1(\zeta)|&\ge &|z^1(\tilde{\zeta}^1,\zeta'')-z^1(\zeta^1, \zeta'')|-|z^1(\tilde{\zeta}^1, \tilde{\zeta}'')-z^1(\tilde{\zeta}^1, \zeta'')|\\
%&\ge&(\min |\partial_{\zeta^1}z^1|)|\tilde{\zeta}^1-\zeta^1|-(\max |\partial_{\zeta''} z^1|)|\tilde{\zeta}''-\zeta''|\\
%&=&
%\end{eqnarray*}
%Suppose not, then since $\mathbb{U}_R$ is connected, by mean value %theorem, we can find a point $\zeta^*\in\mathbb{U}_R$ such that $\mathbb{J}(\zeta^*)$ is degenerate, but this is not true if $R$ is very small.

This completes the sketch of the proof of Proposition \ref{prop-coordinates}, we will now explain the comfortable order of the divisor in the last statement in Theorem \ref{thm-anacpt}. Note that the transition function on the bundle $N_D\rightarrow D$ in terms of $\{z_\alpha^i\}$ are standard ones:
\[
z_\beta^1=a_{\beta\alpha}(z'')z_\alpha^1, \quad z_\beta^k=\phi_{\beta\alpha}^k(z_\alpha'') \mbox{ for } k\ge 2.
\]
By the asymptotical behavior \eqref{asymp3} and its inverse, we see that the transition functions in the $\zeta$-coordinates have the shape:
\[
\zeta_\beta^1=a_{\beta\alpha}(\zeta_\alpha'')\zeta_\alpha^1+O(|\zeta_\alpha^1|^{\nu+1}),\quad \zeta_\beta^k=\phi^k_{\beta\alpha}(\zeta_\alpha'')+O(|\zeta_\alpha^1|^{\nu}).
\]
We know that $\zeta_\beta^i$, for any $1\le i\le n$, is a holomorphic function of $\zeta_\alpha$ outside $D$, and from above expressions it's H\"{o}lder continuous across $D=\{\zeta_\alpha^1=0\}$. So we see that $\zeta_\beta^i$ is holomorphic across $D$ and hence is a holomorphic function of $\zeta_\alpha$. Denote $m=\lceil \nu \rceil=\lceil \eta\rceil=\lceil \lambda\delta \rceil$ (Recall that $\eta=\lambda\delta$ and $\nu=\eta-\epsilon$ for small $\epsilon$). Then the analyticity of holomorphic functions clearly implies that we must have the following improved transition:
\[
\zeta_\beta^1=a_{\beta\alpha}(\zeta_\alpha'')\zeta_\alpha^1+R_{m+1}^1,\quad \zeta_\beta^k=\phi^k_{\beta\alpha}(\zeta_\alpha'')+R_{m}^k,
\]
where $R_{m+1}^1\in \mI_D^{m+1}$, $R_{m}^k\in \mI_D^{m}$, where $\mI_D$ is the ideal sheaf of $D$ generated by $\{\zeta_\alpha^1\}$. By Theorem \ref{coordinate} (see also \eqref{cmftcoord}),  we see that in the compactification, the divisor $D$ is indeed $(m-1)$-comfortably embedded. In this way, we prove theorem \ref{thm-anacpt}.

\subsection{Estimates for the proof of Proposition \ref{prop-coordinates}}\label{sec-NNproof}

Suppose $f$ is a complex-valued function defined on $\bD_R^*\times \bD_R^{n-1}$. Denote $D_j$ either of the differential operators $\frac{\partial}{\partial\zeta^j}$, $\frac{\partial}{\partial\ov{\zeta}^j}$. $D^k$ will denote a general $k$-th order derivative $D^k=D_{i_1}\dots D_{i_k}$ with $i_1,\dots, i_k$ distinct (i.e. we consider only ``mixed'' derivatives). $D^{k,j}=D_{i_1}\dots D_{i_k}$ (resp. $D^{k,\{1,j\}}$) will denote such a derivative with the $i_1, \dots, i_k$ distinct and different from $j$ (resp. $\{1,j\}$). For a fixed positive $\alpha<1$, we denote the difference quotient operators:
\[
\delta_1 f=\frac{f(\tilde{\zeta}^1, \zeta^2, \cdots, \zeta^n)-f(\zeta^1, \zeta^2, \cdots, \zeta^n)}{|\tilde{\zeta}^1-\zeta^1|^{\alpha}} \mbox{ for } 0<|\zeta^1|\le R, 0<|\tilde{\zeta}^1|\le R, \zeta^1\neq \tilde{\zeta}^1.
\]
\[
\delta_i f=\frac{f(\zeta^1,\dots, \tilde{\zeta}^i,\dots,\zeta^n)-f(\zeta^1, \dots, \zeta^i, \dots, \zeta^n) }{|\tilde{\zeta}^i-\zeta^i|^{\alpha}} \mbox{ for } i>1, |\zeta^i|<R,  |\tilde{\zeta}^i|<R, \zeta^i\neq \tilde{\zeta}^i.
\]
%\[
%H^{(k)}_\alpha[f]=\max\{|\delta_{i_1}\dots\delta_{i_k} f|; i_1,\dots, i_k \mbox{ distinct}\}.
%\]
Denote $\delta^m=\delta_{j_1}\cdots \delta_{j_m}$ for $0\le m\le n$ and $j_1, \dots, j_m$ distinct; $\delta^0$ will denote the identity operator; $\delta^{m,1}$ will denote such a difference quotient with $j_1,\dots, j_m$ distinct and different from 1. 
%\[
%H_\alpha[f]=\sum_{m=0}^n\frac{R^{m\alpha}}{m!}\sup|\delta^m f|,\quad
%|f|_n=\sum_{k=0}^n\frac{R^k}{k!}\sup |D^k f|.
%\]
%\[
%|f|_{n+\alpha}=\sum_{k=0}^n\frac{R^k}{k!}H_\alpha[D^k f]\le \sum_{k,m=0}^n\frac{r^{k+m\alpha}}{k!m!}\sup |\delta^mD^k f|.
%\]
\subsubsection{Single-variable estimates}
The following is the standard Schauder estimate for the elliptic operator $\bar{\partial}$ for a single variable. 
\begin{lem}\label{stSchauder}
Assume $\alpha\in (0,1)$ is fixed. There exists a constant $c>0$ such that, if $w\in C^{1,\alpha}(\bD_1(0))$ satisfies $\frac{\partial w}{\partial\ov{\zeta}}=f$ in $\bD_1$ and if $f\in C^{0,\alpha}(\bD_1(0))$, then
\begin{equation}\label{Schauderest}
\|w\|_{C^{1,\alpha}(\bD_{1/2})}\le c\left(\|w\|_{L^{\infty}(\bD_1)}+\|f\|_{C^{0,\alpha}(\bD_1)}\right).
\end{equation}
\end{lem}
\begin{proof}
In the following proof, the constant $c$ may change but does not depend on $f \in C^{0, \alpha}(\bB(0))$.
Denote operators:
\begin{eqnarray*}
T f(\zeta)&=&\frac{1}{2\pi i}\iint_{\bD_1} \frac{f(\tau)}{\tau-\zeta}d\tau\wedge d\bar{\tau}, \quad
Sw(\zeta)=\frac{1}{2\pi i}\int_C\frac{w(\tau)}{\tau-\zeta} d\tau. %=\frac{1}{2\pi i}\int_C\frac{w(\tau)-w(\zeta)}{\tau-\zeta}d\tau+f(\zeta).
\end{eqnarray*}
Then $w\in C^{1,\alpha}(\bD_1)$ satisfies:
\[
w=T\partial_{\bar{\zeta}} w+S w=T f+ Sw .
\]
By Chern \cite[Main Lemma]{Chern}, we have
\[
\|Tf\|_{C^{1,\alpha}(\bD_1)}\le c \|f\|_{C^{0,\alpha}(\bD_1)}
\]
On the other hand, because $Sw=w-Tf$ is holomorphic, we have:
\[
\|Sw\|_{C^{1,\alpha}(\bD_{1/2})}\le c \|Sw\|_{L^{\infty}(\bD_1)}\le c(\| w\|_{L^\infty(\bD_1)}+\|Tf\|_{L^{\infty}(\bD_1)})\le c \|w\|_{L^\infty(\bD_1)}+c \|f\|_{L^{\infty}(\bD_1)}.
\]
\end{proof}
We need to extend the above Schauder estimate to the weighted H\"{o}lder space. We follow \cite[Chapter 2]{PaRi} to define the weighted H\"{o}lder norm for functions on the punctured disks. %Note that in this definition, we only care about the H\"{o}lder norm for two points with comparable distance to the puncture. 
%Note that this definition of weighted norm is slightly different from the definition used in for example \cite{HHN} and \cite{CH1}. Although two norms may be equivalent in some sense, the norm used here following \cite{PaRi} only uses the H\"{o}lder norm for $x$ and $y$ with comparable distances to the puncture. To the author's understanding, this constraint is well adapted to the rescaling argument. 
For any $s>0$, denote the annulus $\{\zeta^1\in\mathbb{C}; s<|\zeta^1|<2s\}$ by $A(s,2s)$. First we define the norm on the annulus: 
\begin{eqnarray*}
[w]_{1,\alpha,s}&:=&\sup_{A(s,2s)} |w|+s \sup_{A(s,2s)}|D_1 w|+s^{\alpha}\sup_{x,y\in A(s,2s)}\frac{|w(x)-w(y)|}{|x-y|^\alpha}+\\
&&\quad\quad +s^{1+\alpha}\sup_{x,y\in A(s,2s)}\frac{|D_1w(x)-D_1w(y)|}{|x-y|^\alpha}.
\end{eqnarray*}
The following is the scaling invariant weighted H\"{o}lder norm for functions on the punctured disk of radius $R$:
\begin{eqnarray*}
\|w\|_{C^{1,\alpha}_\nu(\bD_R(0))}=\sup_{s\in (0, R/2]}  s^{-\nu}[w]_{1,\alpha,s}, %\quad \tilde{H}_{\alpha,\nu}[w]=R^{\nu}\|w\|_{C^{1,\alpha}_\nu(\bD_R(0)}.
\end{eqnarray*}
As pointed out in \cite[Corollary 2.1]{PaRi}, the following Lemma is important for deriving the rescaled Schauder estimate in Lemma \ref{potential}.

Denote $m=\lceil \nu\rceil=\lceil \eta\rceil$ and the area form $\frac{d\tau\wedge d\bar{\tau}}{2\pi \sqrt{-1}}$ by $dV$, or $dV(\tau)$ if we want to emphasize the integration variable. 
%\noindent
%{\bf Notations: } 
%\begin{enumerate}
%\item
%In the single variable case, we will denote the operator $\tilde{T}^1$ by $\tilde{T}$, and 
%\item
%In the following proofs, the measures under the integral signs do not charge the single point $\{0\}$ so that in the estimates we can either use $\bB_R$ or the
%punctured disk $\bB_R^*=\bB_R\setminus\{0\}$ as the domain of integration. 
%\end{enumerate}
For any $f\in C^{1, \alpha}_{\nu-1}(\bB_R)$, define: 
\begin{eqnarray}\label{eq-tildeT}
\tilde{T} f(\zeta)&=&Tf(\zeta)-Tf(0)-\sum_{k=1}^{m-1} (Tf)^{(k)}(0)\frac{\zeta^k}{k!}\nonumber\\
&=&\frac{1}{2\pi i}\left(\iint_{\bB_R}\frac{f(\tau)}{\tau-\zeta}d\tau\wedge d\bar{\tau}-\sum_{k=0}^{m-1} \iint_{\bB_R} \frac{f(\tau)\zeta^k}{\tau^{k+1}}d\tau\wedge d\bar{\tau} \right)\nonumber\\
&=&\frac{1}{2\pi i}\iint_{\bB_R} \frac{f(\tau)\zeta^{m}}{(\tau-\zeta) \tau^{m}}d\tau\wedge d\bar{\tau}.
\end{eqnarray}

\begin{lem}\label{toresch}
Denote $\rho=|\zeta|$ for any $\zeta\in \bB_R^*$. Then there exists a positive constant $C$ independent of $R$, such that
for any $f\in C^{1,\alpha}_{\nu-1}(\bD_R)$, we have:
\begin{equation}\label{toscale}
\|\rho^{-\nu}\tilde{T}f\|_{L^{\infty}(\bD_R)}\le C\|\rho^{1-\nu} f\|_{L^{\infty}(\bD_R)}.
\end{equation}
\end{lem}
\begin{proof}
We can first estimate:
\begin{eqnarray*}
\left|\rho^{-\nu}\tilde{T}f\right|&=&
|\zeta|^{-\nu}\left| \iint_{\bD_R} \frac{f(\tau)\zeta^m}{(\tau-\zeta)\tau^m}dV\right|\\
&\le&\left\||\rho|^{1-\nu}f\right\|_{L^{\infty}} |\zeta|^{m-\nu} \iint_{\bD_R(0)}\frac{dV}{|\tau-\zeta||\tau|^{m+1-\nu}} 
\end{eqnarray*}
We split the integral into three parts:
\[
\iint_{\bD_R(0)}=\iint_{\bD_{\rho/2}(0)}+\iint_{\bD_{\rho/2}(\zeta)}+\iint_{\bD_R(0)\setminus (\bD_{\rho/2}(0)\cup \bD_{\rho/2}(\zeta))}={\bf I+II+III}.
\]
The inequality \eqref{toscale} follows from the following estimates:
\[
{\bf I}\le C\int_0^{\rho/2}\frac{ds}{s^{m-\nu}\rho/2}\le C\rho^{\nu-m}, \quad
{\bf II}\le C\int_0^{\rho/2}\frac{ds}{\rho^{m+1-\nu}}\le C\rho^{\nu-m}.
\]
To estimate part ${\bf III}$, it's easy to see that $|\tau-\zeta|\ge \frac{|\tau|}{4}$ for $\tau\in \bB_R(0)\setminus \bD_{\rho/2}(\zeta)$. So we can estimate for any $\nu<m$:
\[
{\bf III}\le C\int_{\rho/2}^R\frac{ds}{s^{m+1-\nu}}\le \frac{C}{m-\nu}\left(\left(\frac{\rho}{2}\right)^{\nu-m}-R^{\nu-m}\right)\le C \rho^{\nu-m}.
\]
It's clear that \eqref{toscale} follows by combining the above estimates. 
\end{proof}

\begin{lem}\label{potential}
%\begin{eqnarray*}
%&&R^{\nu}\left(\sup \frac{\left|\tilde{T}^1 f \right|}{|\zeta^1|^{\nu}}+\sup \frac{\left|\delta_1\tilde{T}^1 f\right|}{\min(|\zeta^1|,|\tilde{\zeta}^1|)^{\nu-\alpha}}+\sup\frac{\left|D_1\tilde{T}^1 f\right|}{|\zeta^1|^{\nu-1}}+\sup\frac{\left|\delta_1 D_1 \tilde{T}^1 f\right|}{\min(|\zeta^1|,|\tilde{\zeta}^1|)^{\nu-1-\alpha}}\right)\\
%&&\quad\le cR\cdot R^{\nu-1}\left(\sup\frac{\left|f\right|}{|\zeta^1|^{\nu-1}}+\sup\frac{\left|\delta_1 f\right|}{\min(|\zeta^1|,|\tilde{\zeta}^1|)^{\nu-1-\alpha}}\right).
%\end{eqnarray*}
If $f\in C^{0,\alpha}_{\nu-1}(\bD_R)$, then $\tilde{T}f\in C^{1,\alpha}_{\nu}(\bD_R)$ and satisfies:
\[
\|\tilde{T}f\|_{C^{1,\alpha}_\nu (\bD_R)}\le C \|f\|_{C^{0,\alpha}_{\nu-1}(\bD_R)}.
\]
%\[
%|\tilde{T}f|_{1+\alpha,\nu}\le c R |f|_{\alpha,\nu-1}.
%\]
\end{lem}

\begin{proof}
Let $F(\zeta)=\tilde{T}f(\zeta)$. Let $\rho=|\zeta|$. By Lemma \ref{stSchauder}, Lemma \ref{toresch} and standard rescaling argument as in \cite[Corollary 2.1]{PaRi}, we have:
\[
\|\tilde{T}f\|_{C^{1,\alpha}_\nu(\bD_{R/2})}\le  C\|f\|_{C^{0,\alpha}_{\nu-1}(\bD_R)}.
\]
To get estimate on $\bD_R\backslash \bD_{R/2}$, we use the explicit formula of $\tilde{T}$.
As in \cite[(18), (26)]{Chern}, we have:
\begin{eqnarray*}
F_{\ov{\zeta}}=f(\zeta), \quad F_{\zeta}=\frac{1}{2\pi \sqrt{-1}}\iint_{\bD_R(0)}\frac{f(\tau)-f(\zeta)}{(\tau-\zeta)^2}d\tau d\bar{\tau}-\frac{1}{2\pi \sqrt{-1}}\sum_{k=1}^{m-1}k \zeta^{k-1} \iint_{\bD_R(0)}\frac{f(\tau)}{\tau^{k+1}}d\tau\wedge d\bar{\tau}.
\end{eqnarray*}
So that
\begin{eqnarray*}
\left|\frac{|F_{\zeta}|}{|\zeta|^{\nu-1}}\right|&\le &\frac{1}{2\pi}\frac{1}{|\zeta|^{\nu-1}}\iint_{\bD_R(0)}\frac{|f(\tau)-f(\zeta)|}{|\tau-\zeta|^2}dV(\tau)+
\sum_{k=1}^{m-1} k R^{k-\nu} \|\rho^{1-\nu}f\|_{\infty} \int_0^R \frac{ds}{s^{k-\nu+1}}.
%&=&\frac{1}{2\pi}\iint\frac{|f(\tau)-f(\zeta)|}{|\tau-\zeta|^{\alpha} \min(|\tau|,|\zeta|)^{\nu-1-\alpha}}\frac{\min(|\tau|,|\zeta|)^{\nu-1-\alpha}}{|\zeta|^{\nu-1}}\frac{d\tau d\bar{\tau}}{|\tau-\zeta|^{2-\alpha}}\\
\end{eqnarray*}
The second term on the right-hand-side of the above identity is uniformly bounded by $C\|\rho^{1-\nu}f\|_{\infty}$.
To estimate the first integral term, we split it into two parts:
\[
\iint_{\bD_R(0)}=\iint_{\bD_{\rho/2}(0)}+\iint_{\bD_R(0)\backslash \bD_{\rho/2}(0)}={\bf I+II}.\\
\]
Here we need to separate the integral over $\bD_{\rho/2}(0)$ from each estimate since
we only have H\"{o}lder estimate for $x$ and $y$ of comparable lengths. 
Notice that we can assume $R/8\le |\zeta|\le R$ and estimate:
%\[
%\frac{|f(\tau)-f(\zeta)|}{|\tau-\zeta|^2}=\frac{|f(\tau)-f(\zeta)|}{|\tau-\zeta|^{\alpha} \min(|\tau|,|\zeta|)^{\nu-1-\alpha}}\frac{\min(|\tau|,|\zeta|)^{\nu-1-\alpha}}{|\tau-\zeta|^{2-\alpha}}\le C[f]_{\alpha,\nu}
%\frac{\min(|\tau|,|\zeta|)^{\nu-1-\alpha}}{|\tau-\zeta|^{2-\alpha}}.
%\]
%So
%\begin{eqnarray*}
%{\bf I}&\le &C [f ]_{\alpha,\nu}\iint_{\bD_{\rho/2}(0)}\frac{|\tau|^{\nu-1-\alpha}}{|\tau-\zeta|^{2-\alpha}}dV
%\le C\iint_{\bD_{\rho/2}(0)}\frac{|\tau|^{\nu-1-\alpha}}{|\zeta|^{2-\alpha}}dV \le C\pi \rho^{\alpha-2}\rho^{\nu-\alpha+1}\le C\rho^{\nu-1}.
%\end{eqnarray*}
\begin{eqnarray*}
{\bf I}&\le& \frac{1}{2\pi}\iint_{\bD_{\rho/2}(0)}\frac{1}{(|\zeta|-|\tau|)^2}\left(|\tau|^{1-\nu}|f(\tau)|\frac{|\zeta|^{1-\nu}}{|\tau|^{1-\nu}}+|\zeta|^{1-\nu} |f(\zeta)|\right)dV(\tau)\\
&\le& C\|\rho^{1-\nu}f\|_{L^\infty(\bD_R)}\frac{1}{R^{2}}\int_{0}^{R/2} (R^{1-\nu} s^{\nu-1}+1) sds\le C \|\rho^{1-\nu}f\|_{L^{\infty}(\bD_R(0))}.
\end{eqnarray*}
%\begin{eqnarray*}
\[
{\bf II}\le C\iint_{\bD_R(0)\backslash \bD_{\rho/2}(0)}\frac{\|f\|_{C^{0,\alpha}_{\nu-1}} |\tau-\zeta|^{\alpha} R^{-\alpha}}{|\tau-\zeta|^2}dV\\
%&\le&
\le C \|f\|_{C^{0,\alpha}_{\nu-1}}R^{-\alpha} \int_{0}^{2R} s^{\alpha-2+1}ds\le C\|f\|_{C^{0,\alpha}_{\nu-1}} .
\]
%\end{eqnarray*}
So we get $\|\rho^{1-\nu} D_1 \tilde{T}f\|_{L^{\infty}}\le C \|f\|_{C^{0,\alpha}_{\nu-1}}$, i.e. the $C^1$-estimate. This implies the $C^{0, \alpha}$ estimate:
\[
R^{\alpha} \sup_{x,y\in A(R/8,R)}\frac{|w(x)-w(y)|}{|x-y|^\alpha} \le C \|f\|_{C^{0, \alpha}_{\nu-1}}. 
\]
Similarly, one can prove that:
\[
R^{1+\alpha}\sup_{x,y\in A(R/8, R)}\frac{|D_1w(x)-D_1w(y)|}{|x-y|^{\alpha}}\le C \|f\|_{C^{0,\alpha}_{\nu-1}},
\]
with $w=\tilde{T}(f)$. In fact, we can prove the inequality as in \cite[Section 6.1e]{NiWo}, again the only difference is that we need to separate the integral over $\bD_{\rho/2}(0)$ from each estimate since
we only have H\"{o}lder estimate for $x$ and $y$ of comparable lengths. 
%\[
%{\bf II}\le C[f]_{\alpha,\nu}\iint_{\bD_{\rho/2}(\zeta)}|\zeta|^{\nu-1-\alpha} \frac{dV}{|\tau-\zeta|^{2-\alpha}}\le \frac{C}{\alpha}\rho^{\nu-1-\alpha}\rho^{\alpha}\le \frac{C}{\alpha}\rho^{\nu-1}.
%\]
%Part {\bf III} is estimated in a different way:
%\[
%{\bf III}\le C |\zeta|^{-1} \int_{\rho/2}^{R} s^{\nu-\alpha-(2-\alpha)+1}ds\le C(\nu) (\rho^{\nu-1}-R^{\nu-1}).
%\]
\end{proof}
\subsubsection{Multi-variable estimates}\label{sec-multvar}

Similarly to \cite[(3.1)-(3.3)]{NeNi}, we introduce the weighted multiple-H\"{o}lder space by incorporating the weighted 1st order H\"{o}lder space for $\zeta^1$ and the usual 1st order H\"{o}lder spaces for the other variables.  Formally, we introduce various norms:
\begin{enumerate}
\item (Integral part )
\[
\|u\|_{n,\nu}= \sum_{k=0}^{n-1}\left(\frac{R^{k}}{k!}\sup_{\bD_R(0)^*\times \bD_R(0)^{n-1}}\left(\frac{|D^{k,1}u|}{|\zeta^1|^{\nu}}\right)+\frac{R^{k+1}}{(k+1)!}\sup_{\bD_R(0)^*\times
\bD_R(0)^{n-1}}\left(\frac{|D_1D^{k,1}u|}{|\zeta^1|^{\nu-1}}\right)\right). %\le \sum_{k,m=0}^n\frac{R^{k+m\alpha+\nu}}{m!k!}\sup\left(\frac{|\delta^m D^k u|}{\min(|\zeta^1|,|\zeta^1+\delta\zeta^1|)^{\nu}}\right).
\]
%\begin{prop}[\cite{}]
%Assume $f\in C^{k-1,\alpha}_{\nu-1}(\overline{B_1}\backslash\{0\})$ and that $w$ solves $\frac{\partial w}{\partial\ov{\zeta}}=f$ in $B_1\backslash\{0\}$. Further assume that
%\[
%\|\rho^{-\nu}w\|_{L^{\infty}(B_1)}\le c\|\rho^{1-\nu} f\|_{L^{\infty}(B_1)},
%\]
%for some constant $c>0$ which does not depend on $f$, or on $w$. Then, there exists $c'>0$ which does not depend on $f$, or on $w$, such that
%\[
%\|w\|_{C^{k,\alpha}_{\nu}(B_{1/2})}\le c'\|f\|_{C^{k-1,\alpha}_{\nu-1}(B_1)}.
%\]
%\end{prop}
\item 
(Fractional part i.e. difference quotient part):
\begin{eqnarray*}
[u]_{n\alpha,\nu}&=&\sum_{m=1}^{n-1} \left(\frac{R^{m\alpha}}{m!}\sup_{\bD_R(0)^*\times \bD_R(0)^{n-1}}\left(\frac{|\delta^{m,1}u|}{|\zeta^1|^{\nu}}\right)+\frac{R^{(m+1)\alpha}}{(m+1)!}
\sup_{s\in (0,R/2)}s^{\alpha-\nu}\sup_{\{\zeta^1,\tilde{\zeta}^1\in A(s,2s)\}}|\delta_1 \delta^{m,1} u|\right).
\end{eqnarray*}
%\[
%[u]_{n\alpha,\nu}=\sum_{m=1}^{n-1} \frac{R^{m\alpha+\nu}}{m!}\sup\left(\frac{|\delta^{m,1}u|}{|\zeta^1|^{\nu}}\right) +\sum_{m=1}^{n-1} \frac{R^{(m+1)\alpha+(\nu-\alpha)}}{(m+1)!}
%\sup\left( \frac{|\delta_1 \delta^{m,1} u|}{\min(|\zeta^1|, |\zeta^1+\delta\zeta^1|)^{\nu-\alpha}}\right);
%\]
\item (0th-order weighted multiple H\"{o}lder norm)
\begin{eqnarray*}
\|u\|_{n\alpha,\nu}=\tilde{H}_{\alpha,\nu}[u]&=&\sup_{\bD_R(0)^*\times \bD_R(0)^{n-1}}\frac{|u|}{|\zeta^1|^{\nu}}+[u]_{n\alpha,\nu}\\
%&=&\sum_{m=0}^{n-1} \frac{R^{m\alpha+\nu}}{m!}\sup\left(\frac{|\delta^{m,1}u|}{|\zeta^1|^{\nu}}\right) +\sum_{m=0}^{n-1} \frac{R^{(m+1)\alpha+(\nu-\alpha)}}{(m+1)!}\sup\left( \frac{|\delta_1 \delta^{m,1} u|}{\min(|\zeta^1|, |\zeta^1+\delta\zeta^1|)^{\nu-\alpha}}\right);
\end{eqnarray*}
\item (1st-order weighted multiple H\"{o}lder norm)
\begin{eqnarray*}
\|u\|_{n+n\alpha, \nu}&=&\|u\|_{n,\nu}+\sum_{k=0}^{n-1}\left(\frac{R^k}{k!}[D^{k,1}u]_{n\alpha, \nu}+\frac{R^{k+1}}{(k+1)!}[D_1 D^{k,1}u]_{n\alpha,\nu-1}\right)\\
&=&\sum_{k=0}^{n-1}\left( \frac{R^{k}}{k!}\tilde{H}_{\alpha,\nu}[D^{k,1} u]+\frac{R^{k+1}}{(k+1)!}\tilde{H}_{\alpha,\nu-1}[D_1 D^{k,1} u]\right).
\end{eqnarray*}
\item (Partial 1st-order weighted multiple H\"{o}lder norm)
\[
\|u\|^1_{n-1+n\alpha,\nu}=\sum_{k=0}^{n-1}\frac{R^k}{k!}\sup_{\bD_R(0)^*\times \bD_R(0)^{n-1}}\tilde{H}_{\alpha,\nu}[D^{k,1}u].
\]
\[
\|u\|^j_{n-1+n\alpha,\nu}=\sum_{k=0}^{n-2}\left(\frac{R^k}{k!}\sup_{\bD_R(0)^*\times \bD_R(0)^{n-1}} \tilde{H}_{\alpha,\nu}[D^{k,\{1,j\}}u]+\frac{R^{k+1}}{(k+1)!}\tilde{H}_{\alpha,\nu-1}[D_1D^{k,\{1,j\}}u]\right) \mbox{ for } j\ge 2.
\]
\item (Anisotropically-weighted norm for vector of functions) Denote $\mathfrak{z}=(\mathfrak{z}^1(\zeta), \cdots, \mathfrak{z}^n(\zeta))$, $F=(f_{\ov{1}},\cdots, f_{\ov{n}})$. Denote:
\[
\|\mathfrak{z}\|_{n+n\alpha,(\nu+1,\nu)}=\|\mathfrak{z}^1\|_{n+n\alpha,\nu+1}+\sum_{j=2}^n \|\mathfrak{z}^j\|_{n+n\alpha,\nu}.
\]
\[
\quad \|F\|_{n-1+n\alpha,(\nu,\nu+1)}=\|f_{\ov{1}}\|^{1}_{n-1+n\alpha,\nu}+\sum_{j=2}^n \|f_{\ov{j}}\|^{j}_{n-1+n\alpha,\nu+1}.
\]
\end{enumerate}
Now we come back to solve the system \eqref{fixedpteq} which is equivalent to:
\begin{equation}\label{fixedpteq2}
\mathfrak{z}^i=\widetilde{\mathbb{T}}(F^i(\zeta+\mathfrak{z}))=\mathfrak{J}^i[\mathfrak{z}], \mbox{ where }
F^i=\left(f^i_{\ov{l}}\right)=\left(-\sum_{p=1}^n a^i_{\ov{p}}\frac{\partial\ov{z}^p}{\partial\ov{\zeta}^l}\right).
\end{equation}
%then
%\[
%z^i=\zeta^i+\widetilde{\mathbb{T}}(F^i(z)), \quad \mathfrak{z}^i=\widetilde{\mathbb{T}}(F^i(\zeta+\mathfrak{z})) \Longleftrightarrow \mathfrak{z}=\mathfrak{J}[\mathfrak{z}].
%\]
Arguing as in \cite{NeNi}, the following lemma is a consequence of definitions of above norms and Lemma \ref{potential}.
\begin{lem}[{cf. \cite[(3.4), Lemma 4.1, Lemma 4.3]{NeNi}}]\label{actT}
We have the following estimates:

\begin{equation}\label{actT1}
\begin{array}{l}
\|D_jf\|^j_{n-1+n\alpha, \nu}\le \frac{c}{R}\|f\|^{}_{n+n\alpha, \nu}, \quad j=1, \cdots, n; \\
\left\|\tilde{T}^jD_j f\right\|^{l}_{n-1+n\alpha,\nu}\le c \left\|f\right\|^{l}_{n-1+n\alpha,\nu}, \quad j,l=1,\cdots, n, j\neq l; \\
\left\|\tilde{T}^1 f\right\|_{n+n\alpha,\nu+1} \le c R \left\|f\right\|^{1}_{n-1+n\alpha,\nu};\\
\left\|\tilde{T}^j f\right\|_{n+n\alpha,\nu}\le c R \left\| f\right\|^{j}_{n-1+n\alpha,\nu} \mbox{ for } j\ge 2.
\end{array}
\end{equation}
%\begin{enumerate}
%\item
%\[
%\left\|\tilde{T}^jD_j f\right\|^{l}_{n-1+n\alpha,\nu}\le \left\|f\right\|^{l}_{n-1+n\alpha,\nu}, \quad j,l=1,\cdots, n, j\neq l.
%\]
%\item
%\begin{equation}\label{actT1}
%\left\|\tilde{T}^1 f\right\|_{n+n\alpha,\nu+1} \le c \left\|f\right\|^{1}_{n-1+n\alpha,\nu}.
%\end{equation}
%\item
%\[
%\left\|\tilde{T}^j f\right\|_{n+n\alpha,\nu}\le c  \left\| f\right\|^{j}_{n-1+n\alpha,\nu} \mbox{ for } j\ge 2.
%\]
%\end{enumerate}
\end{lem}
\begin{rem}\label{rem-mor}
Note that the moral of the above estimates are:
\begin{enumerate}
\item Differentiation with respect to $z^j$ for $j\neq 1$ keeps the weight unchanged and produces an $R^{-1}$ factor under appropriate norms.  $\tilde{T}^j$ for $j\neq 1$ keeps the weight unchanged and produces an $R$ factor. 
\item Differentiation with respect to $z^1$ decreases the weight and produces an $R^{-1}$ factor. $\tilde{T}^1$ improves the weight by $1$ and produces an extra $R$ factor.
\end{enumerate}
\end{rem}
Packing these estimates for components of $F^1$, $F^j$, the above Lemma implies:
\begin{lem}[{cf. \cite[Theorem 4.1]{NeNi}}]\label{actTT}
\[
\left\|\widetilde{\mathbb{T}}(F^1)\right\|_{n+n\alpha,\nu+1}\le c R \|F^1\|_{n-1+n\alpha,(\nu,\nu+1)};
\quad
\left\|\widetilde{\mathbb{T}}(F^j)\right\|_{n+n\alpha,\nu}\le c R \|F^j\|_{n-1+n\alpha,(\nu-1,\nu)}, \mbox{ for } j\ge 2.
\] 
%Informally, we can write:
%\[
%\widetilde{\mathbb{T}}(F^1)\sim \rho^{\nu+1}, \quad \widetilde{\mathbb{T}}(F^j)\sim \rho^{\nu}, \mbox{ for } j\ge 2.
%\]
\end{lem}
The following lemma follows from the decay rate of $(a^{i}_{\bar{j}})$ in identity \ref{invertible} and the definition of norms defined above. It shows the reason to relax the asymptotics by replacing $\eta$ by $\nu$.
\begin{lem}[cf. {\cite[Lemma 3.1]{NeNi}}]\label{relax}
Suppose $\|\mathfrak{z}\|_{n+n\alpha,(\nu+1,\nu)}\le 1$, then 
\[
\|a^{1}_{\ov{1}}(\zeta+\mathfrak{z})\|_{n-1+n\alpha,\nu}\le K R^{\eta-\nu}(1+R^\nu \|\mathfrak{z}\|_{n+n\alpha,(\nu+1,\nu)}), \quad \|a^{1}_{\ov{k}}(\zeta+\mathfrak{z})\|_{n-1+n\alpha,\nu+1}\le K R^{\eta-\nu}(1+R^\nu \|\mathfrak{z}\|_{n+n\alpha,(\nu+1,\nu)}).
\]
\[
\|a^{k}_{\ov{1}}(\zeta+\mathfrak{z})\|_{n-1+n\alpha,\nu-1}\le K R^{\eta-\nu}(1+R^{\nu}\|\mathfrak{z}\|_{n+n\alpha,(\nu+1,\nu)}), \quad \|a^{j}_{\ov{k}}(\zeta+\mathfrak{z})\|_{n-1+n\alpha,\nu}\le K R^{\eta-\nu}(1+R^{\nu}\|\mathfrak{z}\|_{n+n\alpha,(\nu+1,\nu)}).
\]
\end{lem}
The following lemma is the precise formulation of the estimates in \eqref{eq-rough}. Notice that if $\eta=1$, then we get back the estimate in \cite[Lemma 5.1]{NeNi}.
\begin{lem}[cf. {\cite[Lemma 5.1]{NeNi}}]\label{actF}
If $\|\mathfrak{z}\|_{n+n\alpha,(\nu+1,\nu)}\le R$. Then 
\begin{equation}
\|F^1\|_{n-1+n\alpha,(\nu,\nu+1)}\le C R^{\eta-\nu}(1+R^\nu \|\mathfrak{z}\|_{n+n\alpha, (\nu+1,\nu)}), \;
\left\|F^1[\mathfrak{z}]-F^1[\tilde{\mathfrak{z}}]\right\|_{n+n\alpha,(\nu,\nu+1)}\le C R^{\eta-1} \left\|\mathfrak{z}-\tilde{\mathfrak{z}}\right\|_{n+n\alpha,(\nu+1,\nu)}.
\end{equation}
For $j\ge 2$, we have:
\begin{equation}
\|F^j\|_{n-1+n\alpha, (\nu-1,\nu)}\le C R^{\eta-\nu} (1+R^\nu\|\mathfrak{z}\|_{n+n\alpha,(\nu+1,\nu)}), \;
\left\|F^j[\mathfrak{z}]-F^j[\tilde{\mathfrak{z}}]\right\|_{n+n\alpha,(\nu-1,\nu)}\le C R^{\eta-1} \left\|\mathfrak{z}-\tilde{\mathfrak{z}}\right\|_{n+n\alpha,(\nu+1,\nu)}.
\end{equation}
\end{lem}
\begin{proof}
We prove the first two estimates for $F^1=(f^1_{\bar{1}}, f^1_{\bar{m}})$. We first deal with $f^1_{\bar{1}}$:
\begin{equation}\label{eq-f1b1}
f^1_{\bar{1}}=a^1_{\bar{1}}\frac{\partial \bar{z}^1}{\partial\bar{\zeta}^1}+\sum_{m>1} a^1_{\bar{m}}\frac{\partial \bar{z}^m}{\partial \bar{\zeta}^1}. 
\end{equation}
For the first term on the right-hand-side of \eqref{eq-f1b1}, we have the following estimate.
\begin{eqnarray*}
\left\|a^{1}_{\ov{1}}\frac{\partial \ov{z}^1}{\partial\ov{\zeta}^1}\right\|^1_{n-1+n\alpha, \nu} &\le& \|a^{1}_{\ov{1}}\|^1_{n-1+n\alpha,\nu} \left(1+\left\|\frac{\partial\ov{\mathfrak{z}}^1}{\partial \ov{\zeta}^1}\right\|^1_{n-1+n\alpha,0}\right)\\%\le K \rho^{\nu-1}(1+\rho^{\nu})\\
%&\le& KR^{\eta-\nu} (1+\|\mathfrak{z}^1\|_{n+n\alpha,\nu+1}R^{\nu})%\rho^{\nu-1}
%(1+\left\|\mathfrak{z}^1\right\|_{n+n\alpha,\nu+1}R^{\nu}).
&\lesssim& KR^{\eta-\nu} (1+R^\nu \|\mathfrak{z}\|_{n+n\alpha,(\nu+1,\nu)})%\rho^{\nu-1}
(1+\left\|\mathfrak{z}^1\right\|_{n+n\alpha,\nu+1}R^{\nu-1})\\
&\lesssim& R^{\eta-\nu}(1+R^\nu \|\mathfrak{z}\|_{n+n\alpha, (\nu+1, \nu)}),
\end{eqnarray*}
where we estimated $\|a^1_{\bar{1}}\|^1_{n-1+n\alpha, \nu}$ using Lemma \ref{relax}. Using Remark \ref{rem-mor}, we can estimate: 
\begin{eqnarray*}
\left\|a^1_{\ov{1}}(\zeta+\mathfrak{z})\frac{\partial\ov{z}^1}{\partial\ov{\zeta}^1}-a^1_{\ov{1}}(\zeta+\tilde{\mathfrak{z}})\frac{\partial\ov{\tilde{z}}^1}{\partial\ov{\zeta}^1}\right\|^1_{n-1+n\alpha,\nu}&\le&
\|a^1_{\ov{1}}(\zeta+\mathfrak{z})-a^1_{\ov{1}}(\zeta+\tilde{\mathfrak{z}})\|^1_{n-1+n\alpha,\nu}\left\|\frac{\partial \ov{z}^1}{\partial\ov{\zeta}^1}\right\|^1_{n-1+n\alpha,0}\\
&&\quad+\|a^1_{\ov{1}}(\zeta+\tilde{\mathfrak{z}})\|_{n-1+n\alpha,\nu}\left\|\frac{\partial(\fz^1-\tilde{\fz}^1)}{\partial\ov{\zeta}^1}\right\|^1_{n-1+n\alpha,0}\\
&\le&K R^{\eta-1} \|\fz-\tilde{\fz}\|_{n+n\alpha,(\nu+1,\nu)}.
\end{eqnarray*}
\noindent
In the above estimates, similar with the method in our proof that $\zeta\mapsto z$ gives coordinate charts, we have estimated the difference of $a^{1}_{\ov{1}}(z)-a^{1}_{\ov{1}}(\tilde{z})$ by decomposing into two parts and then uses mean value theorem to get the above estimate (cf. \cite[Page 401]{NeNi}):
\begin{eqnarray*}
\|a^{1}_{\ov{1}}(\zeta+\fz)-a^1_{\ov{1}}(\zeta+\tilde{\fz})\|^1_{n-1+n\alpha, \nu}&=&\|a^1_{\ov{1}}(\zeta+\fz)-a^1_{\ov{1}}(\zeta^1+\tilde{\fz}^1, \zeta''+\fz'')\|_{n-1+n\alpha,\nu}\\
&&\hskip 12mm+\|a^1_{\ov{1}}(\zeta^1+\tilde{\fz}^1, \zeta''+\fz'')-a^1_{\ov{1}}(\zeta^1+\tilde{\fz}^1, \zeta''+\tilde{\fz}'')\|_{n-1+n\alpha,\nu}\\
&\lesssim& R^{\eta-1}\|\fz^1-\tilde{\fz}^1\|_{n+n\alpha, \nu+1}+R^{\eta-1}\|\fz''-\tilde{\fz}''\|_{n+n\alpha, \nu}.
\end{eqnarray*}
The following estimates deal with the second part on the right-hand-side of \eqref{eq-f1b1}.
\begin{eqnarray*}
\left\|a^{1}_{\ov{m}}\frac{\partial \ov{z}^m}{\partial\ov{\zeta}^1}\right\|^1_{n-1+n\alpha, \nu} &\le& \|a^{1}_{\ov{m}}\|^1_{n-1+n\alpha,\nu+1} \left\|\frac{\partial\ov{\mathfrak{z}}^m}{\partial \ov{\zeta}^1}\right\|^1_{n-1+n\alpha,-1}%\le K \rho^{\nu-1}(1+\rho^{\nu})\\
%\le K (1+\|\mathfrak{z}^1\|_{n+n\alpha,\nu}R^{\nu})R^{\eta-\nu}%\rho^{\nu-1}
%\left\|\mathfrak{z}^m\right\|_{n+n\alpha,\nu}R^{\nu}.
\lesssim KR^{\eta-\nu}(1+R^\nu \|\mathfrak{z}\|_{n+n\alpha, (\nu+1, \nu)})%\rho^{\nu-1}
\left\|\mathfrak{z}^m\right\|_{n+n\alpha,\nu}R^{\nu-1}\\
&\lesssim& R^{\eta-\nu}(1+R^{\nu}\|\mathfrak{z}\|_{n+n\alpha, (\nu+1, \nu)}).
\end{eqnarray*}
In the last inequality, we used $\|\mathfrak{z}^m\|_{n+n\alpha, \nu}\le R$.

\begin{eqnarray*}
\left\|a^1_{\ov{m}}(\zeta+\mathfrak{z})\frac{\partial\ov{z}^m}{\partial\ov{\zeta}^1}-a^1_{\ov{m}}(\zeta+\tilde{\mathfrak{z}})\frac{\partial\ov{\tilde{z}}^m}{\partial\ov{\zeta}^1}\right\|^1_{n-1+n\alpha,\nu}&\le&
\|a^1_{\ov{m}}(\zeta+\mathfrak{z})-a^1_{\ov{m}}(\zeta+\tilde{\mathfrak{z}})\|^1_{n-1+n\alpha,\nu+1}\left\|\frac{\partial \ov{z}^m}{\partial\ov{\zeta}^1}\right\|^1_{n-1+n\alpha,-1}\\
&&\quad+\|a^1_{\ov{m}}(\zeta+\tilde{\mathfrak{z}})\|_{n-1+n\alpha,\nu+1}\left\|\frac{\partial(\fz^m-\tilde{\fz}^m)}{\partial\ov{\zeta}^1}\right\|^1_{n-1+n\alpha,-1}\\
&\le&K R^{\eta-1} \|\fz-\tilde{\fz}\|_{n+n\alpha,(\nu+1,\nu)}.
\end{eqnarray*}
We used the estimate:
\begin{eqnarray*}
\|a^{1}_{\ov{m}}(\zeta+\fz)-a^1_{\ov{m}}(\zeta+\tilde{\fz})\|^1_{n-1+n\alpha, \nu+1}&=&\|a^1_{\ov{m}}(\zeta+\fz)-a^1_{\ov{m}}(\zeta^1+\tilde{\fz}^1, \zeta''+\fz'')\|_{n-1+n\alpha,\nu+1}\\
&&\hskip 10mm+\|a^1_{\ov{m}}(\zeta^1+\tilde{\fz}^1, \zeta''+\fz'')-a^1_{\ov{m}}(\zeta^1+\tilde{\fz}^1, \zeta''+\tilde{\fz}'')\|_{n-1+n\alpha,\nu+1}\\
&\lesssim& R^{\eta-1}\|\fz^1-\tilde{\fz}^1\|_{n+n\alpha, \nu+1}+R^{\eta-1}\|\fz''-\tilde{\fz}''\|_{n+n\alpha, \nu}.
\end{eqnarray*}
In the same way, one can verify the other estimates. % $f^1_{\ov{m}}\sim O(\rho^{2\nu+1}+\rho^{\nu+1})$,
%$f^j_{\ov{1}}\sim O(\rho^{\nu-1}+\rho^{2\nu-1})$ and $f^j_{\ov{m}}\sim  O(\rho^{2\nu}+\rho^{\nu})$.

\end{proof}

\subsubsection{Completion of the proof of Proposition \ref{prop-coordinates}}
Combining Lemma \ref{actTT} and \ref{actF}, we get:
\begin{thm}\label{yasuo}
For any $\fz$, $\tilde{\fz}$ satisfying $\|\fz\|_{n+n\alpha, (\nu+1,\nu)}\le R$, $\|\tilde{\fz}\|_{n+n\alpha, (\nu+1,\nu)}\le R$ with $R$ sufficiently small, we have 
\[
\left\|\mathfrak{J}(\fz)\right\|_{n+n\alpha,(\nu+1,\nu)}\le c R^{\eta-\nu}(1+R^\nu \|\fz\|_{n+n\alpha,(\nu+1,\nu)});
\] 
\[
\left\|\mathfrak{J}(\tilde{\fz})-\mathfrak{J}(\fz)\right\|_{n+n\alpha,(\nu+1,\nu)}\le 
cR^{\eta} \left\|\tilde{\fz}-\fz\right\|_{n+n\alpha,(\nu+1,\nu)}.
\] 
\end{thm}
So for $R$ sufficienty small, we indeed get the desired inequalities \eqref{yasuo1} and \eqref{yasuo2} to apply the contraction-iteration principle to get a solution 
to the system \eqref{fixedpteq2}. 
\begin{lem}\label{actsol}
If $\fz$ is a solution to the system \eqref{fixedpteq2}, then $\fz$ is a solution to \eqref{zzetaeq}, i.e.
\begin{equation}\label{zzetaeq3}
g^{i}_{\ov{j}}=\frac{\partial z^i}{\partial \ov{\zeta}^j}+\sum_{p=1}^n a^{i}_{\ov{p}}(z)\frac{\partial \ov{z}^p}{\partial\ov{\zeta}^j}=0,\quad i, j=1,\dots, n.
\end{equation}
\end{lem}
\begin{proof}
We follow the argument in \cite[Page 403]{NeNi}. Using the formula \eqref{partialrel} and calculating as in \cite[(2.11-2.12)]{NeNi} (see also \cite[4.1.2]{NiWo}) we get the following identity
\begin{equation}\label{waitzero}
g^{i}_{\ov{j}}=\sum_{s=0}^{n-2}\frac{(-1)^s}{(s+2)!}\sum{}^j\;\;\tilde{T}^{j_1}\ov{\partial}_{j_1}\cdots \tilde{T}^{j_s}\ov{\partial}_{j_s}\cdot \tilde{T}^k 
[(\partial_{p} a^{i}_{\ov{m}})(\zeta)(\ov{\partial}_{j} \ov{z}^m \cdot g^{p}_{\ov{k}}-\ov{\partial}_{k} \ov{z}^m\cdot g^{p}_{\ov{j}})]
\end{equation}
where $\sum{}^j$ denotes the summation over all $(s+1)$-tuples with $j_1, \dots, j_s, k$ distinct and different from $j$. We claim that from \eqref{waitzero} the following holds:
\begin{eqnarray}\label{lastest}
\|G^1\|_{n-1+n\alpha,(\nu,\nu+1)}+\|G^j\|_{n-1+n\alpha,(\nu-1,\nu)}\le C R^{\eta+\nu} (\|G^1\|_{n-1+n\alpha,(\nu,\nu+1)}+\|G^j\|_{n-1+n\alpha,(\nu-1,\nu)}),
\end{eqnarray}
where we denote $G^{i}=(g^i_{\ov{1}},\cdots, g^i_{\ov{n}})$. Assuming \eqref{lastest} holds, then when $R$ is sufficiently small we have $G^i=0$ and so we indeed get the solution to \eqref{zzetaeq3}. 
To verify the claim, we need to estimate the term in the bracket:
\[
\mathfrak{G}^i_{\ov{j}\ov{k}}:=\sum_{p,m} (\partial_{p} a^{i}_{\ov{m}})(\zeta)(\ov{\partial}_{j} \ov{z}^m \cdot g^{p}_{\ov{k}}-\ov{\partial}_{k} \ov{z}^m\cdot g^{p}_{\ov{j}})=:\sum_{m,p} \frak{G}^i_{\bar{j}\bar{k}mp}.
\]
We will estimate it for different cases of indices. 
\begin{enumerate}
\item ($i=1$, $j=1$) In this case $k\ge 2$ (since $k\neq j$ in $\sum{}^{j}$). 
\begin{enumerate}
\item ($p=1,m=1$) 
Note that $\partial_1 a^1_{\bar{1}}=\tilde{O}(R^{\eta-1})$, $\bar{\partial}_1\bar{z}^1=\tilde{O}(1+R^\nu)$, $\|g^1_{\bar{k}}\|_{n-1+n\alpha,\nu+1}=\tilde{O}(R^{\nu+1})$, $\bar{\partial}_k\bar{z}^1=\bar{\partial}_k\mathfrak{z}^1=\tilde{O}(R^{\nu+1})$, $\|g^1_{\bar{j}}\|_{n-1+n\alpha,\nu}=\tilde{O}(R^\nu)$. So we can estimate the summand as:
\begin{eqnarray*}
\|\mathfrak{G}^1_{\ov{1}\ov{k}mp}\|_{n-1+n\alpha,(\nu,\nu+1)}
&\le& R^{\eta-1} ((1+R^{\nu})\cdot R^{\nu+1} \|G^1\|_{n-1+n\alpha,(\nu,\nu+1)}\\
&&\quad
+R^{\nu+1}\|\bar{\partial}_k\bar{\mathfrak{z}}^1\|_{n-1+n\alpha,\nu+1}\cdot R^{\nu} \|G^1\|_{n-1+n\alpha,(\nu,\nu+1)}) \\
&\le & R^{\eta+\nu} \left(\|G^1\|_{n-1+n\alpha,(\nu,\nu+1)}+\|G^j\|_{n-1+n\alpha,(\nu-1,\nu)} \right).
\end{eqnarray*}
For convenience, we will just write formally that the following holds: 
\begin{equation*}
\mathfrak{G}^1_{\ov{1}\ov{k}mp}=\tilde{O}(\rho^{\eta-1}(\rho^{0+\nu+1}+\rho^{\nu+1+\nu}))=\tilde{O}(\rho^{\eta}\rho^{\nu}).
\end{equation*} 
By the same reason, we can estimate the other summands:
\item ($p\ge 2, m=1$) $\mathfrak{G}^1_{\ov{1}\ov{k}mp}=\tilde{O}(\rho^{\eta}(\rho^{0+\nu}+\rho^{\nu+1+\nu-1}))=\tilde{O}(\rho^{\eta}\rho^{\nu})$.
\item ($p=1, m\ge 2$) $\mathfrak{G}^1_{\ov{1}\ov{k}mp}=\tilde{O}(\rho^{\eta}(\rho^{\nu-1+\nu+1}+\rho^{0+\nu}))=\tilde{O}(\rho^{\eta}\rho^{\nu}) $.
\item ($p\ge 2, m\ge 2$) $\mathfrak{G}^1_{\ov{1}\ov{k}mp}=\tilde{O}(\rho^{\eta+1}(\rho^{\nu-1+\nu}+\rho^{0+\nu-1}))=\tilde{O}( \rho^{\eta}\rho^{\nu})$.
\end{enumerate}
The above four estimates (a)-(d) combined to give:
\begin{eqnarray*}
\|\mathfrak{G}^1_{\bar{1}\bar{k}}\|_{n-1+n\alpha, \nu}%&\lesssim& \rho^\eta (\|G^1\|_{n-1+n\alpha, (\nu, \nu+1)}+\|G^j\|_{n-1+n\alpha, (\nu-1, \nu)})\\
&\lesssim& R^{\eta+\nu} (\|G^1\|_{n-1+n\alpha, (\nu, \nu+1)}+\|G^j\|_{n-1+n\alpha, (\nu-1, \nu)}).
\end{eqnarray*}
The same remark applies to the notations in the following estimates:
\item ($i=1, j\ge 2$) In this case $k$ can be $1$. 
\begin{enumerate}
\item ($k=1$)  We estimate norm $\|\mathfrak{G}^1_{\bar{j}\bar{1}}\|_{n-1+n\alpha, \nu}$:
\begin{enumerate}
\item ($p=1,m=1$) $\mathfrak{G}^1_{\ov{j}\ov{1}mp}=\tilde{O}(\rho^{\eta-1}(\rho^{\nu+1+\nu}+\rho^{0+\nu+1}))=\tilde{O}(\rho^{\eta}\rho^{\nu})$.
\item ($p\ge 2, m=1$) $\mathfrak{G}^1_{\ov{j}\ov{1}mp}=\tilde{O}(\rho^{\eta}(\rho^{\nu+1+\nu-1}+\rho^{0+\nu}))=\tilde{O}(\rho^{\eta}\rho^{\nu})$.
\item ($p=1, m\ge 2$) $\mathfrak{G}^1_{\ov{j}\ov{1}mp}=\tilde{O}(\rho^{\eta}(\rho^{0+\nu}+\rho^{\nu-1+\nu+1}))=\tilde{O}(\rho^{\eta}\rho^{\nu})$.
\item ($p\ge 2, m\ge 2$) $\mathfrak{G}^1_{\ov{j}\ov{1}mp}=\tilde{O}(\rho^{\eta+1}(\rho^{0+\nu-1}+\rho^{\nu-1+\nu}))=\tilde{O}( \rho^{\eta}\rho^{\nu})$.
\end{enumerate}
\item ($k\ge 2$) We use the norm $\|\mathfrak{G}^1_{\bar{j}\bar{k}}\|_{n-1+n\alpha, \nu+1}$.
\begin{enumerate}
\item ($p=1,m=1$) $\mathfrak{G}^1_{\ov{j}\ov{k}mp}=\tilde{O}(\rho^{\eta-1}(\rho^{\nu+1+\nu+1}+\rho^{\nu+1+\nu+1}))=\tilde{O}(\rho^{\eta+\nu}\rho^{\nu+1})$.
\item ($p\ge 2, m=1$) $\mathfrak{G}^1_{\ov{j}\ov{k}mp}=\tilde{O}(\rho^{\eta}(\rho^{\nu+1+\nu}+\rho^{\nu+1+\nu}))=\tilde{O}(\rho^{\eta+\nu}\rho^{\nu+1})$.
\item ($p=1, m\ge 2$) $\mathfrak{G}^1_{\ov{j}\ov{k}mp}=\tilde{O}(\rho^{\eta}(\rho^{0+\nu+1}+\rho^{0+\nu+1}))=\tilde{O}(\rho^{\eta}\rho^{\nu+1})$.
\item ($p\ge 2, m\ge 2$) $\mathfrak{G}^1_{\ov{j}\ov{k}mp}=\tilde{O}(\rho^{\eta+1}(\rho^{0+\nu}+\rho^{0+\nu}))=\tilde{O}(\rho^{\eta}\rho^{\nu+1})$.
\end{enumerate}
\end{enumerate}
\item ($i\ge 2, j=1$) In this case $k\ge 2$. From the expression of $\mathfrak{G}^i_{\ov{j}\ov{k}}$, we see that the only difference from the 
case $i=1, j=1$ lies in the term $\partial_p a^i_{\ov{m}}$. We just need to decrease each order by $1$ to get
\[
\mathfrak{G}^i_{\ov{1}\ov{k}}=\tilde{O}(\rho^{\eta}\rho^{\nu-1}),
\]
or equivalently:
\begin{eqnarray*}
\|\mathfrak{G}^i_{\bar{1}\bar{k}}\|_{n-1+n\alpha,\nu}\le R^{\eta+\nu-1}\left(\|G^1\|_{n-1+n\alpha,(\nu,\nu+1)}+\|G^{j}\|_{n-1+n\alpha,(\nu-1,\nu)}\right).
\end{eqnarray*}

\item ($i\ge 2$, $j\ge 2$) In this case, $k$ can be $1$. Again, we see that the only difference with the case $i=1, j\ge 2$ lies in the term 
$\partial_p a^i_{\ov{m}}$. So we just need to decrease each order by $1$ to get
\[
\mathfrak{G}^i_{\ov{j}\ov{1}}=\tilde{O}(\rho^{\eta}\rho^{\nu-1}), \mbox{ and } \quad \mathfrak{G}^i_{\ov{j}\ov{k}}=\tilde{O}(\rho^{\eta}\rho^{\nu}).
\]

\end{enumerate}
Now from item 1, we have that:
\begin{eqnarray*}
\|g^1_{\ov{1}}\|^1_{n-1+n\alpha,\nu} &\le& C \sum_{k\ge 2}\|\tilde{T}^k \mathfrak{G}^1_{\ov{1}\ov{k}}\|_{n-1+n\alpha, \nu} \\
&\le& C R^{\eta+\nu} (\|G^1\|_{n-1+n\alpha, (\nu,\nu+1)}+\|G^j\|_{n-1+n\alpha,(\nu-1,\nu)}).
\end{eqnarray*}
From item 2, we have for $j\ge 2$, 
\begin{eqnarray*}
\|g^1_{\ov{j}}\|^j_{n-1+n\alpha,\nu} &\le& C(\|\tilde{T}^1\mathfrak{G}^1_{\ov{j}\ov{1}}\|^j_{n-1+n\alpha,\nu+1}+\sum_{k\ge 2}\|\tilde{T}^k\mathfrak{G}^1_{\ov{j}\ov{k}}\|^j_{n-1+n\alpha,\nu+1})\\
&\le& C R^{\eta+\nu} (\|G^1\|_{n-1+n\alpha, (\nu,\nu+1)}+\|G^j\|_{n-1+n\alpha,(\nu-1,\nu)}).
\end{eqnarray*}
Note that we have used the fact from \eqref{actT} that the operator $\tilde{T}^1$ improves the weight from $\nu$ to $\nu+1$. The same argument apply to item 3 and 4 too. So we indeed get the estimate
\eqref{lastest}. 
\end{proof}

%\section{Appendix}
%\begin{thm}[Kantorovitch's theorem]
%Let $a_0$ be a point in $\mathbb{R}^n$, $U$ an open neighborhood of $a_0$ in $\mathbb{R}^n$ and
%$F: U\rightarrow \mathbb{R}^n$ mapping, with its derivative. Define
%\[
%{\bf h}_0=-
%\]
%If the derivative $DF(x)$ satisfies the Lipschitz condition
%\[
%|DF(u_1)-DF(u_2)|\le M |u_1-u_2| \mbox{ for all points } u_1, u_2\in U_0,
%\]
%and if the inequality
%\[
%M\le 1/2
%\]
%is satisfied, the equation $f(x)=0$ has a unique solution in $U_0$ and Newton's method with initial guess $a_0$ converges to it. 
%\end{thm}

%Assume we have an equivariant degeneration $\mX\rightarrow \mathbb{C}$ such that $\mX_0=C$, the general fibre $\mX_t\cong X$.  The Kodaira-Spencer map
%gives an element in $\mathbf{T}^1_C$. 

\section{Appendices}

\subsection{Neighborhoods of complex submanifold after Grauert-Abate-Bracci-Tovena}\label{linearizable}

Assume $S$ is a smooth complex submanifold of $X$. In the introduction, we have recalled the definition of $S(k)$ and the concept of linearizability. Grauert \cite{Grau} showed that the obstruction for extending an isomorphism $S(k-1)\rightarrow S_N(k-1)$ to an isomorphism $S(k)\rightarrow S_N(k)$ lies in the cohomology group $H^1(S, \Theta_X|_S\otimes \mathcal{I}_S^{k}/\mathcal{I}_S^{k+1})$. He also pointed out that this obstruction consists of two parts. To see this, consider the exact sequence:
\[
0\rightarrow \Theta_S\otimes\mI_S^{k}/\mI_S^{k+1} \rightarrow \Theta_{X}|_{S}\otimes \mI_S^{k}/\mI_S^{k+1}\rightarrow N_{S}\otimes\mI_S^{k}/\mI_S^{k+1}\rightarrow 0,
\]
from which we get the long exact sequence:
\[
\dots \rightarrow H^1(S, \Theta_S\otimes\mI_S^k/\mI_S^{k+1})\rightarrow H^1(S, \Theta_X|_S\otimes \mI_S^{k}/\mI_S^{k+1})\rightarrow H^1(S,N_S\otimes \mI_S^k/\mI_S^{k+1})\rightarrow\dots
\]
So roughly speaking, the obstruction comes from two parts, one from $H^1(S, N_S\otimes \mI_S^{k}/\mI_S^{k+1})$ and the other from $H^1(S, \Theta_S\otimes\mI_S^k/\mI_S^{k+1})$. In \cite{ABT}, Abate-Bracci-Tovena explicitly described these two cohomological obstruction classes, and introduced the notion of {\it k-splitting} and {\it k-comfortably embedded} such that
$k$-linearizable=$k$-splitting+$(k-1)$-comfortably embedded with respect to the induced $(k-1)$-th order lifting. For references in the main paper, we record Abate-Bracci-Tovena's results in this section.
\begin{defn}[{\cite[Definition 2.1,2.2]{ABT}}]\label{def-split}
\begin{enumerate}
\item
$S$ is $k$-splitting into $X$ (for some $k\ge 1$) if the exact sequence 
\[
0\longrightarrow \mathcal{I}_S/\mathcal{I}_S^{k+1}\longrightarrow \mathcal{O}_X/\mathcal{I}_S^{k+1}\longrightarrow\mathcal{O}_S\rightarrow 0
\]
splits as a sequence of sheaves of rings. 
\item
A $k$-splitting atlas for $S\subset X$ is an atlas $\{(V_\alpha, z_\alpha)\}$ of $X$ adapted to $S$ (that is, $V_\alpha\cap S\neq \emptyset$ implies $V_\alpha\cap S=\{z_\alpha^1=\cdots=z^m_{\alpha}=0\}$) such that
\[
\left.\frac{\partial^k z^p_\beta}{\partial z^{r_1}\cdots \partial z^{r_k}_\alpha}\right|_S\equiv 0, 
\]
for all $r_1,\dots, r_k=1,\dots, m$, all $p=m+1,\dots, n$, and all indices $\alpha, \beta$ such that $V_\alpha\cap V_\beta\cap S\neq \emptyset$.
\item 
An atlas $\{(V_\alpha, z_\alpha)\}$ adapted to $S$ is adapted to a $k$-th order lifting $\rho: \mathcal{O}_S\rightarrow \mathcal{O}_M/\mI_S^{k+1}$ if
\begin{equation}
\rho[f]_1=\sum_{l=0}^{k}(-1)^l\left[\frac{\partial^l f}{\partial z_\alpha^{r_1}\cdots \partial z^{r_l}_\alpha}z^{r_1}_\alpha\cdots z^{r_l}_\alpha\right]_{k+1},
\end{equation}
for every $f\in \mathcal{O}(V_\alpha)$ and all indices $\alpha$ such that $V_\alpha\cap S\neq \emptyset$.
\end{enumerate}
\end{defn}
In the following, if $S$ is $k$-splitting, we will fix a lifting: $\rho_k: \mO_S\rightarrow \mO_X/\mI_S^{k+1}$. We also denote by $\phi_{h,k}$ the natural map
\begin{equation}\label{defphi}
\phi_{h,k}: \mO_X/\mI_S^{h+1}\rightarrow \mO_X/\mI_S^{k+1}, \mbox{ for } h\ge k.
\end{equation}
 \begin{prop}[{\cite[Proposition 2.2]{ABT}}]\label{obsplit}
Assume that $S$ is $(k-1)$-splitting in $X$; let $\rho_{k-1}: \mathcal{O}_S\rightarrow \mathcal{O}_X/\mathcal{I}_S^{k}$ be a $(k-1)$-th order lifting, and $\mathfrak{V}=\{(V_\alpha, \phi_\alpha)\}$ a $(k-1)$-splitting atlas adapted to $\rho_{k-1}$. Let $\mathfrak{g}^{\rho_{k-1}}_k\in H^1(S, Hom(\Omega_S, \mathcal{I}_S^k/\mathcal{I}_S^{k+1}))$ be the \v{C}ech cohomology class represented by a 1-cocycle $\{(\mathfrak{g}_k^{\rho_{k-1}})_{\beta\alpha}\}$ $\in H^1(\mathfrak{V}_S, Hom(\Omega_S, \mathcal{I}_S^k/\mathcal{I}_S^{k+1}))$ given by
 \begin{equation}\label{spltccl}
 (\mathfrak{g}_k^{\rho_{k-1}})_{\beta\alpha}=-\frac{1}{k!}\left.\frac{\partial^k z^p_\alpha}{\partial z^{r_1}_\beta\dots \partial z^{r_k}_\beta}\right|_S\frac{\partial}{\partial z^p_\alpha}\otimes [z^{r_1}_\beta\dots z^{r_k}_\beta]_{k+1}\in H^0(V_\alpha\cap V_\beta\cap S, \Theta_S\otimes
 \mathcal{I}_S^{k}/\mathcal{I}_S^{k+1}).
 \end{equation}
 Then there exists a k-th order lifting $\rho_{k}:\mathcal{O}_S\rightarrow\mathcal{O}_X/\mathcal{I}_S^{k+1}$ such that $\rho_{k-1}=\phi_{k, k-1}\circ\rho_k$ if and only if $\mathfrak{g}^{\rho_{k-1}}_k=0$. We call 
 this $\mathfrak{g}^{\rho_{k-1}}_k$ the obstruction to k-splitting relative to $\rho_{k-1}$.
 \end{prop}
 \begin{prop}[{\cite[Proposition 3.2]{ABT}}]\label{ODmodule}
Assume $S$ is $k$-splitting in $X$ and let $\rho: \mathcal{O}_S\rightarrow\mathcal{O}_X/\mathcal{I}_S^{k+1}$ be a k-th order lifting, with $k\ge 0$. Then for any $1\le h\le k+1$, the lifting $\rho$ induces a structure of locally $\mathcal{O}_S$-free module on $\mathcal{I}_S/\mathcal{I}_S^{h+1}$ for $1\le h\le k+1$ in such a way that the 
sequence
\begin{equation}\label{splitmodule}
0\longrightarrow\mathcal{I}_S^h/\mathcal{I}_S^{h+1}\longrightarrow \mathcal{I}_S/\mathcal{I}_S^{h+1}\longrightarrow\mathcal{I}_S/\mathcal{I}_S^h\longrightarrow 0
\end{equation}
becomes an exact sequence of locally $\mathcal{O}_S$-free modules. 
 \end{prop}
 \begin{defn}[{\cite[Definition 3.1, 3.2]{ABT}}]\label{def-cft}
 \begin{enumerate}
 \item
If $S$ is $k$-splitting in $X$ and the sequence \eqref{splitmodule} splits for $1\le h\le k+1$, $S$ is called to be k-comfortably embedded in $X$. Denote by $\nu_{h-1,h}:\mI_S/\mI_S^{h}\rightarrow \mI_S/\mI_S^{h+1}$ the  splitting $\mO_S$-morphism of the sequence \eqref{splitmodule} and the comfortable splitting sequence ${\boldsymbol \nu}_{k}=(\nu_{0,1}, \dots, \nu_{k,k+1})$. 
\item
A  k-comfortable atlas is an atlas $\{(V_\alpha, z_\alpha)\}$ adapted to $S$ such that
\[
\frac{\partial z_\beta^p}{\partial z_\alpha^r}\in \mI_S^k, \mbox{ and } \frac{\partial^2 z_\beta^r}{\partial z^{s_1}_\alpha\partial z^{s_2}_\alpha}\in \mI_S^k \quad
\Longleftrightarrow\quad
\left.\frac{\partial^k z^p_{\beta}}{\partial z^{r_1}_{\alpha}\dots\partial z^{r_k}_\alpha}\right|_S\equiv 0, \mbox{ and } \left.\frac{\partial^{k+1}z_\beta^s}{\partial z^{r_1}_\alpha\dots\partial z^{r_{k+1}}_\alpha}\right|_S\equiv 0,
\]
for all $r_1,\dots, r_k=1,\dots, m$, all $p=m+1,\dots, n$, and all indices $\alpha, \beta$ such that $V_\alpha\cap V_\beta\cap S\neq \emptyset$.
\end{enumerate}
\end{defn}
\begin{rem}\label{cftrem}
Any submanifold $S$ is always $0$-comfortably embedded and $0$-linearizable, but is not always $1$-linearizable (which is equivalent to having a splitting tangent sequence). If $S$ is $k$-comfortably embedded, then $S$ is also $k$-splitting.
\end{rem}
\begin{thm}[{\cite[Corollary 3.6]{ABT}}]\label{comfortob}
 Assume there exists a k-th order lifting $\rho_k: \mathcal{O}_S\rightarrow \mathcal{O}_X/\mathcal{I}_S^{k+1}$ such that $S$ is $(k-1)$-comfortably embedded in $X$ with respect to $\rho_{k-1}=\phi_{k, k-1}\circ\rho_k$. Fix a $(k-1)$-comfortable pair $(\rho_{k-1},{\boldsymbol \nu}_{k-1})$, and let $\mathfrak{V}=\{(V_\alpha, z_\alpha)\}$ be a projectable atlas adapted to $\rho_k$ and $(\rho_{k-1},{\boldsymbol\nu}_{k-1})$. Then the cohomology class $\mathfrak{h}^{\rho_k}$ associated to the exact sequence \eqref{splitmodule} is represented by 1-cocycle $\{\mathfrak{h}^{\rho_k}_{\beta\alpha}\}\in H^1(\mathfrak{V}_S, \mathcal{N}_S\otimes\mathcal{I}_S^{k+1}/\mathcal{I}_{S}^{k+2})$ given by
 \[
\mathfrak{h}^{\rho_k}_{\beta\alpha}=-\frac{1}{(k+1)!}\frac{\partial z^{s_1}_\beta}{\partial z^{r_1}_\alpha}\dots\frac{\partial z^{s_{k+1}}_\beta}{\partial z_\alpha^{r_{k+1}}}
 \left.\frac{\partial^{k+1} z^t_\alpha}{\partial z^{s_1}_\beta\dots \partial z^{s_{k+1}}_\beta}\right|_S\partial_{z^t_\alpha}
 \otimes [z^{r_1}_\alpha\dots z^{r_{k+1}}_{\alpha}]_{k+2}.
 \]
 \end{thm}
 \begin{rem}\label{psdiv}
 If $D$ is a smooth divisor, then the obstruction to $k$-comfortable embedding lies in 
 $H^1(D, N_D\otimes \mI_D^{k+1}/\mI_D^{k+2})=H^1(D, (N_D)^{-k})$. If we assume the normal bundle $N_D$ is ample on $D$ and $n-1={\rm dim}D\ge 2$, then the Kodaira-Nakano vanishing theorem
 gives $H^1(D, (N_D)^{-k})=0$ for any $k\ge 1$. So in this case, there is no obstruction to passing from $(k-1)$-comfortable embedding to $k$-comfortable embedding (with respect to any $k$-splitting). Note that $D$ is always $0$-comfortably embedded. So we obtain that, if $N_D$ is ample on $D$ and ${\rm dim} X\ge 3$, then $D$ is $k$-comfortably embedded, if and only if $D$ is $k$-splitting, and if and only if $D$ is $k$-linearizable (see Theorem \ref{crilinear}). 
 \end{rem}
\begin{thm}[{\cite[Theorem 2.1, Theorem 3.5]{ABT}}]\label{coordinate}
$S$ is $k$-splitting in $X$ if and only if there is a $k$-splitting atlas $\mathfrak{V}=\{(V_\alpha, z_\alpha)\}$ of $X$, that is an atlas adapted to $S$ such that
\[
\left\{
\begin{array}{ccll}
z_\beta^r&=&\sum_{s=1}^m (a_{\beta\alpha})^r_s(z_\alpha)z_\alpha^s, & \mbox{ for } r=1,\dots, m, \\
&&&\\
z_\beta^p&=&\phi_{\beta\alpha}^p(z''_\alpha)+R^p_{k+1}, & \mbox{ for } p=m+1,\dots, n,
\end{array}
\right.
\] 
where $z_\alpha''=(z_\alpha^{m+1}, \cdots, z_\alpha^n)$ are local coordinates on $S$, and $R^p_{k+1}$ denotes a term belonging to $\mathcal{I}_S^{k+1}$.
Furthermore, $S$ is k-comfortably embedded in $X$ if and only if there is a k-comfortable atlas $\mathfrak{V}=\{(V_\alpha, z_\alpha)\}$, that is an atlas adapted to $S$ such that 
\[
\left\{
\begin{array}{ccll}
z_\beta^r&=&\sum_{s=1}^m (a_{\beta\alpha})^r_s(z''_\alpha)z_\alpha^s+R^r_{k+2}, & \mbox{ for } r=1,\dots, m, \\
&&&\\
z_\beta^p&=&\phi_{\beta\alpha}^p(z''_\alpha)+R^p_{k+1}, & \mbox{ for } p=m+1,\dots, n,
\end{array}
\right.
\] 
where $R^r_{k+2}\in\mathcal{I}_S^{k+2}$ and $R^p_{k+1}\in \mathcal{I}_S^{k+1}$.
\end{thm}
\begin{thm}[{\cite[Theorem 4.1]{ABT}}]\label{crilinear}
$S$ is k-linearizable if and only if $S$ is $k$-splitting into $X$ and (k-1)-comfortably embedded with respect to the $(k-1)$-th order lifting induced by the $k$-splitting, if and only if there is an atlas $\mathfrak{V}$ such that the changes of coordinates are of the form:
\[
\left\{
\begin{array}{ccll}
z_\beta^r&=&\sum_{s=1}^m (a_{\beta\alpha})^r_s(z''_\alpha)z_\alpha^s+R^r_{k+1}, & \mbox{ for } r=1,\dots, m, \\
&&&\\
z_\beta^p&=&\phi_{\beta\alpha}^p(z''_\alpha)+R^p_{k+1}, & \mbox{ for } p=m+1,\dots, n,
\end{array}
\right.
\] 
where $R^{r}_{k+1}, R_p^{k+1}\in \mathcal{I}_S^{k+1}$.
\end{thm}

%In particular, the obstruction in Proposition \ref{obsplit} is given by:
%\[
%(\mathfrak{g}_k)_{\beta\alpha}=[R_{\beta\alpha, 
%k+1}^p]_{k+2}\otimes \frac{\partial}{\partial z_\alpha^p} \in \Theta_S\otimes\mathcal{I}_S^{k+1}/\mathcal{I}_S^{k+2}.
%\]

\subsection{Deformation of normal algebraic varieties}\label{appcn}

\subsubsection{First order deformations}
%Here we recall Schlessinger's work in \cite{Schl1}, \cite{Schl2} on the deformation of normal affine varieties with isolated singularities. 
Assume $Z$ is an complex analytic variety in $\bC^N$. Choose any analytically open set $W$ of $\bC^N$ and assume that $\cI_Z(W)$ is generated by $\{f_1, \cdots, f_d\}$. Let $\cZ\rightarrow \bD$ be a flat deformation
of $Z$ with $\cZ_0=Z$, which is realized as an embedding deformation of $Z\rightarrow \bC^N$. Assume $\cI_{\cZ_t}(W)$ is generated by $\{f_1+t g_1, \cdots, f_d+t g_d\}$. Then 
by the flatness condition, $\{g_i\}$ induces a morphism:
\[
\bar{g}: \cI_Z/\cI_Z^2 \rightarrow \cO_{\bC^N}/\cI_Z=\cO_Z, \quad \sum_i [f_i h_i]\mapsto \sum_i g_i h_i|_{Z}.
\]
So we get: $\bar{g}\in Hom_{\mathcal{O}_Z}(\cI_Z/\cI_Z^2, \cO_Z)=H^0(Z, N_Z)$. To get the space of first order infinitesimal deformations of $Z$, one considers the conormal exact sequence:
\[
\mathcal{I}_Z/\mathcal{I}_Z^2\rightarrow \Omega_{\mathbb{C}^N}|_{Z}\rightarrow \Omega_{Z}\rightarrow 0,
\]
whose dual defines the sheaf $\mathcal{T}^1_Z$ (see \cite[1.2]{Schl1} and \cite[Proposition II1.25]{GLS07}):
\[
0\rightarrow \Theta_Z\rightarrow \Theta_{\mathbb{C}^N}|_Z\rightarrow N_{Z}\rightarrow \mathcal{T}_Z^1\rightarrow 0.
\]
Since we assumed $Z$ is embedding in $\mathbb{C}^N$, we get the exact sequence:
\begin{equation}\label{exact1}
0\rightarrow H^0(Z, \Theta_Z)\rightarrow H^0(Z, \Theta_{\bC^N}|_Z)\rightarrow H^0(Z, N_Z) \stackrel{\psi_Z}{\longrightarrow} \mathbf{T}_Z^1\rightarrow 0.
\end{equation}
In particular, $\mathbf{T}_Z^1$ is defined so that \eqref{exact1} becomes exact and is not equal to $H^1(Z, \Theta_Z)$ in general.
%Assume $D$ is a projective K\"{a}hler manifold with a positive line bundle $L\rightarrow D$. Consider the affine cone: 
%\begin{equation}\label{conexp}
%C:=C(D,L)={\rm Spec} \bigoplus_{k=0}^{+\infty} H^0(D, L^{\otimes k}).
%\end{equation}
%\[
%\mathbb{T}_{C}^1={\rm Ker} \left(H^1(D, X\otimes L^j)\rightarrow H^1(D, \oplus_{i=1}^N L^{j+n_i}\right).
%\]
%$C$ has a normal isolated singularity at the vertex $\underline{o}$. 
The image of $\bar{g}$ in ${\bf T}_Z^1$ is the first order information of the deformation $\cZ\rightarrow \bD$.
\begin{prop}[{\cite[Theorem 1]{Schl1}}, \cite{Schl2}]\label{prop-Sch}
Assume $Z$ has an isolated normal singularity $o$ and denote by $U=Z\setminus \{o\}$. Then there are exact sequences:
\begin{equation}\label{surjexact}
H^0(U, \Theta_{\mathbb{C}^N}|_U)\rightarrow H^0(U, N_U)\stackrel{\psi_U}{\longrightarrow} \mathbf{T}_Z^1\rightarrow 0
\end{equation}
\begin{equation}\label{injexact}
0 \longrightarrow \mathbf{T}_Z^1 \stackrel{\tau_U}{\rightarrow} H^1(U,\Theta_U)\rightarrow H^1(U, \Theta_{\mathbb{C}^N}|_U)
\end{equation}
\end{prop}
\begin{proof}
For the reader's convenience, we sketch the proof here.
%We have the conormal exact sequence:
%\[
%\mathcal{I}_C/\mathcal{I}_C^2\rightarrow \Omega_{\mathbb{C}^N}|_{C}\rightarrow \Omega_{C}\rightarrow 0,
%\]
%whose dual defines the sheaf $\mathcal{T}^1$:
%\[
%0\rightarrow \Theta_C\rightarrow \Theta_{\mathbb{C}^N}|_C\rightarrow N_{C}\rightarrow \mathcal{T}^1\rightarrow 0.
%\]
%In particular, the first three sheaves are reflexive. Because $C$ is affine, by definition we get the exact sequence:
%\begin{equation}\label{exact1}
%0\rightarrow H^0(C, \Theta_C)\rightarrow H^0(C, \Theta_{\mathbb{C}^N|C})\rightarrow H^0(C, N_{C}) \rightarrow \mathbf{T}_C^1\rightarrow 0.
%\end{equation}
Because $Z$ is normal, by Serre's criterion for normality, $Z$ has depth ${\rm depth}_{\underline{o}}Z\ge 2$ at its vertex. Because the first three sheaves in \eqref{exact1} are reflexive, by \cite[Lemma 1]{Schl2}, the depth
of each is $\ge 2$. So in \eqref{exact1} we can replace $H^0(Z, \cdot)$ by $H^0(U,\cdot)$ to get:
\begin{equation}\label{presurj}
 0\rightarrow H^0(U,\Theta_U)\rightarrow H^0(U,\Theta_{\mathbb{C}^N}|_U)\rightarrow H^0(U, N_U)\rightarrow \mathbf{T}_Z^1\rightarrow 0, 
\end{equation}
On the other hand, because $U$ is smooth and embedded into $\bC^N$, we have 
\[
0\rightarrow \Theta_U\rightarrow \left.\Theta_{\mathbb{C}^N}\right|_U\rightarrow N_U\rightarrow 0,
\]
which gives us the exact sequence:
\begin{equation}\label{preinj}
0\rightarrow H^0(U,\Theta_U)\rightarrow H^0(U,\Theta_{\mathbb{C}^N}|_U)\rightarrow H^0(U, N_U)\stackrel{\delta}{\longrightarrow} H^1(U,\Theta_U)\rightarrow H^1(U, \Theta_{\mathbb{C}^N}|_U).
\end{equation}
Combining \eqref{presurj} and \eqref{preinj}, we get \eqref{surjexact} and \eqref{injexact}.
\end{proof}

%Under the connecting morphism, it's natural that we have $\delta(-\bar{g})=[\theta_{k+1}]$.
 %Then we call the class $\bar{g}\in {\bf T}_Z^1$ to be the reduced Kodaira-Spencer class of $\cZ\rightarrow \bD$.

\subsubsection{Deformation of affine cones} \label{App-cone}
As an example of the above general theory, consider a projective manifold $D\subset \mathbb{P}^{N-1}$. We assume that $D$ is projectively normal in $\mathbb{P}^{N-1}$ so that the affine cone over $D$ is normal and is equal to $C=C(D,H)$ where $H$ is the hyperplane bundle of $\mathbb{P}^{N-1}$.
Then it's easy to verify that (see \cite{Schl1}, \cite{Artin}):
\[
H^0(U,\Theta_{\mathbb{C}^{N}}|_U)=\sum_{j=-\infty}^{+\infty}H^0(D, \mathcal{O}_D(j+1)),\quad H^0(U, N_U)=\sum_{j=-\infty}^{+\infty} H^0(D, N_D(j)).
\]
Decompose $\mathbf{T}_C^1=\sum_{j=-\infty}^{+\infty}\mathbf{T}^1_C(j)$ into weight spaces. Then by \eqref{surjexact} we have the exact sequence:
\begin{equation}\label{wdsurj}
H^0(D, \mathcal{O}_D(j+1))^{N}\stackrel{{\rm Jac}}{\longrightarrow} H^0(D, N_D(j))\longrightarrow \mathbf{T}_C^1(j)\rightarrow 0.
\end{equation}
\begin{exmp}[cf. {\cite[Section 4]{Artin}}, \cite{KS}]\label{exmpclint}
Assume $D^{n-1}\subset \mathbb{P}^{N-1}$ is a complete intersection
\[
D=\bigcap_{i=1}^{N-n}\{F_i=0\}\subset \mathbb{P}^{N-1},
\]
where $F_i$ is a homogeneous polynomial of degree $d_i$. We assume $\{Z_1,\cdots, Z_N\}$ are homogeneous coordinates of $\mathbb{P}^{N-1}$ and denote
\[
R(D, H)=\bigoplus_{m=0}^{+\infty} H^0(D, m H)\cong \mathbb{C}[Z_1,\cdots, Z_{N}]/\langle F_1,\cdots, F_{N-n}\rangle.
\]
Note that this is nothing but the affine coordinate ring of $C(D,H)$. Then
\[
H^0(D, \mO_D(j+1))=H^0(D, (j+1)H)=R(D,H)(j+1); 
\]
\[
H^0(D, N_D(j))=\bigoplus_{i=1}^{N-n} H^0(D, (d_i+j)H)=\bigoplus_{i=1}^{N-n} R(D, H)(d_i+j).
\]
The map 
\[
{\rm Jac}: R(D, H)(j+1)^N\rightarrow \bigoplus_{i=1}^{N-n}R(D,H)(d_i+j)
\] 
is given by the Jacobian matrix 
$\left(\partial F_k/\partial Z^l \right)_{k=1,\cdots, N-n}^{l=1,\cdots, N}$, with the quotient: 
\begin{equation}\label{eq-T1Cj}
\mathbf{T}_C^1(j)=\frac{\bigoplus_{i=1}^{N-n}R(D,H)(d_i+j)}{{\rm Jac}(R(D,H)(j+1)^{\oplus N})}.
\end{equation}
Now assume $\mathcal{G}=\{g_i=g_i(z_1,\cdots, z_{N}), i=1,\cdots, N-n\}$ consists of (not necessarily homogeneous) polynomials. 
We can consider the deformation of $C(D,H)\subset \mathbb{C}^N$ given by:
\[
\mathcal{C}_t=\bigcap_{i=1}^{N-n} \{F_i(z_1,\cdots, z_N)+t g_i=0\}\subset \mathbb{C}^{N}.
\]
If we assume image $[\mathcal{G}]$ in $\mathbf{T}^1_C$ is not zero, then by \eqref{wdsurj}, we see that the weight of this deformation is the weight of $[\mathcal{G}]$.  Note that the polynomials in the image of ${\rm Jac}$ have degree $\ge d_i-1$. So if $g_i$ is of degree $e_i\le d_i-2$, it's easy to see that the $[\mathcal{G}]$ is indeed not zero and the weight is equal to 
$\max\{e_i-d_i\}=-\min\{d_i-e_i\}$.
\begin{rem}
The reason that we assume 
the non vanishing of $[\mathcal{G}]$ is to guarantee the induced map $\mathbb{C}\rightarrow \mathbf{T}^1_C$ does not have a vanishing 1st order derivative. Otherwise, we can consider the reduced Kodaira-Spencer class as the following example shows:
\[
\{z_1^2+z_2^2+z_3^2=0\}\leadsto \{z_1^2+z_2^2+z_3^2+t z_3=0\}. 
\]
We have $\mathbf{T}_C^1=\mathbb{C}[z_1,z_2,z_3]/\langle z_1, z_2, z_3\rangle$. So $\mathcal{G}=(g=z_3)$ gives vanishing image $[\mathcal{G}]=0$. However, we have:
\[
\{z_1^2+z_2^2+z_3^2+t z_3=0\}=\{z_1^2+z_2^2+(z_3+t/2)^2-\frac{t^2}{4}=0\}\cong \{z_1^2+z_2^2+\tilde{z}_3^2-\frac{t^2}{4}=0\}.
\]
So by Definition \ref{def-ord2} and \eqref{eq-KSpsi}, we see that the order of the deformation is equal to $2$ and the weight of the deformation is equal to $-2$. 
\end{rem}
\end{exmp}
Finally we briefly recall Pinkham's results on deformation of isolated singularities with $\mathbb{C}^*$ actions. We state the result in our setting of affine cones. 
\begin{thm}[\cite{Pink, Pink2}]
\begin{enumerate}
\item
There exists a formal versal $\mathbb{C}^*$ equivariant deformation $\mathscr{C}\rightarrow V$ of $C$.
\item
Let $\mathscr{Y}\rightarrow T$ be any formal $\mathbb{C}^*$ equivariant deformation of $X$. Then there exists a $\mathbb{C}^*$ equivariant morphism $\phi: T\rightarrow V$ and a 
$\mathbb{C}^*$ equivariant isomorphism of the deformation $\mathscr{Y}\rightarrow T$ with the pull back $\mathscr{X}\times_{V}T\rightarrow T$.
\end{enumerate}
\end{thm}
Let $t_j$ be homogeneous generators of the maximal ideal of weigh $d(t_j)$. Let $J^{-}$ be the ideal in $\mathcal{O}_V$ generated by $\{t_j; d(t_j)<0\}$. Let $V^{-}$ be the subvariety defined 
by $J^{-}$.
\begin{thm}[{\cite[Theorem 2.9]{Pink2}}]\label{Pink2}
$\mathscr{C}^{-}\rightarrow V^{-}$ extends to a proper flat family $\overline{\mathscr{C}}^-\rightarrow V^-$ of deformations of $\bar{C}$. $\overline{\mathscr{C}}-\mathscr{C}\cong D_\infty\times V^{-}$ and $\overline{\mathscr{C}}^{-}\rightarrow V^{-}$ is a locally trivial deformation near $D_{\infty}$.
\end{thm}

\vskip 2mm
\noindent
Mathematics Department, Stony Brook University, Stony Brook NY, 11794-3651, USA \\
%Email: chi.li@stonybrook.edu

\noindent
{\it Curent address: }\\
Department of Mathematics, Purdue University, West Lafayette, IN, 47907-2067, USA \\
Email: li2285@purdue.edu

\end{document}